\documentclass[reqno,11pt]{preprint}
\usepackage[full]{textcomp}
\usepackage[osf]{newtxtext}
\usepackage{cabin}
\usepackage{enumerate}
\usepackage{enumitem}
\usepackage{bbm}
\usepackage{dsfont}

\usepackage[varqu,varl]{inconsolata}
\usepackage[cal=boondoxo]{mathalfa}

\usepackage{comment}
\usepackage{hyperref}
\usepackage{breakurl}
\usepackage{mathrsfs}
\usepackage{booktabs}
\usepackage{caption}

\usepackage{tikz}
\usetikzlibrary{decorations.pathreplacing,decorations.markings}
\usepackage{amsmath} 
\usepackage{mhequ} 
\usepackage{mhsymb} 
\usepackage{mhenvs} 
\usepackage{microtype}
\usepackage{amssymb} 
\usepackage{eufrak}

\usepackage{wasysym}
\usepackage{centernot}

\usepackage{oldgerm}

\newcommand{\Exp}{\mathbf{E}}

\newcommand{\wN}{w^{\eps,n}}
\newcommand{\vN}{v^{\eps,n}}
\newcommand{\tvN}{\tilde v^{\eps,n}}

\newcommand{\WN}{\mathcal{W}^{\eps,n}}
\newcommand{\Weps}{\mathcal{W}^{\eps}}

\renewcommand{\R}{\mathbb{R}}

\renewcommand{\N}{\mathbb{N}}

\newcommand{\T}{\mathbb{T}}
\renewcommand{\Z}{\mathbb{Z}}
\renewcommand{\E}{\mathbb{E}}
\renewcommand{\P}{\mathbb{P}}

\newcommand{\J}{\mathbb{J}}

\newcommand{\BM}{\mathbf{M}}
\newcommand{\BP}{\mathbf{P}}
\newcommand{\bj}{\mathbf{j}}

\newcommand{\bz}{\mathbf{z}}

\newcommand{\cA}{\mathcal{A}}

\newcommand{\cD}{\mathcal{D}}
\newcommand{\cF}{\mathcal{F}}
\newcommand{\cG}{\mathcal{G}}
\newcommand{\cH}{\mathcal{H}}

\newcommand{\cL}{\mathcal{L}}
\newcommand{\cM}{\mathcal{M}}
\newcommand{\cN}{\mathcal{N}}

\newcommand{\cS}{\mathcal{S}}

\newcommand{\cW}{\mathcal{W}}

\newcommand{\dd}{\mathrm{d}}      

\newcommand{\eff}{\mathrm{eff}}

\newcommand{\SH}{\mathscr{H}}

\newcommand{\ST}{\mathscr{T}}

\newcommand{\fw}{\mathfrak{w}}

\newcommand{\sym}{\mathrm{sym}}

\def\one{\mathrm{(I)}}

\newcommand{\1}{\mathds{1}}

\def\restr{\mathord{\upharpoonright}}

\newcommand{\vphi}{\varphi}
\def\sym{{\mathrm{sym}}}
\newcommand{\Di}{\mathrm{diag}}
\newcommand{\oD}{\mathrm{off}}
\newcommand{\oDi}{\mathrm{off_1}}
\newcommand{\ooDi}{\mathrm{off_2}}

\newcommand{\SHE}{\mathrm{SHE}}

\def\one{\mathrm{(I)}}

\def\indN#1{{	\J^\eps_{#1}	}}

\newcommand{\fock}{\Gamma L^2}	

\def\sint{I}					
\def\wc{\SH}					

\def\swn{{	\eta	}}

\newcommand{\gen}{\cL^\eps}

\newcommand{\genn}{\cL^\eps_{i,n}}

\newcommand{\geneff}{\cL^\eff}

\newcommand{\gensy}{\cL_0}

\newcommand{\gena}{\cA^{\eps}}

\newcommand{\genap}{\cA^{\eps}_+}
\newcommand{\genam}{\cA^{\eps}_-}
\newcommand{\genan}{\cA^{\eps}_{i,n}}
\newcommand{\genapn}{\cA_{i,n}^{\eps,+}}

\newcommand{\cGN}{\cG^\eps}

\newcommand{\NN}{\cN^\eps}

\newcommand{\Ll}{\mathrm{L}^\eps}

\colorlet{darkblue}{blue!90!black}
\colorlet{darkred}{red!90!black}
\colorlet{darkgreen}{green!50!black}
\colorlet{darkyellow}{yellow!90!black}

\title{Gaussian Fluctuations for the stochastic Burgers equation in dimension $d\geq 2$}

\begin{document}

\maketitle

\vspace{-2cm}

\noindent{ \bf Giuseppe Cannizzaro$^1$,  Massimiliano Gubinelli$^2$, Fabio Toninelli$^3$}
\newline

\noindent{\small $^1$University of Warwick, UK, %
    $^2$University of Oxford, UK, %
    $^3$Technical University of Vienna, Austria\\}

\noindent{\small\email{giuseppe.cannizzaro@warwick.ac.uk, 
gubinelli@maths.ox.ac.uk, \\
fabio.toninelli@tuwien.ac.at}}
\newline

\begin{abstract}
The goal of the present paper is to 
establish a framework which allows to rigorously determine the large-scale Gaussian  
fluctuations for a class of singular SPDEs at and above criticality, and therefore beyond 
the range of applicability of pathwise techniques, such as the theory of Regularity Structures. 
To this purpose, we focus on a $d$-dimensional generalization of the Stochastic Burgers equation (SBE) 
introduced in [H. van Beijeren, R. Kutner and H. Spohn, Excess noise for driven diffusive systems, PRL, 1985]. 
In both the critical $d=2$ and super-critical $d\geq 3$ cases, we show that the scaling limit of 
(the regularised) SBE is given by a stochastic heat equation with non-trivially renormalised coefficient, 
introducing a set of tools that we expect to be applicable more widely. 
For $d\ge3$  the scaling adopted is the classical diffusive one, while in $d=2$ it is the 
so-called {\it weak coupling} scaling which corresponds to tuning down the strength of the interaction 
in a scale-dependent way. 
%
%
%
%
\end{abstract}

\bigskip

\setcounter{tocdepth}{2} 
\tableofcontents

\section{Introduction}

The study of (singular) Stochastic Partial Differential Equations (SPDEs) has known tremendous advances 
in recent years. Thanks to the groundbreaking theory of regularity structures~\cite{Hai}, 
paracontrolled calculus~\cite{Para} or renormalization group techniques~\cite{Kup}, a local 
solution theory can be shown to hold for a large family of SPDEs in the so-called {\it sub-critical} 
dimensions (see~\cite{Hai} for the notion of {\it sub-criticality} in this context). 
In contrast, SPDEs in the {\it critical} and {\it super-critical} dimensions
are, from a mathematical viewpoint, still 
largely unexplored and poorly understood. 
The aim of the present paper is to 
establish a framework which allows to rigorously determine the large-scale (Gaussian) 
fluctuations for a class of SPDEs at and above criticality under (suitable) diffusive scaling. 

To present the ideas and the techniques in our analysis, 
we will focus on a specific example, the Stochastic Burgers Equation (SBE), which, in any dimension, 
can (formally) be written as 
\begin{equ}[e:BKS]
    \partial_t\eta = \tfrac{1}{2}\Delta\eta+\fw\cdot\nabla\eta^2 + \div \boldsymbol{\xi}
 \end{equ}
where the scalar field $\eta=\eta(x,t)$ depends on $t\ge 0$ and $x\in \R^d$ with $d\ge1$, 
$\fw\in \R^d$ is a fixed vector and $\boldsymbol{\xi}=(\xi_1,\dots,\xi_d)$ is a $d$-dimensional space-time white noise on 
$\R_+\times \R^d$, 
i.e. a centred $d$-dimensional Gaussian process whose covariance is formally given by 
$\Exp[\xi_i(t,x)\xi_j(s,y)]=\delta(t-s)\delta(x-y)\1_{i=j}$. 

The above equation was introduced by van Beijeren, Kutner and Spohn in
\cite{BKS85} as an approximation of the fluctuations of general driven diffusion systems with one conserved quantity, 
e.g. the Asymmetric Simple Exclusion Process (ASEP) on $\Z^d$. 
The nonlinear term in~\eqref{e:BKS} expresses the fact that the velocity at which the fluctuation density 
travels depends on its magnitude. In the context of ASEP, 
comes from the gradient of the current and is due to its asymmetry, so that in particular the vector $\fw$  
models how the ASEP under consideration is asymmetric in the different spatial directions. 
In the {\it sub-critical} dimension $d=1$,~\eqref{e:BKS} reduces to the (derivative of the)  
celebrated Kardar--Parisi--Zhang (KPZ) equation~\cite{Kardar} whose local solution theory is by now 
well-understood~\cite{KPZ,KPZreloaded,GonJara,GPuni} and whose large-scale statistics 
have been recently determined~\cite{QSkpz,Vir}. 
In higher dimensions the only available results are in the 
discrete setting. For $d\geq 3$, corresponding to the super-critical case, 
ASEP has been shown to be diffusive at large scales~\cite{EMY,LY} -- 
its fluctuations are Gaussian 
and given by the solution of a Stochastic Heat equation (SHE) with additive noise,
see~\cite{CLO,LOV} and~\cite{Komorowski2012} for a review. 
The critical dimension $d=2$ is more subtle. Only qualitative information is available and 
no scaling limit has been obtained thus far. The most important result is that of~\cite{Yau}, 
which states that two-dimensional ASEP is logarithmically superdiffusive, 
meaning that the diffusivity $D(t)$ satisfies $D(t)\sim (\log t)^{2/3}$ for $t$ large.\footnote{The diffusivity measures 
how the correlations of a process grow in space as a function of time so that, denoting by $\ell(t)$ 
the correlation length, it formally satisfies $\ell(t)\sim \sqrt{ t D(t)}$. }

As mentioned above, our focus is on the critical and super-critical dimensions, 
i.e. we will be studying~\eqref{e:BKS} (in the continuum) in $d\geq 2$. 
In both cases, we will prove that the large-scale behaviour is Gaussian and that the fluctuations 
evolve as a SHE with {\it renormalized coefficients} (see Theorem~\ref{thm:main}), 
but while for $d\ge3$ the scaling 
adopted is the classical diffusive scaling, for $d=2$ we will consider the so-called {\it weak coupling} 
scaling. This corresponds to tuning down the strength of the nonlinearity (morally, the ``norm'' of the vector $\fw$) 
at a suitable rate together with the scale parameter, see~\eqref{eq:lambdaeps}. 
The tuning is chosen so that we can ``tame'' the 
growth caused by the non-linear term without destroying its large-scale effects  
(see e.g.~\cite{CD,CSZ,Gu2020,CES,CETWeak} for a similar choice of scaling in the critical dimension for different equations). 
\medskip

Let us mention that the study of large-scale Gaussian fluctuations of space-time random fields 
with local non-linear random dynamics is  not restricted the context of SPDEs but
is a classical topic in probability theory and mathematical physics.
From the more general point of view of theoretical physics, large-scale Gaussian fluctuations are expected 
for equilibrium diffusive systems above their critical dimension 
and this belief is informally explained via (non-rigorous) Renormalization Group (RG) arguments 
and effective dynamical field theories. 

A variety of tools have been used to mathematically justify these heuristics in the {\it static} setting, i.e. 
for the equilibrium measure of local non-linear systems. RG methods have been applied 
to derive the scaling limit of unbounded spin systems (and other statistical mechanical models)  at and above their critical dimension $d\ge 4$  in~\cite{bauerschmidtRenormalisationGroupMethod2015a, bauerschmidtScalingLimitsCritical2014, bauerschmidtIntroductionRenormalisationGroup2019}\footnote{the authors further claim 
that ``Our emphasis here is on the critical dimension $d = 4$, which is more difficult than dimensions $d > 4$."}. 
The related problem of triviality of $\Phi^4_d$ theories for $d > 4$ was treated in~\cite{Aiz82, Fro82, fernandezRandomWalksCritical1992} using apriori bounds on correlation functions, an approach 
which ultimately lead to a proof in the more challenging critical dimension $d=4$, 
see~\cite{aizenmanMarginalTrivialityScaling2021a}. 
Another class of equilibrium models for which scaling limits have been successfully determined 
are equilibrium interface models, so called $\nabla \phi$ models in all dimensions~\cite{brydgesGradPhiPerturbations1990, naddafHomogenizationScalingLimit1997, Funaki19971,  giacominEquilibriumFluctuationsNabla2001a}, for which the Helffer-Sj\"ostrand representation of correlations together 
with homogeneization arguments was exploited. Analogous results have been proven, with very different methods, for \emph{discrete} $\nabla\phi$ models in dimension $d=2$ (in particular dimer models, both integrable ones \cite{Kenyon} and their non-integrable perturbations \cite{Giul}).

On the other hand, establishing large-scale Gaussian fluctuations 
for {\it dynamical} problems has proven to be challenging  
and we are not aware of general results in this direction.
The major issue seems to stem from the fact that there is currently no framework 
which allows to naturally identify the law of the limit process. 
A lot of attention has been recently devoted to the above mentioned KPZ equation in high dimension, 
whose expression in its most general form is as~\eqref{e:BKS} but 
with nonlinearity given by $\langle \nabla\eta, Q\nabla\eta\rangle$, for $Q$ 
a $d\times d$ deterministic symmetric matrix, 
and a real valued space-time white noise in place of $\div\boldsymbol{\xi}$. 
Most of the works focus on the case in which $Q$ is given by the identity matrix  
and the specific form of the resulting equation plays a crucial role. Indeed, one can use the Cole--Hopf 
transformation to linearise the SPDE and turn it into a {\it linear} multiplicative 
stochastic heat equation or the associated directed-polymer models. 
These latter models possess an ``explicit" representation (via, e.g., Feynman-Kac formula) 
and one can then leverage Malliavin calculus tools to derive the scaling limit 
(see~\cite{dunlapRandomHeatEquation2021, guEdwardsWilkinsonLimit2018} for 
the multiplicative stochastic heat equation and~\cite{Gu2018b, CCM1, LZ, CNN} for KPZ). 
In this context, let us  mention the recent striking results on the fluctuations of KPZ at the critical dimension $d=2$ 
in the weak coupling scaling (or the intermediate disorder regime) starting with \cite{CD,CSZ,Gu2020} (for KPZ)
and culminating with the characterisation of the stochastic heat flow~\cite{caravennaCritical2dStochastic2023}.
An alternative approach is that used in~\cite{magnenScalingLimitKPZ2018}, 
which is based on re-expressing the SPDE as a functional integral via the Martin--Siggia--Rose formalism 
(essentially a Girsanov transformation, which again is only possible in view of the specific form of the equation) 
and then leveraging constructive quantum field theory techniques and RG ideas to control 
the large scale fluctuations in the regime of small non-linearity. 

In the critical dimension $d=2$, for $Q$ diagonal and with diagonal entries $1$ and $-1$ (so-called 
Anisotropic KPZ or AKPZ), progress has been made in~\cite{CET}, 
where it is shown that the equation is logarithmically superdiffusive, and in~\cite{CETWeak}, 
where the large-scale behaviour in the weak coupling scaling is determined. In both articles, 
the tools adopted are closer to those in the present paper and we will comment more on them below. 

Let us also mention recent work in extending the homogeneization methods 
(originated in the work of Naddaf and Spencer~\cite{naddafHomogenizationScalingLimit1997}) to dynamical problems~\cite{armstrongQuantitativeHydrodynamicLimits2022, cardaliaguetScalingLimitsStochastic2022}.
\medskip

One of the main advantages of the approach taken in the present paper is that 
it moves away from the usual technique of Cole--Hopf transformation, which is {\it not applicable}, 
towards a set of tools which on the one hand makes a link with interacting particle systems, 
i.e. via the martingale problem and the associated infinite-dimensional generator, and on the other 
has better chances to apply more generally. A crucial tool for us is the fact that we have control 
over the invariant measure for~\eqref{e:BKS}, which significantly simplifies the already involved 
analysis of the generator. While this is another non-generic situation, there is a wide class of examples 
in which a similar control is available and to which our tools potentially directly apply, 
e.g. (super-)critical Stochastic Navier--Stokes~\cite{GT,CK,HZZ},  
diffusions in divergence-free random vector fields~\cite{CHT}, lattice gas models~\cite{LRY} and many others. 

Our results open the way to gradually attack more challenging models and we discuss some open problems below in Section~\ref{sec:openproblems}.

Before turning to the core of the proof, in the next section we will precisely state our main result 
and outline the ideas behind our arguments.

\subsection{The equation and main result}

As written~\eqref{e:BKS} is meaningless as the noise is too irregular for the non-linearity to be well-defined. 
Nonetheless, as we are interested in the large-scale behaviour, 
we regularise it, so to have a well-defined field at the microscopic level, and 
our goal then becomes to control its fluctuations while zooming out at the correct scale. 
We choose the regularisation in such a way to retain a fundamental property of the solution, 
namely its (formal) invariant measure. 
This amounts to smoothening the quadratic term via a Fourier cut-off as follows
\begin{equ}[eq:toroN]
    \partial_t\eta =\frac{1}{2}\Delta\eta+ \fw\cdot\Pi_1\nabla(\Pi_1\eta)^2  + (-\Delta)^{\frac{1}{2}}\xi,
\end{equ}
where for $a>0$, $\Pi_a$ acts in Fourier space as
\begin{equ}[eq:FourPi]
\widehat{\Pi_a\eta}(k)=\widehat{\eta}(k)1_{|k|\leq a}\,.
\end{equ}
Without loss of generality, 
we also replaced $\div\boldsymbol{\xi}$ in~\eqref{e:BKS} with $(-\Delta)^{\frac{1}{2}}\xi$ for $\xi$ a space-time white noise 
(see~\eqref{e:fLapla} for the definition of the fractional Laplacian), since the two can be easily seen to have the same law. 

\begin{remark}\label{rem:pignolerie}
In principle, the results below can be shown to hold using the same techniques of the present paper, 
even if instead of the Fourier cut-off we used a more general mollifier. That said, 
to simplify some (minor but annoying) technical points, we will stick to $\Pi_1$ and further assume that 
  \begin{itemize}[noitemsep]
  \item for $d\ge3$, $1/\eps\in \N+1/2$ and that the norm in \eqref{eq:FourPi} is the sup-norm $|\cdot|_\infty$, 
  \item for $d=2$, the norm  in \eqref{eq:FourPi}  is the Euclidean norm $|\cdot|$ (no further assumption on $\eps$ is made).
  \end{itemize}
We refer the reader to Remark~\ref{rem:G} for some additional (technical) advantages of our choice. 
\end{remark}

In order to avoid technicalities related to infinite volume issues, 
we restrict~\eqref{eq:toroN} to 
the $d$-dimensional torus $\mathbb T^d_\eps$ of side-length $2\pi /\eps$. Here, $\eps$ is 
the ``scale'' parameter which will later be sent to $0$. 
Note that via the diffusive rescaling
\begin{equ}[e:scaling]
    \eta^\eps(x,t)\eqdef \eps^{-\frac d2}\eta\left(x/\eps,t/\eps^{2}\right)\,, 
\end{equ}
equation~\eqref{eq:toroN} becomes
\begin{equ}\label{e:BurgersScaled}
    \partial_t\eta^\eps = \frac{1}{2}\Delta\eta^\eps +\lambda_\eps \fw\cdot\Pi_{1/\eps}\nabla(\Pi_{1/\eps}\eta^\eps)^2  + (-\Delta)^{\frac{1}{2}}\xi
\end{equ}
where now the spatial variable takes values in the $d$-dimensional torus $\T^d\eqdef \T^d_1$ of side-length $2\pi$, 
and $\xi$ is a space-time white noise on $\R^+\times \mathbb T^d$. 
The coefficient $\lambda_\eps$ depends on the dimension 
and is defined as
\begin{equ}[eq:lambdaeps]
  \lambda_\eps\eqdef
  \begin{cases}
   \frac1{\sqrt{\log \eps^{-2}}} &\text{if $d=2$}\\
    \eps^{\frac{d}{2}-1} & \text{if $d\ge3$}.
  \end{cases}
\end{equ}
While in the super-critical case $d\geq 3$, $\lambda_\eps$ is directly determined by~\eqref{e:scaling}, 
for $d=2$,~\eqref{e:BKS} is formally scale invariant under the diffusive scaling  
(another reason why such dimension is critical) 
so that one would have $\lambda_\eps=1$. The choice made in~\eqref{eq:lambdaeps} 
corresponds to the so-called {\it weak coupling} scaling alluded to earlier and we will comment more on it 
after we state the main result. Before doing so, let us give the notion of solution for~\eqref{e:BurgersScaled} 
we will be working with. 

\begin{definition}\label{def:StatSol}
We say that $\eta^\eps$ is a stationary solution to~\eqref{e:BurgersScaled} with coupling constant $\lambda_\eps$ if 
$\eta^\eps$ solves~\eqref{e:BurgersScaled} and its initial condition is $\eta(0,\cdot)\eqdef\mu$, for $\mu$ 
a spatial white noise on $\T^d$, 
i.e. a mean zero Gaussian distribution on $\T^d$ such that 
\begin{equ}[e:CovS]
\E[\mu(\phi)\mu(\psi)]=\langle\phi,\psi\rangle_{L^2}\,,\qquad \text{for all $\phi,\,\psi\in L^2_0(\T^d)$}
\end{equ}
where $\langle\cdot,\cdot\rangle_{L^2}$ is the usual scalar product in $L^2(\T^d)$ and 
$L^2_0(\T^d)$ is the space of zero-average square integrable functions on $\T^d$. 
We will denote by $\P$ the law of $\mu$ and by $\E$ the corresponding expectation.
\end{definition}

The reason why the solution $\eta^\eps$ of~\eqref{e:BurgersScaled} started from a spatial white noise 
is called {\it stationary} is that $\P$ is its invariant measure 
and, as we will prove in Lemma \ref{lem:generalproperties} below,  this holds irrespective of the choice of $\fw,\eps$. 
Instead, the law $\mathbf P$ (with corresponding expectation $\mathbf E$) of a stationary 
solution to~\eqref{e:BurgersScaled} clearly depends on $\fw,\eps$. 

For $T>0$, let us denote by $C([0,T],\cS'(\T^d))$ the space of continuous 
functions with values in the space of tempered distributions $\cS'(\T^d)$ and by $\cF$ the Fourier transform 
on $\cS'(\T^d)$ (see~\eqref{e:FT} for a definition). 
We are now ready to state the main result of the paper. 

\begin{theorem}\label{thm:main}
Let $\fw\in\R^d\setminus\{0\}$ and $T>0$. 
For $\eps>0$, let $\lambda_\eps$ be defined according to \eqref{eq:lambdaeps} and 
$\eta^\eps$ be the stationary solution of \eqref{e:BurgersScaled} started from a spatial white noise $\mu$. 
Then, there exists a constant $D_\SHE>0$, depending only on the dimension $d$ and the norm $|\fw|$ of $\fw$, 
such that, as $\eps\to0$, $\eta^\eps$ converges in law in $C([0,T],\cS'(\T^d))$ to the unique stationary solution $\eta$ of
\begin{equ}[e:SHEintro]
\partial_t \eta =\tfrac12(\Delta + D_\SHE(\fw\cdot\nabla)^2)\eta + (-\Delta - D_\SHE(\fw\cdot\nabla)^2)^{1/2}\xi\,,\qquad \eta(0,\cdot)=\mu\,,
\end{equ}
where $(\fw\cdot\nabla)^2$ is the linear operator defined as $\cF((\fw\cdot\nabla)^2\phi)(k)\eqdef -(\fw\cdot k)^2\cF(\phi)(k)$, for $\phi\in \cS'(\T^d)$ and $k\in\Z^d$. 

In the case $d=2$, $D_\SHE$ is explicit and given by the formula
\begin{equ}
  \label{eq:DSHEd2}
  D_\SHE= \frac1{|\fw|^2}\left[\left(\frac{3|\fw|^2}{2\pi} +1\right)^{\frac{2}{3}}-1\right],
\end{equ}
$|\fw|$ being the Euclidean norm of $\fw$.
\end{theorem}

Before proceeding to the proof of the above theorem, let us make a few comments on 
its statement and on the choice of $\lambda_\eps$ in~\eqref{eq:lambdaeps}. 
Notice that in any dimension $d\geq 2$, the constant $D_\SHE$ appearing in the limiting 
equation is {\it strictly} positive. Hence, despite the presence of the vanishing factor 
$\lambda_\eps$ in front of the nonlinearity, not only the quadratic term does not go to $0$ 
but actually produces a new ``Laplacian'' and a new noise. Further, this new Laplacian (and noise), in a sense,  
``feels'' the small-scale behaviour of the field $\eta$ in~\eqref{eq:toroN} 
as it depends on the vector $\fw$, which in turn describes 
its microscopic dynamics. A similar phenomenon has already been observed for other 
critical and super critical SPDEs, even though we are not aware of examples 
in which a Laplacian of the form above was derived. 

We have already pointed out that in $d=2$~\eqref{e:BKS} is formally scale invariant. 
That said, the above mentioned result in~\cite{Yau} 
for ASEP suggests that also~\eqref{eq:toroN} is logarithmically superdiffusive and the additional 
diffusivity can only come from the non-linear term. What the previous theorem 
shows is that the choice in~\eqref{eq:lambdaeps} guarantees that 
the nonlinearity ultimately gives a non-vanishing order 1 contribution. 
What might appear puzzling is that, by translating the result of Yau to~\eqref{e:BurgersScaled}, 
the diffusivity $D^\eps(t)$ of the scaled process grows as $|\log \eps|^{2/3}$ so that one might be led 
to think that $\lambda_\eps$ should be chosen accordingly. While this is not 
the case, the expression in~\eqref{eq:DSHEd2} formally implies the result of Yau, as can 
be seen by taking $\fw=\fw_\eps$ such that $|\fw_\eps|\eqdef \lambda_\eps^{-1}$. 

Furthermore, a similar scaling has been considered 
in the critical dimension $d=2$ in the context of the KPZ 
equation~\cite{CD,CSZ,Gu2020, caravennaCritical2dStochastic2023}, 
of the Anisotropic KPZ equation~\cite{CES,CETWeak} and of the stochastic Navier--Stokes equation~\cite{CK}. 
What is interesting is that, even though these latter examples in principle have 
different large scale diffusivity, i.e., $D(t)\sim t^\beta$ for some $\beta>0$~\cite{Forrest} for KPZ, 
$D(t)\sim \sqrt{\log t}$ for AKPZ 
and Navier--Stokes ~\cite{CET,Adler}, they all display {\it non-trivial} 
(in that the non-linearity nontrivially contributes to the limit) {\it Gaussian
fluctuations} under the same weak coupling scaling $\lambda_\eps=|\log\eps|^{-1/2}$, thus suggesting some sort of universality for it. 
This is to be compared with the one-dimensional case, in which, setting $\lambda_\eps\eqdef\sqrt\eps$, 
one recovers the so-called {\it weakly asymmetric} scaling under which convergence of 
ASEP to the one-dimensional KPZ equation was first shown in~\cite{BG} and since then for a wide 
variety of models. 
\subsection{Open problems}\label{sec:openproblems}

Our results raise several interesting questions, as
\begin{itemize}[noitemsep]
\item obtain Gaussian fluctuations for~\eqref{e:BKS} in the case of more general non-linearity, 
i.e. $\fw\cdot\nabla \eta^2$ replaced by $\fw\cdot\nabla F(\eta)$, for instance for polynomial $F$. 
The main difficulty is to obtain operator estimates similar to those of Lemma \ref{l:GenaBound}, 
which though seem highly non-trivial.
\item Prove the analog of Theorem \ref{thm:main} (for $d=2$) for ASEP on $\Z^2$ in the same limit of weak asymmetry.
\item In dimension $d=2$, define $\tilde\eta^\eps(x,t)=\eps^{-1}\eta( x/\eps, t/(\eps^2\tau(\eps)))$ 
(to be compared with \eqref{e:scaling}) and find the right correction $\tau(\eps)$ to the diffusive scaling 
so that a non-trivial scaling limit as $\eps\to0$ exists. This corresponds to a \emph{strong coupling regime}. 
On the basis of \cite{Yau}, the natural guess is $\tau(\eps)=|\log \eps|^{2/3}$ but the identification of the limit 
is hard as the regularising properties of the Laplacian vanish. 
\end{itemize}

More challenging will be to move towards models for which the invariant measure is non-Gaussian or non-explicit. 
One guiding heuristic could be that the limiting Gaussian fluctuations and 
associated Gaussian stationary distribution could still provide a good  setting for the analysis.

\subsection{Idea of the proof }\label{sec:ideas}

The idea of the proof finds its roots in the approach detailed in~\cite{Komorowski2012} in the context of 
interacting particle systems. To summarise it and see how it translates to the present context, 
let us consider the weak formulation of~\eqref{e:BurgersScaled} started from a spatial white noise $\mu$,  
which, for a given test function $\phi$, reads 
\begin{equ}[e:WeakIntro]
\eta^\eps_t(\phi)-\mu(\phi)-\tfrac12\int_0^t \eta^\eps_s(\Delta\phi)\dd s-\int_0^t\cN^\eps_\phi(\eta^\eps_s)\dd s=M_t(\phi)
\end{equ}
where $\cN^\eps_\phi$ stands for the nonlinearity tested against $\phi$ and $M$ is the martingale 
associated to the space-time white noise $\xi$. Now, upon assuming the sequence $\{\eta^\eps\}_\eps$ 
to be tight in the space $C([0,T]\cS'(\T^d))$, we see that all the terms in the above expression 
converge (at least along subsequences) but we have no information concerning the nonlinearity. 
For Theorem~\ref{thm:main} to be true, we need this latter term to produce both a {\it dissipation} term 
of the form of $D_\SHE\gensy^\fw\eta$, where $\gensy^\fw\eqdef\frac12(\fw\cdot\nabla)^2$, and a {\it fluctuation} 
part which should encode the additional noise in~\eqref{e:SHEintro}. 
In other words, the problem is to identify a {\it fluctuation-dissipation} relation~\cite{LY}, i.e. to determine $D_\SHE>0$ 
and $\cW^\eps$ such that 
\begin{equ}[e:FD]
\cN^\eps_\phi(\eta^\eps)=-\gen \cW^\eps(\eta^\eps) + D_\SHE\gensy^\fw\eta^\eps(\phi)
\end{equ} 
where $\gen$ is the generator of $\eta^\eps$. Indeed, since by Dynkin's formula, there exists a martingale $\cM^\eps(\Weps)$ 
such that
\begin{equ}
\cW^\eps(\eta^\eps_t)-\cW^\eps(\mu)-\int_0^t \gen\cW^\eps(\eta^\eps_s)\dd s=\cM_t^\eps(\Weps), 
\end{equ}
given~\eqref{e:FD}, we could rewrite the nonlinear term in~\eqref{e:WeakIntro} as
\begin{equ}[e:NonlinExp]
\int_0^t\cN^\eps_\phi(\eta^\eps_s)\dd s=\int_0^t D_\SHE\gensy^\fw\eta_s^\eps(\phi)\dd s -\cM_t^\eps(\Weps)+o(1)\,,
\end{equ}
where $o(1)$ is a vanishing summand which contains the boundary terms.
The advantage of the above is that we have expressed the nonlinearity
in terms of a drift which captures the additional diffusivity and a
martingale part which instead econdes the extra noise.  At this point,
tightness would ensure convergence of the former and we would be left
to prove that the sequence $\cM^\eps(\Weps)$ converges to a martingale
with the correct quadratic variation. 
\medskip

The (hardest) problem is clearly the derivation of the fluctuation-dissipation relation and 
this is the point from which our analysis departs from that in~\cite{Komorowski2012}. 
Even though, as we will see, it is enough to determine~\eqref{e:FD} approximately (i.e. that the difference 
of left and right hand side is small in a suitable sense), in the equation there are two unknowns, 
$D_\SHE\gensy^\fw$ which is {\it not} apriori given to us, and 
$\cW^\eps$ which, even if we were given the previous, is the solution of an infinite dimensional equation. 
To separate the two issues, we introduce a suitable truncation which removes the second 
summand from the right hand side of~\eqref{e:FD}, so that we can first solve for $\cW^\eps$ and 
then project back and determine $D_\SHE\gensy^\fw$. 
This is done in $d\geq 3$ and $d=2$ in very different ways. In the former case, we will control 
$\cW^\eps$ similarly to~\cite{LY}, as we can show it satisfies a graded sector condition 
in the spirit of~\cite[Section 2.7.4]{Komorowski2012}. Nonetheless, 
their functional analytic approach does not apply in our setting and in particular, for the identification 
of $D_\SHE$ and $\gensy^\fw$, we devise a new method which is detailed in Section~\ref{sec:Diff3}. 
For $d=2$ instead, we introduce a novel ansatz which is based on the idea that at large scales 
the generator of $\eta^\eps$ should approximate a {\it modulated version} of the generator of~\eqref{e:SHEintro}. 
This ansatz simplifies dramatically the (iterative) analysis performed in~\cite{CETWeak} and 
is heavily based on the Replacement Lemma in Section~\ref{sec:ReplLemma}.

\subsection*{Organisation of the article}

The rest of this work is organised as follows. 
Below we introduce notations, conventions, function spaces and 
elements of Wiener space analysis which will be used throughout. 
Section~\ref{sec:EqandGen} is the bulk of the paper. In Section~\ref{sec:Burgers}, 
after recalling basic properties of~\eqref{e:BurgersScaled} 
and its generator, we discuss the one of the main technical tools, i.e. the 
Replacement Lemma for $d=2$. In Section~\ref{sec:apriori} we determine apriori estimates on the solution to the 
generator equation and in the next two we identify $D_\SHE$ and $\gensy^\fw$, 
so to obtain a refined version of the fluctuation-dissipation relation~\eqref{e:FD}. 
Section~\ref{sec:ProofofMain} is devoted to the proof of Theorem~\ref{thm:main}. 
At last, Appendix~\ref{a:Approx} contains some technical steps necessary in the proof of the 
Replacement Lemma.

\subsection*{Notations, function spaces and Wiener space analysis}

We let $\T^d$ be the $d$-dimensional torus of side length $2\pi $. 
We denote by $\{e_k\}_{k\in\Z^d}$ the Fourier basis defined via 
$e_k(x) \eqdef \frac{1}{(2\pi)^{d/2}} e^{i k \cdot x}$ which, for all $j,\,k\in\Z^d$, satisfies 
$\langle e_k, e_{-j}\rangle_{L^2}= \mathds{1}_{k=j}  $. 

The Fourier transform of a function $\phi\in L^2(\T^d)$ will be denoted by 
$\cF(\phi)$ or $\hat\phi$ and, for $k\in\Z^d$ is given by the formula
	\begin{equation}\label{e:FT}
	\cF(\varphi)(k) =\hat\varphi(k)\eqdef  \int_{\T^d} \varphi(x) e_{-k}(x)\dd x\,, 
	\end{equation}
so that in particular
\begin{equation}\label{e:FourierRep}
\varphi(x) = 
\sum_{k\in\Z^d} \hat\varphi(k) e_k(x)\,,\qquad\text{for all $x\in\T^d$. }
\end{equation}
Let $\cS(\T^d)$ be the space of smooth functions on $\T^d$ and $\cS'(\T^d)$ 
the space of real-valued distributions given by the dual of $\cS(\T^d)$. 
For any $\eta\in\cS'(\T^d)$ and $k\in\Z^d$, 
we will denote its Fourier transform by 
\begin{equation}\label{e:complexPairing}
\hat \eta(k)\eqdef \eta(e_{-k})\,.
\end{equation} 
Note that $\overline{\hat\eta(k)}=\hat\eta(-k)$. Since the zero mode $\hat\eta(0)$ of the solution is automatically zero, 
we will only care about $\hat\eta(k)$ for $k\in\Z^d_0\eqdef\Z^d\setminus\{0\}$. 
Moreover, we recall that the Laplacian $\Delta$ on $\T^d$ has eigenfunctions $\{e_k\}_{k \in \Z^d}$ 
with eigenvalues $\{-|k|^2\,:\,k\in\Z^d\}$, so that the operator $(-\Delta)^{\frac12}$ in \eqref{eq:toroN} is defined
by its action on the basis elements 
	\begin{equation}\label{e:fLapla}
	(-\Delta)^{\frac12} e_k\eqdef |k|e_k\,,\qquad k\in\Z^d\,.
	\end{equation}	
        In particular, $(-\Delta)^{ \frac12}$ is an invertible linear
        bijection on distributions with null $0$-th Fourier mode.
We denote by $H^1(\T^d)$ the space of mean-zero functions $\phi$ such that the norm 
\begin{equ}
\|\phi\|_{H^1}^2\eqdef\|(-\Delta)^{\frac12}\phi\|_{L^2}^2\eqdef\sum_{k\in\Z^d_0} |k|^2 |\hat\phi(k)|^2
\end{equ} 
 is finite.  
\medskip 

Let $(\Omega,\cF,\P)$ be a complete probability space and 
$\eta$ be a mean-zero spatial white noise on the $d$-dimensional torus $\T^d$, i.e. 
$\eta$ is a Gaussian field with covariance
\begin{equ}\label{eq:spatial:white:noise}
\E[\swn(\vphi)\swn(\psi)]=\langle \vphi, \psi\rangle_{L^2}
\end{equ}
where $\varphi,\psi\in H\eqdef L^2_0(\T^d)$, 
the space of square-integrable functions with $0$ total mass, 
and $\langle\cdot,\cdot\rangle_{L^2}$ is the usual scalar product in $L^2(\T^d)$. 
For $n\in\N$, let $\SH_n$ be the {\it $n$-th homogeneous Wiener chaos}, i.e. the closed linear subspace of 
$L^2(\Omega)$ generated by the random variables $H_n(\eta(h))$, 
where $H_n$ is the $n$-th Hermite polynomial, and 
$h\in H$ has norm $1$. By~\cite[Theorem 1.1.1]{Nualart2006}, $\wc_n$ and $\wc_m$ are orthogonal whenever 
$m\neq n$ and $L^2(\Omega)=\bigoplus_{n}\SH_n$. 
Moreover, there exists a canonical contraction $\sint{}:\bigoplus_{n\ge 0} L^2(\T^{2n}) \to L^2(\Omega)$, 
which restricts to an isomorphism $\sint{}:\fock \to L^2(\Omega)$ on the Fock space $\fock:=\bigoplus_{n \ge 0} \fock_n$, 
where $\fock_n$ denotes the space $L_\sym^2(\T^{dn})$ of functions in $L^2_0(\T^{dn})$ which are symmetric
with respect to permutation of variables. The restriction $\sint_n$ of $\sint$ to $\fock_n$, 
is itself an isomorphism from $\fock_n$ to $\SH_n$
and, by~\cite[Theorem 1.1.2]{Nualart2006}, for every $F\in L^2(\Omega)$ there exists a family of kernels 
$(f_n)_{n\in\N}\in\fock$ such that $F=\sum_{n\geq 0} I_n(f_n)$ and 
\begin{equ}[e:Isometry]
\E[F^2]\eqdef\|F\|^2 = \sum_{n\geq 0} n!\|f_n\|_{L^2_n}^2
\end{equ}
where $L^2_n\eqdef L^2(\T^{dn})$, and we take the right hand side as the definition of the scalar product on $\fock$, i.e. 
\begin{equ}[e:ScalarPFock]
 \langle f,g\rangle=\sum_{n\geq 0} \langle f_n,g_n\rangle \eqdef \sum_{n\geq 0} n!\langle f_n,g_n\rangle_{L^2_n}.
\end{equ}

\begin{remark}\label{rem:NotAbuse}
In view of the isometry between $\fock$ and $L^2(\Omega)$, throughout the paper, we will abuse notation 
and denote with the same symbol operators acting on $\fock$ and their composition 
with the isometry $I$, which is an operator acting on $L^2(\Omega)$. 
\end{remark}

For $F(\eta)=f(\eta(\phi_1),\dots,\eta(\phi_n))$, $f\colon\R^n\to\R$ smooth and growing at most polynomially at infinity and 
$\phi_1,\dots,\phi_n\in L^2_0(\T^d)$, we define the Malliavin derivative $DF$ according to 
\begin{equ}[e:MalDer]
DF(\eta)(\cdot)\eqdef \sum_{i=1}^n \partial_i f(\eta(\phi_1),\dots,\eta(\phi_n))\phi_i(\cdot)
\end{equ}
and, for $k\in\Z_0^d$, denote its Fourier transform by 
\begin{equation}\label{e:Malliavin}
D_{k} F(v)\eqdef \cF(DF(v))(k)=\langle DF, e_{-k}\rangle\,.
\end{equation}
Notice that, for $p\in\Z$, the action of $D_p$ on $f\in\fock_n$ is 
\begin{equ}[e:MallDerFock]
\cF(D_pf)(k_{1:n-1})=n \hat f_n(k_{1:n-1},p)\,,
\end{equ}
so that in particular $D_p f\in\fock_{n-1}$. 
At last, we also recall the following integration by parts formula on Wiener space, i.e. 
\begin{equation}\label{eq:intbyparts}
\E[G D_kF]=\E[G\langle DF, e_{-k}\rangle]= \E[-F D_kG + FG\hat{\eta}(k)]\,. 
\end{equation}

Throughout the paper, we will write $a\lesssim b$ if there exists a constant $C>0$ such that $a\leq C b$. 
We will adopt the previous notations only in case in which 
the hidden constants do not depend on any quantity which is relevant for the result. 
When we write $\lesssim_T$ for some quantity $T$, it means that the constant $C$ implicit in the bound depends on $T$.

\section{The equation and its generator}\label{sec:EqandGen}

In this section, we first collect a number of preliminary properties of the solution $\eta^\eps$ to 
the regularised Burgers equation~\eqref{e:BurgersScaled} and its generator (Section~\ref{sec:Burgers}). 
Then, we carry out a detailed analysis of the generator equation (a.k.a. resolvent equation) and 
obtain the main estimates we will need in order to determine the large--scale behaviour of $\eta^\eps$.

\subsection{Multidimensional Burgers generator}\label{sec:Burgers}

As in Section~\ref{sec:ideas}, let us first write~\eqref{e:BurgersScaled} in its weak formulation. 
For $\phi\in\cS(\T^d)$ and $t\geq 0$, it reads 
\begin{equ}[e:SPDEweak]
\eta^\eps_t(\phi)-\eta^\eps_0(\phi)=\int_0^t\eta^\eps_s(\Delta\phi)\dd s+\int_0^t\cN^\eps_\phi(\eta^\eps)\dd s +\int_0^t\xi(\dd s, (-\Delta)^{1/2}\phi)
\end{equ}
where $\eta^\eps_0$ is the initial condition, $\xi$ is a space-time white noise so that 
\begin{equ}
\xi(\dd t, (-\Delta)^{1/2}\phi)=\sum_{k\in\Z^d_0} |k|\hat\phi(k)\dd B_t(k)
\end{equ}
for $B(k)$'s complex valued Brownian motions satisfying 
$\overline{B(k)}= B(-k)$ and $\dd\langle B(k), B(\ell) \rangle_t =  \mathds{1}_{\{k+\ell=0\}}\, \dd t$, 
and $\NN_\phi(\eta)$ is the nonlinearity tested against $\phi$, i.e.
\begin{equs}[e:nonlin1]
\NN_\phi(\eta)&\eqdef \lambda_\eps \fw\cdot\Pi_{1/\eps}\nabla(\Pi_{1/\eps}\eta)^2(\phi)\\
&=\frac{\iota}{(2\pi)^{\frac{d}2}}\lambda_\eps\sum_{\ell,m}  \indN{\ell,m}\fw\cdot(\ell+m)\hat\phi(-\ell-m)\hat\eta(\ell)\hat\eta(m)
\end{equs}
with $\iota=\sqrt{-1}$ and 
\begin{equ}[eq:J]
    \indN{\ell, m}= \1\{0<\eps |\ell|_\infty\le 1, 0<\eps|m|_\infty\le 1,0<\eps|\ell+m|_\infty\le 1\}\,.
\end{equ}
In accordance with Remark \ref{rem:pignolerie}, for $d=2$ the sup-norm 
$|\cdot|_\infty$ in~\eqref{eq:J} is replaced by the Euclidean norm $|\cdot|$ instead.

From~\eqref{e:SPDEweak} follows that 
the Fourier modes $\hat \eta^\eps(k)$, $|k|\ge \eps^{-1}$ evolve like
independent (and $\eps$-independent) Ornstein-Uhlenbeck processes.

\begin{lemma}
\label{lem:generalproperties}
For every deterministic initial condition $\eta_0$, the solution $t\mapsto \eta^\eps_t$ 
of \eqref{e:SPDEweak} exists globally in time and is a strong Markov process. 
The generator $\gen$ of $\eta^\eps$ can be written as
$\gen=\gensy+\gena$ and the action of $\gensy$ and $\gena$ on smooth cylinder functions $F$
is given by 
 \begin{align}
 (\gensy F)(\eta) &\eqdef \frac{1}{2} \sum_{k \in \Z^d_0} |k|^2 (-\hat\eta(-k) D_k  +  D_{-k}D_k )F(\mu) \label{e:gens}\\
   (\gena F)(\eta) &\eqdef \frac\iota{(2
    \pi)^{d/2}}\lambda_\eps \sum_{m,\ell \in \Z_0^d} \indN{\ell,m} \fw\cdot (\ell+m) \hat \eta(m) \hat \eta(\ell) D_{-m-\ell} F(\eta). \label{e:gena}
 \end{align}
 Moreover, $\gensy$ is symmetric with respect to $\mathbb P$ while $\gena$ is skew-symmetric.
Finally, the law $\mathbb P$ of the average-zero space white noise is stationary.
\end{lemma}
\begin{proof}
Very similar arguments were provided in a number of references for equations which 
share features similar to those of~\eqref{e:SPDEweak}, e.g. \cite{CES, GubinelliJara2012} and, 
more comprehensively,~\cite{Gubinelli2016}, so 
we will limit ourselves to sketch some of the proofs. 

For the global in time existence of~\eqref{e:SPDEweak} for fixed $\eps$, we refer the reader to \cite[Prop. 3.4]{CES} 
or~\cite[Section 3]{GubinelliJara2012}, while the expressions \eqref{e:gens} and \eqref{e:gena} 
can be easily derived by It\^o's formula and the definition of Malliavin derivative in~\eqref{e:MalDer}.

To see that $\mathbb P$ is stationary, we start by noting that
stationarity for the linear equation with $w=0$ is well-known (and
easy to check), see e.g.~\cite[Section 2.3]{Gubinelli2016}, 
where it is further shown that $\gensy$ is symmetric with respect to $\P$. 
Hence, by~\cite{Ech}, it suffices to check that
$\E [\gena F] = 0$ for every cylinder function $F$. In fact, 
this follows immediately if we prove that $\gena$ is skew-symmetric, 
since, if so, $\E [\gena F]=-\E [F\gena 1]=0$ as the Malliavin derivative of the constant random variable $F=1$ is $0$. 
Now, to prove skew-symmetry for $\gena$, we apply Gaussian integration by parts~\eqref{eq:intbyparts} 
and get 
\begin{equs}
\E [G \gena F]=&-\E[F\gena G]\\
&+ \frac{\iota\lambda_\eps}{(2
 \pi)^{d/2}} \sum_{m,\ell \in \Z_0^d}\indN{\ell,m} \fw\cdot (\ell+m) \E \Big[\hat\eta(-\ell-m)\hat\eta(m)\hat\eta(\ell) F(\eta)G(\eta)\Big]\\
 =& -\mathbb E F \gena G +\E\Big[ \langle \cN^\eps_\phi(\eta),\eta\rangle_{L^2}  F(\eta)G(\eta)\Big]
\end{equs}
so that, since 
\begin{equ}[e:InvMeas]
\langle \cN^\eps_\phi(\eta),\eta\rangle_{L^2}=\lambda_\eps\langle \fw\cdot\nabla(\Pi_{1/\eps}\eta)^2, (\Pi_{1/\eps}\eta)\rangle=\frac{\lambda_\eps}{3}\langle \fw\cdot\nabla(\Pi_{1/\eps}\eta)^3, 1\rangle=0
\end{equ}
the proof is completed. 
\end{proof}

Next, we want to determine the action of $\gensy$ and $\gena$ on the Fock space $\fock$. 
To lighten notations, for a set of integers $I$, we will denote by $k_I$ 
the vector $(k_i)_{i\in I}$ and 
\begin{equ}[e:notations] 
|k_I|^2\eqdef \sum_{i\in I}|k_i|^2\qquad\text{and}\qquad k_{1:n}\eqdef k_{\{1:n\}}\,,
\end{equ}
where  $\{1:n\}\eqdef \{1,\dots,n\}$.

\begin{lemma}\label{lem:essid}
Let $\eps>0$, $\gensy$ and $\gena$ be the operators defined according to~\eqref{e:gens} and~\eqref{e:gena} 
respectively. Then, $\gena$ can be written as
  \begin{equ}[eq:split]
    \gena=\genap+\genam,\qquad\text{with}\qquad \genam=-(\genap)^*\,.
  \end{equ}
For any $n\in\N$, the action of the operators $\gensy,\genam,\genap$ on $f\in\fock_n$ 
is given by
\begin{equs}
\cF(\gensy f) (k_{1:n})=&-\tfrac12|k_{1:n}|^2\hat f(k_{1:n})\\
\cF(\genap f) (k_{1:n+1})=&-\frac{2\iota}{(2\pi)^{d/2}} \frac{\lambda_\eps}{n+1}\times\label{e:gen}\\
&\times\sum_{1\le i<j\le n+1} \,[\fw\cdot (k_i+k_j)]\indN{k_i,k_j}\hat f(k_i+k_j,k_{\{1:n+1\}\setminus\{i,j\}})\\
\cF(\genam  f) (k_{1:n-1})=&\frac{2\iota}{(2\pi)^{d/2}} n\lambda_\eps\,\sum_{j=1}^{n-1}(\fw\cdot k_j) \sum_{\ell+m=k_j}\indN{\ell,m}\hat f(\ell,m,k_{\{1:n-1\}\setminus\{j\}}).
\end{equs}
and, if $n=1$, then $\genam f$ is identically $0$. {Moreover, for
any $i=1,\dots,d$, the momentum operator $M_i$, defined for
$f\in\fock_n$ as
$\cF(M_if)(k_{1:n})=(\sum_{j=1}^n k^i_j)\hat f(k_{1:n})$, with $k^i_j$
the $i^{th}$ component of the vector $k_j$, commutes with
$\gensy, \genap$ and $\genam$. }
\end{lemma}
\begin{proof}
The proof is identical to that of~\cite[Lemma 3.5]{CES}, to which we refer the interested reader. 
{Regarding the fact that the $\gensy$, $\gena_\pm$ commute with the momentum operator, 
this is immediate to verify. }
\end{proof}

In order to state the following lemma, we introduce the so-called {\it number operator}.

\begin{definition}\label{def:NoOp}
We define the {\it number operator} $\cN\colon \fock\to\fock$ to be the linear diagonal operator acting on $\psi=(\psi_n)_n\in\fock$ 
as $(\cN\psi)_n\eqdef n \psi_n$. 
\end{definition}
 
Our analysis will need some preliminary estimates on the asymmetric part of the generator, 
which corresponds to proving the so-called {\it graded sector condition} in~\cite[Section 2.7.4]{Komorowski2012}. 

\begin{lemma}\label{l:GenaBound}
For $d\ge2$ here exist a constant $C=C(d)>0$ such that for every $\psi\in\fock$ the following estimate holds
\begin{equ}\label{e:Abound}
\|(-\gensy)^{-\frac12}\gena_{\sigma} \psi\|\leq C \|\sqrt{\cN}(-\mathcal L_0)^\frac12 \psi\|\,,
\end{equ}
for $\sigma\in\{+,-\}$. 
In particular, for $\sigma\in\{+,-\}$, it  implies 
\begin{equ}\label{e:Tbound}
\|(-\gensy)^{-\frac12}\gena_{\sigma} (-\gensy)^{-\frac12} \psi\|\leq C \|\sqrt{\cN} \psi\|\,.
\end{equ}
\end{lemma}
\begin{proof}
Clearly, once we establish~\eqref{e:Abound},~\eqref{e:Tbound} follows immediately 
by choosing $\psi$ therein to be $(-\gensy)^{-1/2}\rho$ for $\rho\in\fock$.

For~\eqref{e:Abound}, we claim that we only need to prove that for any $\psi\in\fock_n$ and $\rho\in\fock_{n+1}$, we have
\begin{equ}[e:AprioriFirst]
|\langle\rho,\genap \psi\rangle|\lesssim \sqrt{n}\Big(\gamma\|(-\gensy)^{1/2}\psi\|^2+\frac{1}{\gamma} \|(-\gensy)^{1/2}\rho\|^2\Big)\,.
\end{equ}
We will show~\eqref{e:AprioriFirst} at the end. 
Assuming it holds, we first derive~\eqref{e:Abound} for $\genap$. 
The variational characterisation of the $(-\gensy)^{-1/2}\fock$-norm gives 
\begin{equs}
\|&(-\gensy)^{-1/2}\genap\psi\|^2=\sup_{\rho\in \fock_{n+1}}\left(2\langle \rho,\genap  \psi\rangle-
    \| (-\cL_0)^{\frac12}\rho\|^2\right)\\
&\leq \sup_{\rho\in \fock_{n+1}}\left(2C_0\sqrt{n}\Big(\gamma \|(-\gensy)^{1/2}\psi\|^2 + \frac{1}{\gamma}\|(-\gensy)^{1/2}\rho\|^2\Big)-
    \| (-\cL_0)^{\frac12}\rho\|^2\right)\\
    &\lesssim n\|(-\gensy)^{1/2}\psi\|^2=\|\sqrt{\cN} (-\gensy)^{1/2}\psi\|^2
\end{equs}
where in the first step we used~\eqref{e:AprioriFirst} (and $C_0$ is the universal 
constant implicit in that inequality) while in the last we chose $\gamma\eqdef 2C_0\sqrt{n}$.

For $\genam$ instead, we use that, by~\eqref{eq:split}, $\genam=-(\genap)^\ast$. Invoking
again the variational formula above, we deduce
\begin{equs}
\|&(-\gensy)^{-1/2}\genam\psi\|^2=\sup_{\rho\in \fock_{n-1}}\left(2\langle \rho,\genam  \psi\rangle-
    \| (-\cL_0)^{\frac12}\rho\|^2\right)\\
 &=\sup_{\rho\in \fock_{n-1}}\left(-2\langle \genap\rho, \psi\rangle-
    \| (-\cL_0)^{\frac12}\rho\|^2\right)\\
&\leq \sup_{\rho\in \fock_{n-1}}\left(2C_0\sqrt{n}\Big(\gamma \|(-\gensy)^{1/2}\rho\|^2 + \frac{1}{\gamma}\|(-\gensy)^{1/2}\psi\|^2\Big)-
    \| (-\cL_0)^{\frac12}\rho\|^2\right)\\
    &\lesssim n\|(-\gensy)^{1/2}\psi\|^2=\|\sqrt{\cN} (-\gensy)^{1/2}\psi\|^2
\end{equs}
where this time we chose $\gamma\eqdef 1/(2C_0\sqrt n)$.  
\medskip

It remains to prove~\eqref{e:AprioriFirst}. By definition of $\genap$, we have
\begin{equs}
|\langle\rho,\genap \psi\rangle|=&n!\frac{2}{(2\pi)^{\frac{d}2}} \lambda_\eps\Big|\sum_{k_{1:n+1}}\hat\rho(-k_{1:n+1})\times\\
&\times\sum_{1\le i<j\le n+1} [\fw\cdot (k_i+k_j)]\indN{k_i,k_j}\hat \psi(k_i+k_j,k_{\{1:n+1\}\setminus\{i,j\}})\Big|
\\
=&\frac{(n+1)! n}{(2\pi)^{\frac{d}2}}  \lambda_\eps\Big|\sum_{k_{1:n+1}}\hat\rho(-k_{1:n+1})[\fw\cdot(k_1+k_2)]\indN{k_1,k_2} \hat\psi(k_1+k_2, k_{3:n+1})\Big|\,.
\end{equs}
Let us look at the sum over $k_1,k_2$. This equals
\begin{equs}
&\Big|\sum_{k_{1:2}}\hat\rho(-k_{1:n+1})[\fw\cdot (k_1+k_2)]\indN{k_1,k_2} \hat\psi(k_1+k_2, k_{3:n+1})\Big|\\
&=\Big|\sum_{q} (\fw\cdot q)\hat\psi(q, k_{3:n+1})\sum_{k_1+k_2=q}\indN{k_1,k_2} \hat\rho(-k_{1:n+1})\Big|\\
&\lesssim \Big(\sum_q (\fw\cdot q)^2|\hat\psi(q, k_{3:n+1})|^2\Big)^{1/2}\Big(\sum_q \Big|\sum_{k_1+k_2=q}\indN{k_1,k_2} \hat\rho(-k_{1:n+1})\Big|^2\Big)^{1/2}\\
&=:A\times\one\label{e:Adef}
\end{equs}
We leave $A$ as it stands and focus on $\one$. Note that this can be bounded by 
\begin{equs}
\one&=\Big(\sum_q \Big|\sum_{k_1+k_2=q}\indN{k_1,k_2} \hat\rho(-k_{1:n+1})\Big|^2\Big)^{1/2}\\
&\leq \Big(\sum_q \sum_{k_1+k_2=q}|k_{1:2}|^2|\hat\rho(-k_{1:n+1})|^2\sum_{k_1+k_2=q}\frac{\indN{k_1,k_2}}{|k_{1:2}|^2}\Big)^{1/2}\\
&\lesssim \lambda_\eps^{-1} \Big(\sum_q \sum_{k_1+k_2=q}|k_{1:2}|^2|\hat\rho(k_{1:n+1})|^2\Big)^{1/2}=:\lambda_\eps^{-1} B\label{e:Bdef}
\end{equs}
where we used that for $d=2$, $\lambda_\eps^2=(\log\eps^{-2})^{-1}$ and 
\begin{equ}[e:CrucialBound2]
\lambda_\eps^2\sum_{k_1+k_2=q}\frac{\indN{k_1,k_2}}{|k_{1:2}|^2}\leq \frac{\eps^2}{\log\eps^{-2}}\sum_{\eps<|\eps k_1|\leq 1}\frac{1}{|\eps k_{1}|^2}\lesssim  \frac{1}{\log\eps^{-2}}\int_{|x|\in[\eps,1]}\frac{\dd x}{|x|^2}\lesssim 1
\end{equ}
while for $d\geq 3$, $\lambda_\eps^2=\eps^{d-2}$ and 
\begin{equs}[e:CrucialBound3+]
\eps^{d-2}\sum_{k_1+k_2=q}\frac{\indN{k_1,k_2}}{|k_{1:2}|^2}&\leq\eps^{d}\sum_{|\eps k_1|\leq 1}\frac{1}{|\eps k_{1}|^2}\lesssim \int_{|x|\leq 1}\frac{\dd x}{|x|^2}\lesssim 1\,.
\end{equs}
As a consequence, for any $\gamma>0$, we have 
\begin{equs}
|\langle\rho,\genap \psi\rangle|&\lesssim (n+1)! n\sum_{k_{3:n+1}}AB\leq (n+1)! n\Big(\frac{\gamma}{\sqrt n} \sum_{k_{3:n+1}} A^2 +\frac{\sqrt n}{\gamma}\sum_{k_{3:n+1}} B^2\Big)\\
&\leq (n+1)! n\Big(\frac{1}{n!}\frac{\gamma}{n^{3/2}}\|(-\gensy)^{1/2}\psi\|^2 +\frac{1}{(n+1)!}\frac{n^{-1/2}}{\gamma}\|(-\gensy)^{1/2}\rho\|^2\Big)\\
&\leq \sqrt n\gamma\|(-\gensy)^{1/2}\psi\|^2 +\frac{\sqrt n}{\gamma}\|(-\gensy)^{1/2}\rho\|^2\label{e:AlmostDoneBound}
\end{equs}
where we used that, by the definition of $A$ and $B$ in~\eqref{e:Adef} and~\eqref{e:Bdef} respectively, 
we have 
\begin{equ}
\sum_{k_{3:n+1}} A^2\lesssim\sum_{k_{1:n}}  |k_1|^2|\hat\psi(k_{1:n})|^2\leq\frac{n^{-1}}{n!} \|(-\gensy)^{1/2}\psi\|^2
\end{equ}
and similarly for $B$. Then,~\eqref{e:AprioriFirst} follows. 
\end{proof}

Before turning to the analysis of the generator equation, we need some further estimates which are 
specific to the case $d=2$.

\subsubsection{The replacement lemma: $d=2$}\label{sec:ReplLemma}


{While in dimension $d\geq 3$ the inverse Laplacian is integrable close to the origin, 
and this will be crucial in identifying the limiting diffusivity, }
in $d=2$, this is not the case. To overcome the lack of integrability, we devise an alternative route, 
which (morally) approximates the inverse of the generator with a linear 
and  {\it diagonal} (both in Fourier and in the chaos) operator, 
given by a non-trivial order $1$ perturbation of $\gensy$. 

To be more precise, we need to introduce some notation. For $\eps>0$, let $\Ll$ be the function 
defined on $[1/2,\infty)$ as 
\begin{equ}[e:L]
\Ll(x)\eqdef \lambda_\eps^2 \log \left(1+\frac{1}{\eps^2 x}\right)\,.
\end{equ}
Further, let $G$ be the function
\begin{equ}
  \label{eq:defG}
G(x)\eqdef \frac1{|\fw|^2}\left[\left(\frac{3|\fw|^2}{2\pi} x+1\right)^{\frac{2}{3}}-1\right],
\end{equ}
and $\cGN$ be the operator on $\fock$ given by 
\begin{equ}[e:G]
\cGN\eqdef G(\Ll(-\gensy))
\end{equ}
which means that for $\psi\in\fock_n$, $n\in\N$, the action of $\cGN$ on $\psi$ is 
\begin{equ}
\cF(\cGN\psi)(k_{1:n})\eqdef G(\Ll(\tfrac12|k_{1:n}|^2))\hat\psi(k_{1:n})\,,\qquad k_{1:n}\in\Z^{2n}_0\,.
\end{equ}
We are now ready to state the main result of this section, namely, the replacement lemma. 
Recall the definition of the scalar product $\langle\cdot,\cdot\rangle$ on $\fock$ given in~\eqref{e:ScalarPFock}.

\begin{lemma}[Replacement Lemma]\label{l:Replacement}
There exists a constant $C>0$ such that for every $\psi_1,\psi_2\in\fock$ we have 
\begin{equs}\label{e:Replacement}
|\langle \big[-\genam(-\gensy-\gensy^\fw\cGN)^{-1}&\genap+\gensy^\fw\cGN\big]\psi_1,\psi_2\rangle|\\
&\leq C\lambda_\eps^2 \|\cN(-\gensy)^{1/2}\psi_1\|\|\cN(-\gensy)^{1/2}\psi_2\|
\end{equs}
where $\cGN$ is the operator defined according to~\eqref{e:G}  and \eqref{eq:defG} and, for $n\in\N$, 
$\gensy^\fw$ is the operator acting on $\psi\in\fock_n$ as 
\begin{equ}[e:L0w]
\cF(\gensy^\fw\psi)(k_{1:n})=-\tfrac12(\fw\cdot k)^2_{1:n}\hat\psi(k_{1:n})\,,\qquad \text{for all $k_{1:n}\in\Z^{2n}_0$}
\end{equ}
and $(\fw\cdot k)_{1:n}^2\eqdef \sum_{i=1}^n(\fw\cdot k_i)^2$.
\end{lemma}
\begin{remark}\label{rem:G}
As it appears from the proof of Proposition~\ref{p:Approx} in the appendix, 
the form of the function $G$ in \eqref{e:L} is dictated by the
following fixed point equation 
\begin{equ}[e:IntegralG]
G(x)=\frac{1}{\pi}\int_0^x   \frac{\dd y}{\sqrt{1+|\fw|^2 G(y)}}\,.
\end{equ}
Now, if we had chosen a different regularisation for the nonlinearity, 
e.g. the $|\cdot|_\infty$-distance in~\eqref{eq:J} instead of the Euclidean one or a smooth regularisation 
instead of the Fourier cut-off in~\eqref{e:nonlin1}, then we would have not obtained an explicit $G$ 
but the statement above would have remained true. 

Apart from working with an explicit function, the advantage of our choice lies in the fact that 
it allows us to make an explicit connection with the work of Yau~\cite{Yau} on the two-dimensional ASEP. 
Indeed, the exponent 
$2/3$ in~\eqref{eq:defG} precisely reflects that obtained therein and strongly suggests that, if $\lambda_\eps$ were 
taken to be a constant independent of $\eps$, then the diffusivity would diverge as $(\log t)^{2/3}$ for $t\to\infty$. 
\end{remark}
\begin{proof}
Notice first that, for any $n\in\N$, the operator $\genam(-\gensy-\gensy^\fw\cGN)^{-1}\genap$ maps $\fock_n$ 
into itself, so that, to establish~\eqref{e:Replacement}, it suffices to consider $\psi_1,\psi_2\in\fock_n$. Moreover, we can write
\begin{equ}[e:ScalProd]
  \langle \big[-\genam(-\gensy-\gensy^\fw\cGN)^{-1}\genap\big]\psi_1,\psi_2\rangle = \langle (-\gensy-\gensy^\fw\cGN)^{-1}\genap\psi_1,\genap\psi_2\rangle.
\end{equ}
Recalling the form of $\genap$ in~\eqref{e:gen}, we see that the scalar product at the right hand side can be split into a
{\it diagonal part}, corresponding to making the same choice for the indices $i,j$ in the two occurrences of $\genap$, 
an {\it off-diagonal part of type 1}, corresponding to choosing $i,j$ and $i',j'$ with one element in common, 
and {\it off-diagonal part of type 2}, corresponding to all remaining choices.  
(See for instance \cite[Lemma 3.6]{CET}, where a similar distinction was made.)

To write this in formulas, let $\cS^\eps \eqdef (-\gensy-\gensy^\fw\cGN)^{-1}$ and 
denote with $\sigma^\eps =(\sigma^\eps _n)_{n\ge1}$ its Fourier multiplier, i.e. for $\psi\in\Gamma L^2_n$ 
$\cF(\cS^\eps  \psi)(k_{1:n})=\sigma^\eps _n(k_{1:n})\hat \psi(k_{1:n})$ where 
\begin{equs}[e:sigman]
  \sigma^\eps _n(k_{1:n})&=\frac2{|k_{1:n}|^2+(\fw\cdot k)_{1:n}^2G(\Ll(\frac12|k_{1:n}|^2))}\,.
\end{equs}
Then, the scalar product in~\eqref{e:ScalProd} satisfies
\begin{equ}
	\langle \cS^\eps \genap\psi_1,\genap\psi_2\rangle=\langle \cS^\eps \genap\psi_1,\genap\psi_2\rangle_{\Di} +\sum_{i=1}^2 \langle \cS^\eps \genap\psi_1,\genap\psi_2\rangle_{\oD_i}
\end{equ}
where diagonal part, given by the first summand, is defined as 
\begin{equs}\label{e:Diag}
	\langle \cS^\eps \genap\psi_1,&\genap\psi_2\rangle_{\Di}\eqdef n!\, n \,\frac{4 \lambda_\eps^2}{(2\pi)^2} \times\\
				& \times\sum_{k_{1:n}}\tfrac12(\fw\cdot k_1)^2\overline{\hat\psi_1(k_{1:n})}{\hat\psi_2(k_{1:n})}\sum_{\ell+m=k_1}\sigma^\eps _{n+1}(\ell,m,k_{2:n})  \indN{\ell,m}
\end{equs}
while off-diagonal parts of type $1$ and $2$ are respectively given by 
\begin{equs}\label{e:OffDiag1}
	\langle &\cS^\eps \genap\psi_1,\genap\psi_2\rangle_{\oDi}\eqdef n! \,n\, c_{\oDi}(n)\sum_{k_{1:n+1}} \sigma^\eps _{n+1}(k_{1:n+1})     \indN{k_1,k_2}\indN{k_1,k_3}\times\\
	&\times(\fw\cdot (k_1+k_2))  (\fw\cdot (k_1+k_3))
			\overline{\hat{\psi_1}(k_1+k_2,k_3,k_{4:n+1})}{\hat\psi_2(k_1+k_3,k_2,k_{4:n+1})}\,,
\end{equs}
and
\begin{equs}\label{e:OffDiag2}
\langle &\cS^\eps \genap\psi_1,\genap\psi_2\rangle_{\ooDi}\eqdef n! \,n\, c_{\ooDi}(n)\sum_{k_{1:n+1}} \sigma^\eps _{n+1}(k_{1:n+1})     \indN{k_1,k_2}\indN{k_1,k_3}\times\\
			&\times (\fw\cdot (k_1+k_2))  (\fw\cdot (k_3+k_4))\overline{\hat \psi_1(k_1+k_2,k_{3:4},k_{5:n+1})}{\hat \psi_2(k_3+k_4,k_{1:2},k_{5:n+1})}
\end{equs}
where, for $i=1,\,2$, $c_{\oD_i}(n)$ is an explicit positive constant only depending on $n$ 
and such that $c_{\oD_i}(n)= O(n^{i+1})$.
\medskip

We are ready to control the left hand side of~\eqref{e:Replacement}. At first, we split all the diagonal parts 
from the off-diagonal 
\begin{equs}
|\langle &\big[-\genam(-\gensy-\gensy^\fw\cGN)^{-1}\genap+\gensy^\fw\cGN\big]\psi_1,\psi_2\rangle\label{e:Di}|\\
&\leq  \Big|\langle\cS^\eps \genap\psi_1, \genap\psi_2\rangle_{\Di}+ \langle \gensy^\fw\cGN\psi_1,\psi_2\rangle\Big|
+\sum_{i=1,2}|\langle\cS^\eps \genap\psi_1, \genap\psi_2\rangle_{\mathrm{off}_i}|\,.
\end{equs}
Let us begin by estimating the first summand. 
By~\eqref{e:G},~\eqref{e:L0w} and~\eqref{e:Diag}, we
see that
\begin{equs}
\Big|\langle&(-\gensy-\gensy^\fw\cGN)^{-1}\genap\psi_1, \genap\psi_2\rangle_{\Di}+ \langle \gensy^\fw\cGN\psi_1,\psi_2\rangle\Big|\\
&=n! n\Big|\sum_{k_{1:n}}\tfrac12(\fw\cdot k_1)^2\hat\psi_1(-k_{1:n})\hat\psi_2(k_{1:n})\Big(P^\eps(k_{1:n})- G(\Ll(\tfrac12|k_{1:n}|^2))\Big)\Big|
\end{equs}
where $P^\eps$ is defined as 
\begin{equ}[e:P1]
P^\eps(k_{1:n})\eqdef\frac{\lambda_\eps^2}{\pi^2} \sum_{\ell + m=k_1} \frac{\indN{\ell,m}}{\tilde\Gamma + \tilde\Gamma^wG(\Ll(\tilde\Gamma))}
\end{equ}
and, to shorten the notations, we wrote
\begin{equ}
\tilde\Gamma\eqdef\tfrac12(|\ell|^2+|m|^2+|k_{2:n}|^2)\,,\qquad\tilde\Gamma^w\eqdef \tfrac12((\fw\cdot\ell)^2+(\fw\cdot m)^2+(\fw\cdot k)^2_{2:n})\,.
\end{equ}
By Proposition~\ref{p:Approx}, 
we can upper bound the above by
\begin{equs}
&n! n\sum_{k_{1:n}}\tfrac12(\fw\cdot k_1)^2|\hat\psi_1(k_{1:n})||\hat\psi_2(k_{1:n})|\Big|P^\eps(k_{1:n})- G(\Ll(\tfrac12|k_{1:n}|^2))\Big|\\
&\lesssim \lambda_\eps^2 n! n \sum_{k_{1:n}}|k_1|^2|\hat\psi_1(k_{1:n})||\hat\psi_2(k_{1:n})|=\lambda_\eps^2 n! \sum_{k_{1:n}}|k_{1:n}|^2|\hat\psi_1(k_{1:n})||\hat\psi_2(k_{1:n})|\\
&\lesssim \lambda_\eps^2 \|(-\gensy)^{1/2}\psi_1\|\|(-\gensy)^{1/2}\psi_2\|\label{e:CSH1}
\end{equs}
so that this term satisfies~\eqref{e:Replacement}. 

We are now left to estimate the off diagonal terms. 
This is done by following {\it mutatis mutandis} the same steps as in the
proof of~\cite[Lemma 3.4]{CETWeak}, so we will only point out the necessary changes.
For the off-diagonal terms of type 1, we need to control
\begin{equs}[e:od1]
|\langle\cS^\eps \genap\psi_1, \genap\psi_2&\rangle_{\mathrm{off}_1}|\lesssim  n!c_{\oD_1}(n)\lambda_\eps^2\sum_{j_{1:3},k_{3:n}}|j_1+j_2|\, |j_1+j_3| \\
 & \times \overline{\hat \psi_1(j_1+j_2,j_3,k_{3:n})}
  \hat \psi_2(j_1+j_3,j_2,k_{3:n})\sigma^\eps _{n+1}(j_{1:3},k_{3:n})
\end{equs}
where $\sigma^\eps _{n+1}$, given in~\eqref{e:sigman}, is the Fourier multiplier of $\cS^\eps $ in $\fock_{n+1}$
and, as usual, we have bounded $|\fw\cdot k|\lesssim |k|$.
As in \cite[Eq. (3.21)]{CETWeak}, an application of Cauchy-Schwarz shows that \eqref{e:od1} is upper  bounded by
\begin{equ}[e:od12]
  \lambda_\eps^2c_{\oD_1}(n)\prod_{i=1}^2\Big(n!\sum_{k_{1:n}}|\hat\psi_i(k_{1:n})|^2 |k_1|^2 |k_2|\sum_{j_1+j_2=k_1}\frac1{|j_2|}\sigma^\eps _{n+1}(j_{1:2},k_{2:n})
  \Big)^{\frac12}\,.
\end{equ}
This is the same expression as in \cite[Eq. (3.21)]{CETWeak}, with $\sigma^\eps _{n+1}$ 
replacing $J^N$ there. Since $\sigma^\eps _{n+1}$ satisfies
\[
\sigma^\eps _{n+1}(j_{1:2},k_{2:n})\lesssim \frac1{|j_{1:2}|^2+|k_{2:n}|^2},
  \]
  which is the same estimate that $J^N$ satisfies, one can argue as in \cite{CETWeak} and deduce
  \begin{equ}[e:thesame]
\sum_{j_1+j_2=k_1}\frac1{|j_2|}\sigma^\eps _{n+1}(j_{1:2},k_{2:n})\lesssim \frac1{|k_{1:n}|}\le \frac1{|k_2|}.
\end{equ}
Now, since $c_{\oD_1(n)}\lesssim n^2$,~\eqref{e:od12} is upper bounded 
by a constant times $\lambda_\eps^2 \|\sqrt{\cN}(-\gensy)^{\frac12}\psi_1\|$ $\|\sqrt{\cN}(-\gensy)^{\frac12}\psi_2\|$.

As for the off-diagonal terms of type 2, we proceed similarly, 
recalling that, this time, $ c_{\oD_2}(n)\lesssim n^3$. Namely, 
applying Cauchy-Schwarz and exploiting once more~\eqref{e:thesame}, 
we bound them as
\begin{equs}
  |\langle&\cS^\eps \genap\psi_1, \genap\psi_2\rangle_{\mathrm{off}_2}|\\
  &\lesssim  \lambda_\eps^2c_{\oD_2}(n)\prod_{i=1}^2\Big(n!\sum_{k_{1:n}}|\hat\psi_i(k_{1:n})|^2 |k_1|^2 |k_2||k_3|\sum_{j_1+j_2=k_1}\frac{\sigma^\eps _{n+1}(j_{1:2},k_{2:n})}{|j_1||j_2|}
  \Big)^{\frac12}\\
  &\lesssim \lambda_\eps^2 n! c_{\oD_2}(n)\prod_{i=1}^2\Big(
    \sum_{k_{1:n}}|\hat\psi_i(k_{1:n})|^2\frac{|k_1| |k_2||k_3|}{|k_{1:n}|}
    \Big)^{1/2}\\
    &\lesssim \lambda_\eps^2 n^3 \prod_{i=1}^2\Big(
      n!    \sum_{k_{1:n}}|\hat\psi_i(k_{1:n})|^2|k_1|^2
      \Big)^{1/2}\lesssim \lambda_\eps^2 \prod_{i=1}^2 \|\mathcal N(-\gensy)^{1/2}\psi_i\|
    \end{equs}
    as desired.
\end{proof}

\subsection{A priori estimates on the truncated generator equation}\label{sec:apriori}

The goal of this section is to derive a priori estimates on suitable elements of the 
$n$-th inhomogeneous Fock space, which will be crucial in characterising the limit $\eps\to 0$ 
of the solution to~\eqref{e:BurgersScaled}. The specific functions we need will depend on the dimension,  
as for $d=2$ we want to take full advantage of the Replacement Lemma~\ref{l:Replacement}. Let us begin with some definitions. 

Let $n\in\N$, $i\geq 2$ and $P^n_i$ be the projection onto $n$-th inhomogeneous Fock space with the first 
$i-1$ chaos removed, i.e. onto $\bigoplus_{j=i}^n\Gamma L^2_j$. 
The truncated generator $\genn$, and the corresponding truncated operators $\genan$, $\genapn$ 
and $\genapn$ are respectively defined as 
\begin{equ}[e:genn]
\genn\eqdef \gensy +\genan\,,\quad \genan\eqdef P_i^n\gena P_i^n\quad \text{and}\quad\cA_{i,n}^{\eps,\sigma}\eqdef P_i^n\gena_\sigma P_i^n\,,
\end{equ}
for $\sigma\in\{+,-\}$. 
\medskip

Let $g\in\bigoplus_{j=i}^m\Gamma L^2_j$, $i\leq m\le n$. For $d\geq 2$, let $u^{\eps,n}$ be the unique solution 
of the {\it truncated generator equation} which is defined as 
\begin{equation}\label{eq:TruncGenEq}
-\genn u^{\eps,n}=g\,.
\end{equation}
In the specific case of $d=2$, we further consider $\tilde u^\eps$ which satisfies 
\begin{equ}[e:AnsatzNon]
\tilde u^{\eps}\eqdef (-\gensy-\gensy^\fw\cGN)^{-1}\genap \tilde u^{\eps} +(-\gensy-\gensy^\fw\cGN)^{-1}g
\end{equ}
where $\gensy^\fw$ and $\cGN$ are defined according to~\eqref{e:L0w} and~\eqref{e:G} respectively. 
Note that, even if in~\eqref{e:AnsatzNon} there is no truncation in the chaos, the equation admits a unique solution. 
Indeed, it is a triangular equation 
which can be explicitly solved starting from $\tilde u^{\eps}_{1}=\dots=\tilde u^{\eps}_{i-1}=0$, and then inductively 
setting 
\begin{equ}
\tilde u^{\eps}_{j}\eqdef (-\gensy-\gensy^\fw\cGN)^{-1}\genap \tilde u^{\eps}_{j-1} +(-\gensy-\gensy^\fw\cGN)^{-1}g_{j}\,,\qquad j\geq i\,.
\end{equ}
That said, we will be concerned with the projection of $\tilde u^\eps$ onto $\oplus_{j=i}^n\fock_j$, that 
we denote by $\tilde u^{\eps,n}$ and can be easily seen to solve
\begin{equ}[e:Ansatz]
\tilde u^{\eps,n}\eqdef (-\gensy-\gensy^\fw\cGN)^{-1}\genap P_i^{n-1} \tilde u^{\eps,n} +(-\gensy-\gensy^\fw\cGN)^{-1}g\,.
\end{equ}
\medskip

{As we will see in the proof of the main convergence theorem (see Section~\ref{sec:Conv}), 
a case which will play a crucial role for us is when the input $g$ above 
coincides with the action of $\genap$ on a smooth element of a fixed Fock space. 
For this, consider
test functions $\phi, \psi\in\cS(\T^d)$ 
and let $f_1\eqdef\phi\in\fock_1$ and \[f_2=[\phi\otimes\psi]_{sym}\eqdef \frac{\phi\otimes\psi+\psi\otimes \phi}2\] be the symmetric version of $\phi\otimes\psi$, 
which lives in $\fock_2$.}
For $i=2,3$, we will study $\vN$ (that we distinguish from 
$u^{\eps,n}$ as the right-hand side is fixed $\genap f_{i-1}$), 
that is the solution to the equation 
\begin{equ}[e:nonlinGenEq]
-\genn \vN={\genap f_{i-1}}\,.
\end{equ}
and for $d=2$, $\tvN$ that solves
\begin{equ}[e:nonlinAnsatz]
\tvN= (-\gensy-\gensy^\fw\cGN)^{-1}\genap P_i^{n-1} \tvN +(-\gensy-\gensy^\fw\cGN)^{-1}{\genap f_{i-1}}\,.
\end{equ} 

We are ready to state the main result of this section. 

\begin{theorem}\label{thm:GenEqu}
{Let $\phi,\psi\in\cS(\T^d)$, $f_1\eqdef\phi$ and $f_2=[\phi\otimes\psi]_{sym}$ be the symmetric version of $\phi\otimes\psi$}.  {Let $i=2,3$}. For $n\in\N$, let $\wN$ 
be the solution of the truncated generator equation $\vN$ in~\eqref{e:nonlinGenEq} if $d\geq 3$ or be $\tvN$
given by~\eqref{e:nonlinAnsatz} if $d=2$. Then, 
for any $k\in\N$, there exists a constant $C>0$, depending only on $k,m$ and the dimension $d$, 
and $\eps_d(n)>0$ such that for any $\eps<\eps_d(n)$ the following estimate holds
\begin{equ}
\|(-\gensy)^{-\frac12}(-\gen \wN-{\genap f_{i-1}+\genam \wN_i)\|\leq \frac{C}{n^{k}}\|(-\gensy)^{\frac{1}{2}}f_{i-1}\|}\,,\label{e:ApproxH1}
\end{equ}
where we recall that $ \wN_i$ is the component of $w^{\varepsilon,n}$ in the $i^{th}$ chaos.
If $d\geq 3$, $\eps_d(n)$ can be taken to be $1$. 
\end{theorem}

%

As a first step in the proof of the previous theorem, we derive weighted estimates 
on the solution to~\eqref{e:nonlinGenEq} which hold in {\it any} dimension $d\geq 2$. 
The argument we exploit follows closely 
that of \cite[Lemma 2.5]{LandimYau1997} and since the proof works for any choice of $\psi$, 
we formulate and prove the result in the more 
general case of~\eqref{eq:TruncGenEq}.

\begin{proposition}\label{P:AprioriGeneratorEq}
Let $m\in\N$, {$i=2,3$} and $g\in\bigoplus_{j=i}^m\Gamma L^2_j$. For $n\in\N$, $n\geq m$, let $u^{\eps,n}$ 
be the solution of the truncated generator equation in~\eqref{eq:TruncGenEq} in any dimension $d\geq 2$. 
Then, there exists a 
positive constant $C=C(m,k)$ independent of $n,\eps$ and $\psi$, such that
\begin{equation}\label{e:AprioriMain}
\|\mathcal N^k (-\gensy)^\frac12 u^{\eps,n}\|\leq C\|(-\gensy)^{-\frac12}g\|\,,
\end{equation}
where $\cN$ is the number operator in Definition~\ref{def:NoOp}. 
\end{proposition}
\begin{proof}
As $n$ and $\eps$ are fixed throughout the proof, we will write $u$ and $u_j,\,j=1,\dots,n$ in place of 
$u^{\eps,n}$ and $u^{\eps,n}_j$, respectively. 
Also, by convention, we let $u_{n+1}=0=u_{i-1}=\dots=u_1$. To denote constants which do not depend on $\eps,n$ or $g$ 
we will use $C$, and such $C$ might change from line to line. 
%

Let us test the $j$-th component, $j\geq i$, with $\gen u$, which, since by Lemma~\ref{lem:essid} 
$(\genap)^\ast=-\genam$, gives
  \begin{equs}
    \langle u_j,\gen u\rangle&= \langle u_j, \cL_0 u_j\rangle+\langle u_j,\genap u_{j-1}\rangle+\langle u_j,\genam u_{j+1}\rangle\\
    &=  \langle u_j, \cL_0 u_j\rangle- \Big[\langle u_{j+1},\genap u_j\rangle-\langle u_{j},\genap u_{j-1}\rangle\big].
  \end{equs}
Via a summation by parts, for any $a>0$, we have
  \begin{equs}
    \sum_{j=i}^n (a+j^{2k})\langle u_j,(-\cL_0) u_j\rangle=  &\sum_{j=i}^n (a+j^{2k})\langle u_j,(-\gen) u\rangle\\
    &+\sum_{j=i}^n (j^{2k}-(j-1)^{2k})\langle u_j,\genap u_{j-1}\rangle.
  \end{equs}
  Next, note that $-\gen u =g-\genap u_n-\genam u_i$ so that, by orthogonality of Fock spaces with different indices,
  \begin{equ}
    |\langle u_j,(-\gen) u\rangle|= |\langle u_j,g_j\rangle|\le \frac12 \langle u_j,(-\cL_0)u_j\rangle+\frac12\langle g_j,(-\cL_0)^{-1}g_j\rangle.
  \end{equ}
  As a consequence, there exists some finite strictly positive constant $C=C(a,m,k)$, 
  which might change from line to line,  such that
  \begin{equs}[eq:freddo]
\sum_{j=i}^n &(a+j^{2k})\langle u_j,(-\cL_0) u_j\rangle\\
&\le C\Big(\|(-\cL_0)^{-\frac12}g\|^2+\sum_{j=i}^n (j^{2k}-(j-1)^{2k})\langle u_j,\genap  u_{j-1}\rangle\Big)\\
&\leq C\Big(\|(-\cL_0)^{-\frac12}g\|^2+\sum_{j=i}^n j^{2k-1}\langle u_j,\genap  u_{j-1}\rangle \Big)   
  \end{equs}
where $ C= C(k)$ is a constant 
chosen in such a way that $(j^{2k}-(j-1)^{2k})\sqrt j\le  C j^{2k-\frac12}$. 
To handle the last sum, we bound
  \begin{equs}[e:Gammaaa]
    |\langle u_j,\genap  u_{j-1}\rangle |&\le \|(-\cL_0)^{\frac12} u_j\| \| (-\cL_0)^{-\frac12}\genap  u_{j-1}\|\\
    &\leq C\sqrt{j} \|(-\cL_0)^{\frac12} u_j\| \| (-\cL_0)^{\frac12} u_{j-1}\|\\
   &\leq C\sqrt{j}\Big(\frac12 \|(-\cL_0)^{\frac12} u_j\|^2+\frac12\| (-\cL_0)^{\frac12} u_{j-1}\|^2\Big)
  \end{equs}
where in the second bound we used~\eqref{e:Abound} 
in Lemma~\ref{l:GenaBound}. We now plug the result into~\eqref{eq:freddo} and rearrange the terms, 
so that we conclude 
\begin{equ}\label{eq:freddo2}
\sum_{j=i}^n j^{2k}\Big(1+\frac a{j^{2k}}-\frac{C}{\sqrt j}\Big)\langle u_j,(-\cL_0) u_j\rangle\le C(1+\|(-\cL_0)^{-\frac12}g\|^2)\,. \end{equ}
Choosing $a$ sufficiently large, 
in such a way that $(1+\frac a{j^{2k}}-\frac{C}{\sqrt j})\ge 1/2$ for every $j\ge2$,~\eqref{e:AprioriMain} follows.
\end{proof}

In the next two lemmas, we consider the specific case of $d=2$. 
In the first one, we show that, for any $g$, $\tilde u^{\eps,n}$ defined according to~\eqref{e:Ansatz} 
can be estimated in terms of $g$. 

\begin{lemma}\label{l:aprioriAnsatz}
For $d=2$, let $g\in\bigoplus_{j=i}^m\Gamma L^2_j$, $n\in\N$, $i\leq m\leq n$ 
and $\tilde u^{\eps,n}$ be given by~\eqref{e:Ansatz}. 
Then, 
there exists a constant $C>0$ independent of $n,\eps$ and $\eps_2(n)>0$ such that for every $\eps<\eps_2(n)$ 
\begin{equ}[e:apriorivtilde]
\|(-\gensy)^{1/2}\tilde u^{\eps,n}\|^2
\leq C \|(-\gensy)^{-1/2}g\|^2\,.
\end{equ}
\end{lemma}
\begin{proof}
At first, we study how the generator $\gen$ acts on $\tilde u^{\eps,n}$. We have
\begin{equs}[e:genAnsatz]
-\gen \tilde u^{\eps,n}=&(-\gensy-\gensy^\fw\cGN)\tilde u^{\eps,n}-\genap  \tilde u^{\eps,n}-\genam \tilde u^{\eps,n}+\gensy^\fw\cGN \tilde u^{\eps,n}\\
=&g -\genam (-\gensy-\gensy^\fw\cGN)^{-1}g -\genap \tilde u^{\eps,n}_n+\gensy^\fw\cGN \tilde u^{\eps,n}_n\\
&+\big[-\genam(-\gensy-\gensy^\fw\cGN)^{-1}\genap+\gensy^\fw\cGN\big]P_i^{n-1} \tilde u^{\eps,n}
\end{equs}
where in the last step we used~\eqref{e:Ansatz}. 
Now, $\genap \tilde u^{\eps,n}_n\in\fock_{n+1}$, so that, since $ \tilde u^{\eps,n}\in\oplus_{j=i}^n\fock_j$, 
the two are orthogonal. Therefore, by testing both sides by $ \tilde u^{\eps,n}$, we obtain
\begin{equs}
\|&(-\gensy)^{1/2} \tilde u^{\eps,n}\|^2\\
&=\langle  \tilde u^{\eps,n}, g-\genam (-\gensy-\gensy^\fw\cGN)^{-1}g\rangle+\langle  \tilde u^{\eps,n}_n, \gensy^\fw\cGN \tilde u^{\eps,n}_n\rangle\\
&\quad+\langle  \tilde u^{\eps,n}, \big[-\genam(-\gensy-\gensy^\fw\cGN)^{-1}\genap+\gensy^\fw\cGN\big]P_i^{n-1} \tilde u^{\eps,n}\rangle\\
&\leq \tfrac12\|(-\gensy)^{1/2} \tilde u^{\eps,n}\|^2+\tfrac12\|(-\gensy)^{-1/2}[g-\genam (-\gensy-\gensy^\fw\cGN)^{-1}g]\|^2 \\
&\quad- \|(-\gensy^\fw\cGN)^{1/2} \tilde u^{\eps,n}_n\|^2+C\lambda_\eps^2 n^2 \|(-\gensy)^{1/2} \tilde u^{\eps,n}\|^2\\
&\leq C'\|(-\gensy)^{-1/2}g\|^2+(\tfrac12 + Cn^2\lambda_\eps^2)\|(-\gensy)^{1/2} \tilde u^{\eps,n}\|^2\label{e:apriorivtilde1}
\end{equs}
where in the last step we neglected the negative term at the right hand side, we used~\eqref{e:Tbound} together with 
the positivity of $-\gensy^\fw\cGN$ and the fact that 
the operators $\gensy,\gensy^\fw$ and $\cGN$ commute to control the second term containing $g$, 
and the Replacement Lemma~\ref{l:Replacement} to bound the third summand. 
Now, we choose $\eps_2(n)$ in such a way that $C n^2\lambda_{\eps_2(n)}^2<1/2$. 
Therefore,~\eqref{e:apriorivtilde} follows upon rearranging the terms in~\eqref{e:apriorivtilde1}. 
\end{proof}

In the next lemma, we turn our attention to $\tvN$ given in~\eqref{e:nonlinAnsatz} and
prove that it satisfies the same weighted estimates as those in~\eqref{e:AprioriMain}. 
To do so, in view of Proposition~\ref{P:AprioriGeneratorEq} and the previous Lemma, 
it suffices to control the difference between $\tvN$ and the solution of the truncated generator equation $\vN$. 

\begin{lemma}\label{l:AnsatzvsGenEq}
Let $d=2$. For $n\in\N$ and $i=2,3$, 
let $\tvN$ be given by~\eqref{e:nonlinAnsatz} and {$f_1,f_2$ be as in Theorem~\ref{thm:GenEqu}}. 
Then, for any $k\in\N$ there exists a constant $C>0$ and $\eps_2(n,k)>0$ such that for every $\eps<\eps_2(n,k)$ 
\begin{equ}[e:AprioriAnsatz]
\|\cN^k(-\gensy)^{\frac12}\tvN\|\leq C{\|(-\gensy)^{\frac12}f_{i-1}\|}\,.
\end{equ}
\end{lemma}
\begin{proof}
Let $\vN$ be the solution of~\eqref{e:nonlinGenEq} in $d=2$. Notice that, trivially, 
\begin{equ}[e:TwoSum]
\|\cN^k(-\gensy)^{1/2}\tvN\|\leq n^k\|(-\gensy)^{1/2}(\tvN-\vN)\|+\|\cN^k(-\gensy)^{1/2}\vN\|\,.
\end{equ}
Now, in view of Proposition~\ref{P:AprioriGeneratorEq}, the second summand is bounded by (a constant times) 
$\|(-\gensy)^{-1/2}{\genap f_{i-1}}\|$. For the first, 
we claim that for any $j$ there exists $C>0$ such that for $\eps$ small enough, we have
\begin{equ}[e:AprioriDiff]
\|(-\gensy)^{1/2}(\tvN-\vN)\|^2\leq C\big( n^{-j} +n^2\lambda_\eps^2\big)\|(-\gensy)^{-1/2}{\genap f_{i-1}}\|^2\,.
\end{equ}
Assuming the claim, we choose in~\eqref{e:AprioriDiff} 
$j=k$ and $\eps_2(n,k)>0$ such that $n^{k+2}\lambda_{\eps_2(n,k)}<1$, 
{ so that we obtain 
\begin{equ}
\|\cN^k(-\gensy)^{1/2}\tvN\|\lesssim \|(-\gensy)^{-1/2}\genap f_{i-1}\|\lesssim \|(-\gensy)^{1/2} f_{i-1}\|^2
\end{equ}
where the last step follows by~\eqref{e:Abound}. Hence,~\eqref{e:AprioriAnsatz} is proved. }
\medskip

\newcommand{\pes}{g_\eps^\sharp}
We now prove~\eqref{e:AprioriDiff}. 
To shorten the notation, set $\pes\eqdef (-\gensy-\gensy^\fw\cGN)^{-1}{\genap f_{i-1}}$. 
We begin by evaluating $-\gen$ on $\tvN-\vN$. 
To do so, we exploit~\eqref{e:genAnsatz} and the fact that, as noted in the proof of Proposition~\ref{P:AprioriGeneratorEq}, 
since $\vN$ solves~\eqref{e:nonlinGenEq}, we have 
\begin{equ}
-\gen \vN-{\genap f_{i-1}+\genam \vN_i}
=  -\genap \vN_n\,.
\end{equ}
This means that 
\begin{equs}[e:genDiff]
-\gen (\tvN-\vN)=&-\genam (\pes-\vN_i)-\genap(\tvN_n-\vN_n)+\gensy^\fw\cGN \tvN_n\\
&+\big[-\genam(-\gensy-\gensy^\fw\cGN)^{-1}\genap+\gensy^\fw\cGN\big]P_i^{n-1} \tvN\,.
\end{equs}
Now, since $\pes\in\fock_i$, $\genam (\pes-\vN_i)\in\fock_{i-1}$ and 
$\genap(\vN_n-\tvN_n)\in\fock_{n+1}$, which implies that they are orthogonal 
to both $\tvN$ and $\vN$ as these belong to $\oplus_{j=i}^n\fock_j$. 
Hence, by testing both sides of~\eqref{e:genDiff} by $\tvN-\vN$, we obtain
\begin{equs}\label{e:aprioriDiff1}
\|(-\gensy)^{1/2}&(\tvN-\vN)\|^2=\langle \tvN-\vN, \gensy^\fw\cGN \tvN_n\rangle\\
&+\langle \tvN-\vN, \big[-\genam(-\gensy-\gensy^\fw\cGN)^{-1}\genap+\gensy^\fw\cGN\big]P_i^{n-1}\tvN\rangle\,.
\end{equs}
Let us analyse the two terms at the right hand side separately. 
For the first, note that the operator $\gensy^\fw\cGN$ is negative, so that 
\begin{equs}
\langle \tvN-\vN, \gensy^\fw\cGN\tvN_n\rangle&=\langle \vN_n, \gensy^\fw\cGN\tvN_n\rangle -\langle \tvN_n, (-\gensy^\fw\cGN)\tvN_n\rangle\\
&\leq \langle -\vN_n, (-\gensy^\fw\cGN)\tvN_n\rangle\\
&\lesssim \|(-\gensy)^{1/2}\vN_n\|\|(-\gensy)^{1/2}\tvN_n\|
\end{equs}
and in the last bound we used that $-\gensy^\fw\cGN\lesssim -\gensy$ and that $\gensy, \gensy^\fw$ and $\cGN$ commute. 
Now, the a priori estimates in Proposition~\ref{P:AprioriGeneratorEq} and 
Lemma~\ref{l:aprioriAnsatz} allow to upper bound the previous by
\begin{equs}
 n^{-k} \|\cN^k(-\gensy)^{1/2}\tvN\|\|(-\gensy)^{1/2}\vN_n\|\lesssim n^{-k} \|(-\gensy)^{-1/2}{\genap f_{i-1}}\|^2\,.
\end{equs}
For the second term in~\eqref{e:aprioriDiff1}, we apply the Replacement Lemma~\ref{l:Replacement} 
first and the same a priori estimates as above, thus getting
\begin{equs}
\langle \tvN-\vN, &\big[\genam(-\gensy-\gensy^\fw\cGN)^{-1}\genap+\gensy^\fw\cGN\big]P_i^{n-1}\tvN\rangle\\
&\leq C \lambda_\eps^2n^2\|(-\gensy)^{1/2}(\tvN-\vN)\|\|(-\gensy)^{1/2}\tvN\|\\
&\leq \tfrac12C\lambda_\eps^2n^2\Big(\|(-\gensy)^{1/2}(\tvN-\vN)\|^2+\|(-\gensy)^{1/2}\tvN\|^2\Big)\\
&\leq \tfrac12C\lambda_\eps^2n^2\Big(\|(-\gensy)^{1/2}(\tvN-\vN)\|^2+\|(-\gensy)^{-1/2}{\genap f_{i-1}}\|^2\Big)\,,
\end{equs}
where the constant $C$ changed in the last two lines. Now, by using that for $\eps<\eps_2(n,k)$, 
$C n^2\lambda_\eps^2<1$, collecting the previous bounds and rearranging the terms,~\eqref{e:AprioriDiff} follows. 
\end{proof}

We are now ready for the proof of Theorem~\ref{thm:GenEqu}, which is an easy consequence 
of Lemma~\ref{l:GenaBound} and, for $d\geq 3$, Proposition~\ref{P:AprioriGeneratorEq}, while, 
for $d=2$ Lemmas~\ref{l:aprioriAnsatz} and~\ref{l:AnsatzvsGenEq}.

\begin{proof}[of Theorem~\ref{thm:GenEqu}]
Let us first treat the case of $d\geq 3$. As noted in the proof of Proposition~\ref{P:AprioriGeneratorEq}, 
since $\vN$ solves~\eqref{e:nonlinGenEq}, we have 
\begin{equ}[e:genu]
-\gen \vN-{\genap f_{i-1}+\genam \vN_i}
=  -\genap \vN_n\,.
\end{equ}
By~\eqref{e:Abound}, for any $k\ge 1$, we have 
\begin{equ}[e:chaosdecay]
\|(-\gensy)^{-\tfrac12}\genap \vN_n\|\lesssim \sqrt{n}\|(-\gensy)^{\tfrac12}\vN_n\|\leq \frac1{n^{k}}\|\cN^{k+1}(-\gensy)^{\tfrac12} \vN\|\,.
\end{equ}
Hence,~\eqref{e:ApproxH1} follows directly by~\eqref{e:AprioriMain} and~\eqref{e:Abound}. 
\medskip

We now turn to $d=2$, in which case $\tvN$ satisfies~\eqref{e:nonlinAnsatz}. 
Notice that, by the recursive definition of $\tvN$, $\tvN_1=\dots=\tvN_{i-1}=0$ and therefore
\begin{equ}
\tvN_i=(-\gensy-\gensy^\fw\cGN)^{-1}\genap f_{i-1}\,.
\end{equ}
Hence, upon replacing $\psi$ with $\genap f_{i-1}$ in~\eqref{e:genAnsatz}, we see that 
\begin{equs}
-\gen \tvN-{\genap f_{i-1}}+&{\genam \tvN_i}\\
=&-\gen \tvN-{\genap f_{i-1}}+\genam (-\gensy-\gensy^\fw\cGN)^{-1}{\genap f_{i-1}}\\
=&-\genap \tvN_n+\gensy^\fw\cGN \tvN_n\\
&+\big[-\genam(-\gensy-\gensy^\fw\cGN)^{-1}\genap+\gensy^\fw\cGN\big]P_i^{n-1} \tvN\,.
\end{equs} 
Now, for the first two terms we use Lemma~\ref{l:GenaBound} and $-\gensy^\fw\cGN\lesssim-\gensy$, 
so that we can bound them by 
\begin{equs}
\|(-\gensy)^{-1/2}[-\genap \tvN_n+\gensy^\fw\cGN \tvN_n]\|&\lesssim \sqrt{n}\|(-\gensy)^{1/2}\tvN_n\|\\
&\lesssim n^{-k}\|(-\gensy)^{1/2}{f_{i-1}}\|
\end{equs}
where the last step follows by Lemma~\ref{l:AnsatzvsGenEq}. 
For the third term, we note that 
the $\fock$-norm satisfies a variational formulation, i.e.
\begin{equs}
\|&(-\gensy)^{-1/2}\big[-\genam(-\gensy-\gensy^\fw\cGN)^{-1}\genap+\gensy^\fw\cGN\big]P_i^{n-1} \tvN\|\\
&=\sup_{\|\rho\|=1}\langle\rho, (-\gensy)^{-1/2}\big[-\genam(-\gensy-\gensy^\fw\cGN)^{-1}\genap+\gensy^\fw\cGN\big]P_i^{n-1} \tvN\rangle\\
&=\sup_{\|\rho\|=1}\langle(-\gensy)^{-1/2}\rho, \big[-\genam(-\gensy-\gensy^\fw\cGN)^{-1}\genap+\gensy^\fw\cGN\big]P_i^{n-1} \tvN\rangle\\
&\lesssim C \lambda_\eps^2 \sup_{\|\rho\|=1}\|\rho\|\|\cN^2(-\gensy)^{1/2}\tvN\|=C \lambda_\eps^2\|\cN^2(-\gensy)^{1/2}\tvN\|\\
&\leq C \lambda_\eps^2n^{2}\|(-\gensy)^{1/2}{f_{i-1}}\|\label{e:VarRepl}
\end{equs}
where in the third step we used the Replacement Lemma~\ref{l:Replacement} and in the last, 
Lemma~\ref{l:aprioriAnsatz}. By choosing $\eps_2(n,k)>0$ such that $\lambda_{\eps_2(n,k)}^2n^{2+k}<1$,~\eqref{e:ApproxH1} follows at once. 
%
\end{proof}

\subsection{The diffusivity}

{The main advantage of Theorem~\ref{thm:GenEqu} is that it allows to replace the nonlinearity with 
the sum of two terms: one involves the generator of our process, and therefore encodes its fluctuations, 
while the other consists of $\genam$ applied to an element in the Fock space of degree $i=2,3$,
so that overall it is of degree $i-1$. This suggests that this latter term should give us the additional 
diffusivity appearing in the limit~\eqref{e:SHEintro}. 
The goal of this section is to analyse this object and determine its behaviour. 

The analysis for $d=2$ and $d\geq 3$ will be very different, but the limits we need to establish 
have the same expression. Let us first state the main theorem and then provide different proofs in 
the two cases. 

\begin{theorem}\label{thm:MainConv2}
  Let $\phi,\psi\in\cS(\T^d)$, $f_1\eqdef\phi$ and $f_2=[\phi\otimes\psi]_{sym}$. For
  $n\in\N$, let $\wN$ be the solution of the truncated generator
  equation $\vN$ in~\eqref{e:nonlinGenEq} if $d\geq 3$ or be $\tvN$
  given by~\eqref{e:nonlinAnsatz} if $d=2$. {Let
    $i=2,3$}. Then
\begin{equ}
\lim_{n\to\infty}\limsup_{\eps\to 0} \|(-\gensy)^{-\frac12}[\genam\wN_i-\cD f_{i-1}]\|=0\,,\label{e:Diffusivity2}
\end{equ}
where $\cD$ is the operator introduced in Theorem~\ref{thm:main}, given by $\cD\eqdef D_{\SHE}\gensy^\fw$ 
for $\gensy^\fw$ defined according to~\eqref{e:L0w} and $D_\SHE>0$ as in the aforementioned statement. 
Further, we have  
\begin{equ}
\lim_{\eps\to0}\|\wN\|=0\,.\label{e:L2norm2}
\end{equ}
\end{theorem}

Let us begin with the proof the above theorem in $d=2$ since, thanks to the replacement lemma, it is way easier. 

\begin{proof}[of Theorem~\ref{thm:MainConv2} in $d=2$]
Notice first that the recursive definition of $\tvN$ in~\eqref{e:nonlinAnsatz} immediately gives
\begin{equ}[e:A-D2]
\genam \tvN_i=\genam (-\gensy-\gensy^\fw\cGN)^{-1}\genap f_{i-1}\,.
\end{equ}
As a consequence, we have 
\begin{equs}
\|(-\gensy)^{-\frac12}&[\genam\tvN_i-\cD f_{i-1}]\|\\
=&\|(-\gensy)^{-\frac12}\big[\genam (-\gensy-\gensy^\fw\cGN)^{-1}\genap f_{i-1}-\cD f_{i-1}]\|\\
\leq&\|(-\gensy)^{-\frac12}\big[\genam (-\gensy-\gensy^\fw\cGN)^{-1}\genap-\gensy^\fw\cGN\big] f_{i-1}]\|\\
&+\|(-\gensy)^{-\frac12}[\gensy^\fw\cGN-\cD]f_{i-1}\|
\end{equs}
For the first term at the right hand side, we use the same variational formulation as in~\eqref{e:VarRepl} so that 
following the same steps, we deduce
\begin{equs}
\|(-\gensy)^{-\frac12}&\big[\genam (-\gensy-\gensy^\fw\cGN)^{-1}\genap-\gensy^\fw\cGN\big] f_{i-1}]\|\\
&=\sup_{\|\rho\|=1}\langle (-\gensy)^{-\frac12}\rho, \big[-\genam(-\gensy-\gensy^\fw\cGN)^{-1}\genap+\gensy^\fw\cGN\big] f_{i-1}\rangle\\
&\lesssim  \lambda_\eps^2 \sup_{\|\rho\|=1}\|\rho\|\|(-\gensy)^{\frac12}f_{i-1}\|= \lambda_\eps^2\|(-\gensy)^{\frac12}f_{i-1}\|
\end{equs}
where in the last step we used the Replacement Lemma~\ref{l:Replacement}. 
For the second, the analysis is identical for $f_1$ or $f_2$, so we only consider the latter. 
Using the explicit expression of $\gensy^\fw$, $\cGN$,  $\cD$ and $f_2$, we see that
\begin{equs}
\|(-\gensy)^{-\frac12}&\big[\gensy^\fw\cGN-\cD\big]f_2\|^2\\
&=\frac12\sum_{k_{1:2}} \frac{(\fw\cdot k)_{1:2}^4}{|k_{1:2}|^2}\Big[G(\Ll(\tfrac12|k_{1:2}|^2))- D_\SHE\Big]^2
\\&\times(|\hat\phi(k_1)|^2|\hat\psi(k_2)|^2+\hat\phi(k_1)\hat\phi(k_2)\hat\psi(-k_1)\hat\psi(-k_2))\,.
\end{equs} 
Since $\phi$ is smooth we can take the limit in $\eps$ inside the sum and, provided we 
take $D_\SHE=G(1)$, we conclude that the right hand side goes to $0$. 
Hence, the proof of~\eqref{e:Diffusivity2} is completed. 
\medskip

Concerning~\eqref{e:L2norm2}, we note that for any $j\in\N$ and $\psi\in\fock_j$, 
\begin{equs}
\|(-\gensy-&\gensy^\fw\cGN)^{-1}\genap\psi\|^2\\
&\lesssim  j^2\lambda_\eps^2\sum_{k_{1:j+1}}\frac{\indN{k_1,k_2}(\fw\cdot (k_1+k_2))^2|\hat\psi(k_1+k_2, k_{3:j+1})|^2}{[\frac12 |k_{1:j+1}|^2+\frac12(\fw\cdot k)^2_{1:j+1}G(\Ll(\frac12|k_{1:j+1}|^2))]^2}\\
&\lesssim j^2\lambda_\eps^2\sum_{k_{1:j}}(\fw\cdot k_1)^2|\hat\psi(k_{1:j})|^2\sum_{\ell+m=k_1}\frac{1}{ (|\ell|^2+|m|^2)^2}\,.
\end{equs}
Now, the inner sum is clearly finite, and therefore we deduce 
\begin{equ}[e:L2control]
\|(-\gensy-\gensy^\fw\cGN)^{-1}\genap\psi\|^2\lesssim\lambda_\eps^2\|\cN (-\gensy)^{1/2}\psi\|^2\,.
\end{equ}
We apply the previous estimate to the definition of $\tvN$, so that 
we obtain
\begin{equs}[e:uffa]
\|\tvN\|^2&\leq \|(-\gensy-\gensy^\fw\cGN)^{-1}\genap [\tvN +f_{i-1}]\|^2\\
&\lesssim \lambda_\eps^2\|\cN(-\gensy)^{1/2} [\tvN +f_{i-1}]\|^2
\end{equs}
where in the first step, we replaced the truncated operator $\cA^{\eps,+}_{i,n}$ with $\genap$. 
By the weighted a priori estimate on $\tvN$ in Lemma~\ref{l:AnsatzvsGenEq}, 
we see that the norm at the right hand side of~\eqref{e:uffa} is bounded uniformly in $\eps$. 
Now, since $\lambda_\eps$ goes to $0$ as $\eps\to 0$,~\eqref{e:L2norm2} follows. 
\end{proof}

We now turn to the case $d\geq 3$. The proof of~\eqref{e:Diffusivity2} (and~\eqref{e:L2norm2}) requires 
to be able to pass to limit first in $\eps$ and then in $n$. 
Since the first limit is the most delicate, we single it out in the statement below, which 
also includes~\eqref{e:L2norm2}, and postpone its proof 
to the next subsection. 

\begin{proposition}\label{p:epsLim}
  Consider the setting of Theorem~\ref{thm:MainConv2} for $d\ge 3$. 
Then, for every $n$ fixed, there exists a unique constant $D^n>0$ 
such that, for $i=2,3$,
\begin{equ}[e:nH1limit]
\lim_{\eps\to 0} \|(-\gensy)^{-\frac12}[\genam\vN_i-\cD^{n-i+2} f_{i-1}]\|=0\,,
\end{equ}
where $\cD^n$ is the operator given by $\cD^n\eqdef D^n\gensy^\fw$ 
for $\gensy^\fw$ defined according to~\eqref{e:L0w}.

Further, as $\eps\to 0$, $\vN$ converges to $0$ in $\fock$.
\end{proposition}

We now turn to the proof of Theorem~\ref{thm:MainConv2} for $d\geq 3$, for which, thanks to Proposition~\ref{p:epsLim} 
we only need to ensure that the limit in $n$ of the constants $D^n$ is unique and strictly positive.

\begin{proof}[of Theorem~\ref{thm:MainConv2} in $d\geq 3$]
Notice that~\eqref{e:L2norm2} was established in Proposition~\ref{p:epsLim}, so that we only need to 
focus on~\eqref{e:Diffusivity2}. For this, in light of~\eqref{e:nH1limit}, it suffices to prove that 
the sequence of operators $\{\cD^n\}_n$ converges to a unique limit, which translates into 
showing that the sequence of constants $\{D^n\}_n$ is Cauchy. 
The definition of the operator $\cD^n$ ensures that for any $f_1=\phi\in\cS(\T^d)$, we have 
\begin{equ}[e:D1]
\|(-\gensy)^{-\frac12}[\cD^{n+1}-\cD^n]f_1\|=|D^{n+1}-D^n|\|(-\gensy)^{-\frac12}(-\gensy^\fw)f_1\|\,.
\end{equ}
Furthermore,  the left hand side can be bounded by 
\begin{equs}[e:D2]
\|(-\gensy)^{-\frac12}[\cD^{n+1}-\cD^n]f_1\|\leq& \|(-\gensy)^{-\frac12}[\genam v^{\eps,n+1}_2-\cD^{n+1} f_{1}]\|\\
&+\|(-\gensy)^{-\frac12}[\genam\vN_2-\cD^{n} f_{1}]\| \\
&+\|(-\gensy)^{-\frac12}\genam(v^{\eps,n+1}_2-\vN_2)\|\,,
\end{equs}
where, for $\eps>0$, any $n\in\N$ and $\phi=f_1$ any smooth function, 
$\vN$ is the solution of the corresponding generator equation~\eqref{e:nonlinGenEq} truncated at level $n$.
To control the last term, we apply~\eqref{e:Abound} to get
\begin{equ}[e:A-Cauchy]
\|(-\gensy)^{-\frac12}\genam(v^{\eps,n+1}_2-\vN_2)\|\lesssim \|(-\gensy)^{\frac12}(v^{\eps,n+1}-\vN)\|\,.
\end{equ}
By definition of $v^{\eps,n+1}$ and $\vN$, we deduce that 
\begin{equs}
-\gen (v^{\eps,n+1}-v^{\eps,n})=&-\gen_{2,n+1}v^{\eps,n+1}-\genap v_{n+1}^{\eps,n+1}-\genam v_2^{\eps,n+1}\\
&+\gen_{2,n}v^{\eps,n}+\genap v_n^{\eps,n}+\genam v_2^{\eps,n}\\
=&-\genap (v_{n+1}^{\eps,n+1}-v_n^{\eps,n})-\genam (v_{2}^{\eps,n+1}-v_2^{\eps,n})\,.
\end{equs}
We now test both sides by $v^{\eps,n+1}-v^{\eps,n}$ and, using the antisymmetry of $\gena$, obtain
\begin{equs}
\|(-\gensy)^{1/2}(v^{\eps,n+1}-v^{\eps,n})\|^2&=\langle v^{\eps,n+1}-v^{\eps,n}, -\genap (v_{n+1}^{\eps,n+1}-v_n^{\eps,n})\rangle\\
&=\langle v^{\eps,n+1}-v^{\eps,n}, \genap v_n^{\eps,n}\rangle\label{e:DiffBoundn}
\end{equs}
where the terms containing $\genam (v_{2}^{\eps,n+1}-v_2^{\eps,n})\in\fock_1$ and 
$\genap v_{n+1}^{\eps,n+1}\in\fock_{n+2}$ dropped out because of  
orthogonality of Fock spaces with different indices 
(recall that $v_{1}^{\eps,n+1}=v_1^{\eps,n}=0=v_{n+2}^{\eps,n+1}=v_{n+1}^{\eps,n}$). 
Now, the right hand side can be bounded above by 
\begin{equs}
\langle v^{\eps,n+1}-&v^{\eps,n}, \genap v_n^{\eps,n}\rangle\\
&\leq \tfrac12 \|(-\gensy)^{1/2}(v^{\eps,n+1}-v^{\eps,n})\|^2 +2\|(-\gensy)^{-1/2}\genap v_n^{\eps,n}\|^2\,.
\end{equs}
We can go back to~\eqref{e:DiffBoundn} 
and get 
\begin{equ}[e:Intermezzo]
\|(-\gensy)^{\frac12}(v^{\eps,n+1}-v^{\eps,n})\|^2\lesssim \|(-\gensy)^{-\frac12}\genap v_n^{\eps,n}\|^2\lesssim \|\cN(-\gensy)^{\frac12}v_n^{\eps,n}\|^2
\end{equ}
where the last step follows by Lemma~\ref{l:GenaBound}. To control the last term, 
we apply Proposition~\ref{P:AprioriGeneratorEq}, which, for any $j>1$, gives 
\begin{equs}
\|\cN(-\gensy)^{\frac12}v_n^{\eps,n}\|^2\leq n^{-2(j-1)}\|\cN^{j}(-\gensy)^{\frac12}v_n^{\eps,n}\|^2\lesssim n^{-2j} \|(-\gensy)^{\frac12}f_1\|^2\,.
\end{equs}
At this point, we upper bound the right hand side of~\eqref{e:D2} with the above so that, 
by using~\eqref{e:D1} and taking the limit in $\eps$, ~\eqref{e:nH1limit} ensures that 
\begin{equ}
|D^{n+1}-D^n|\|(-\gensy)^{-\frac12}(-\gensy^\fw)f_1\|\lesssim n^{-2j} \|(-\gensy)^{\frac12}f_1\|^2\,.
\end{equ} 
Since $f_1$ is an arbitrary smooth test function, we can choose it to be such that the norm at the left hand side 
is non-zero. Consequently, the sequence $\{D^n\}_n$ is Cauchy and therefore converges to a unique 
constant $D_\SHE$, so that the proof of~\eqref{e:Diffusivity2} is concluded. 
\medskip

It remains to show that $D_\SHE$ is strictly positive. To do so, fix $f_1\eqdef e_{k}$ for 
$k\in\Z^2$ such that $w\cdot k\neq 0$ and set $\vN[-k]$ to be  
the solution of the truncated generator equation whose input function is $f_1$. Notice that 
\begin{equs}\label{e:Dpos}
D^n \frac{(w\cdot k)^2}{\sqrt 2|k|}&=\|(-\gensy)^{-\frac12}(-\cD^n)f_1\|\\
&\geq \|(-\gensy)^{-\frac12}\genam\vN_2[-k]\| -\|(-\gensy)^{-\frac12}[\genam\vN_2[-k]-\cD^{n} f_{1}]\|\,.
\end{equs}
Let us consider the first term at the right hand side more carefully. 
Since $\genam \vN_2[-k]\in\fock_1$, we can expand it in the Fourier basis, i.e. 
\begin{equs}[e:TransInv]
\genam \vN_2[-k]&=\sum_{j\in\Z^d_0} \langle \genam \vN_2[-k], e_{-j} \rangle e_j\\
&=-\sum_{j\in\Z^d_0} \langle \vN_2[-k], \genap e_{-j} \rangle e_j=- \langle  \vN_2[-k], \genap e_{-k} \rangle e_{k}\,,
\end{equs}
where the last step is a consequence of the translation invariance of
the equation, i.e. the commutation between the operators $\gensy$, $\gena$ with the momentum operator 
as stated in Lemma~\ref{lem:essid}.  We can now exploit orthogonality of
Fock spaces with different indices and the antisymmetry of $\gena$, to
see that
\begin{equs}
\langle \vN_2[-k], \genap e_{-k} \rangle&= \langle \vN[-k], -\gen \vN[k] \rangle \\
&=\langle \vN[-k], -\gensy \vN[k] \rangle=\|(-\gensy)^{1/2}\vN[-k]\|^2\,. 
\end{equs}
It follows that 
\begin{equ}
\|(-\gensy)^{-\frac12}\genam\vN_2[-k]\| = \frac{1}{|k|} \|(-\gensy)^{1/2}\vN[-k]\|^2
\end{equ}
so that we will only establish a lower bound on the latter norm. 
We have 
\begin{equs}
  \label{e:princvar}
\|&(-\gensy)^{1/2}\vN[-k]\|^2=\langle \overline{\genap e_{-k}}, (-\cL^\eps_{2,n})^{-1}\genap e_{-k}\rangle\\
&=\sup_{\rho\in\fock} \left\{2\langle \rho, \genap e_{-k}\rangle -\langle \rho, (-\gensy)\rho\rangle-\langle \rho, (-\gena_{2,n})^\ast(-\gensy)^{-1}(-\gena_{2,n})\rho\rangle\right\}\\
&\geq \sup_{\rho\in\fock_2} \left\{2\langle \rho, \genap e_{-k}\rangle -\langle \rho, (-\gensy)\rho\rangle-\langle \genap\rho, (-\gensy)^{-1}\genap\rho\rangle\right\}\,,
\end{equs}
where in the second equality we used \cite[Theorem 4.1]{Komorowski2012}, in
last step we restricted the supremum to $\fock_2$, thus making it smaller, and exploited 
Lemma~\ref{lem:essid} and the truncation $P^n_2$ in the definition of $\gena_{2,n}$ in~\eqref{e:genn} to 
simplify the last term. 
According to Lemma~\ref{e:Abound}, for $\psi\in\fock_2$, there exists a constant $C>0$ such that 
\begin{equ}
\langle \genap\psi, (-\gensy)^{-1}\genap\psi\rangle=\|(-\gensy)^{-1/2}\genap \psi\|^2\leq C\|(-\gensy)^{1/2} \psi\|^2
\end{equ}
which implies a further lower bound of the form 
\begin{equs}
\|(-\gensy)^{1/2}\vN[-k]\|^2&\geq\sup_{\psi\in\fock_2} \left\{2\langle \psi, \genap e_{-k}\rangle -\langle \psi, (-\gensy)\psi\rangle-C\langle \psi, (-\gensy)\psi\rangle\right\}\\
&=\frac{1}{1+C}\|(-\gensy)^{-1/2}\genap e_{-k}\|^2\,.
\end{equs}
We are left with estimating $\|(-\gensy)^{-1/2}\genap e_{k}\|^2$ from below. Thanks to~\eqref{e:gen}, 
we have an explicit expression for the latter, which reads
\begin{equs}
  \label{e:whichreads}
\|(-\gensy)^{-1/2}\genap e_{k}\|^2&=\frac{2}{(2\pi)^{d}}\lambda_\eps^2\sum_{\ell+m=k} \indN{\ell,m}\frac{(\fw\cdot(\ell+m))^2}{|\ell|^2+|m|^2}\\
&=\frac{2}{(2\pi)^{d}}(\fw\cdot k)^2 \left(\lambda_\eps^2\sum_{\ell+m=k}\frac{\indN{\ell,m}}{|\ell|^2+|m|^2}\right)\,.
\end{equs}
As a consequence, from~\eqref{e:Dpos} and the previous estimates, we obtain 
\begin{equs}
D^n \gtrsim \left(\lambda_\eps^2\sum_{\ell+m=k}\frac{\indN{\ell,m}}{|\ell|^2+|m|^2}\right) -\frac{|k|}{(w\cdot k)^2}\|(-\gensy)^{-\frac12}[\genam\vN_2[-k]-\cD^{n} f_{1}]\|\,. 
\end{equs}
Now, taking the $\eps\to0$ limit at both sides we see that, by~\eqref{e:nH1limit}, the second term converges to $0$, 
while the first goes to a strictly positive constant independent of $n$. 
Thus, the strict positivity of $D_\SHE$ follows by passing to the limit in $n$ so that the proof of the theorem is concluded. 
\end{proof}
}

\subsection{Proof of Proposition~\ref{p:epsLim}}\label{sec:Diff3}

The proof of Proposition~\ref{p:epsLim} is the most delicate part of the present paper. 
Since in $d\geq 3$ we have no analogue of the replacement lemma (and we do not even expect it to hold), we need to resort to 
a different strategy. Our proof is centred around a suitable expansion of the norms of $\vN$ (see e.g.~\eqref{e:LimNorm} below) 
which involves several subsequent applications of the operators $\genap$, $\genam$ interposed 
by inverses of $-\gensy$. In order to see why such an expansion is meaningful, 
it is convenient to write it in terms of operators which are skew-Hermitian and bounded in $\fock$ (uniformly in $\eps$), 
and this is the first thing we will do. 

Let $n\in\N$, $n\geq 2$, be fixed throughout the section. For $i=2,3$, let $\eps>0$ and 
$\cA^\eps_{i,n}$ be given as in \eqref{e:genn}. 
We introduce the operators $T^\eps_{i,n}$ and for $\sigma\in\{+,-\}$, $T^{\eps,\sigma}, T^{\eps,\sigma}_{i,n}$, 
according to
\begin{equs}[eq:papiro]
  T^{\eps,\sigma}&\eqdef (-\cL_0)^{-1/2}(\cA^\eps_{\sigma})(-\cL_0)^{-1/2}\,,\\
  T^\eps_{i,n}&\eqdef (-\cL_0)^{-1/2}(\cA^\eps_{i,n})(-\cL_0)^{-1/2}\,,\\
  T^{\eps,\sigma}_{i,n}&\eqdef (-\cL_0)^{-1/2}(\cA^{\eps,\sigma}_{i,n})(-\cL_0)^{-1/2}\,.
\end{equs}
Note that, by Lemma~\ref{l:GenaBound} and in particular~\eqref{e:Tbound}, 
each of the operators above is  bounded in $\fock$ uniformly in $\eps$ (the bound depends on $n$, 
but this is irrelevant as $n$ is fixed) and, by Lemma~\ref{lem:essid}, it is easy to see that $T^\eps_{i,n}$ 
is skew-Hermitian. 

To have a notation for multiple applications of the operators $T^{\eps,\sigma}$, $\sigma\in\{+,-\}$, 
we set, for $m\ge1$ and $a\ge 0$, $\Pi^{(n)}_{a,m}$ to be the set
of simple random walk paths $p=(p_0,\dots,p_a)$ consisting of $a$ steps, starting at height $1$, i.e. $p_0=1$,
ending at height $m$, i.e. $p_a=m$, that do not reach height $n+1$ and that can have
height $1$ only at the endpoints.  
Given $p\in \Pi^{(n)}_{a,m}$ we let $|p|\eqdef a$ and
$\mathcal T^\eps_p$ be defined according to 
\begin{equ}[e:TTT]
\ST^\eps_p\eqdef  T^{\eps,\sigma_a}T^{\eps,\sigma_{a-1}}\dots T^{\eps,\sigma_1}
\end{equ}
where $\sigma_i=+$ if $p_{i}-p_{i-1}=1$, and $\sigma_i=-$ if $p_{i}-p_{i-1}=-1$, and $\mathcal T^\eps_p\eqdef 1$ if $|p|=0$.
In other words, there is a factor $T^{\eps,+}$
(resp. $T^{\eps,-}$) whenever the path takes a step up
(resp. down)\footnote{For instance, for $a=4$ the path
  $p=(1,2,3,2,3)\in \Pi^{(n)}_{4,3}, n\ge 4$ corresponds to $\mathcal T^\eps_p=T^{\eps,+}T^{\eps,-}T^{\eps,+}T^{\eps,+}$.}.

We enclose the most technical aspect of the proof of Proposition~\ref{p:epsLim} 
in Proposition~\ref{p:limite} below. Before giving the proof of the latter, 
we will show how it implies the former. 

\begin{proposition}
  \label{p:limite} 
For $j_1,j_2\in\Z^d_0$, let $\bj$ be either $j_1$ or $j_{1:2}$, and denote by $e_{\bj}$ either $e_{j_1}$ or 
$e_{j_{1:2}}\eqdef [e_{j_1}\otimes e_{j_2}]_{sym}$,  the symmetric version  of $e_{j_1}\otimes e_{j_2}$, 
where $e_{j_i}$ is the $j_i$-th element of the Fourier basis. 
In the notations introduced above, for every $a_1,a_2\geq0$ and $p^r\in \Pi^{(n)}_{a_r,1}$, $r=1,2$  (so that in particular $a_r\in 2\N$)
the limit\footnote{in \eqref{e:cp}, the overbar denotes complex conjugation}
\begin{equation}
  \label{e:cp}
  \lim_{\eps\to0} \langle \overline {\mathcal T^\eps_{p^1} e_{\bj}},\mathcal T^\eps_{p^2} e_{\bj}\rangle= {c(p^1)} c(p^2)
\left(\frac{\|(-\gensy)^{-\frac12}(-\gensy^\fw) e_{\bj}\|}{|\bj|/\sqrt{2}}\right)^{\sum_{r=1,2}\1_{a_r\ne0}}
\end{equation}
exists, with $c(p)\in\R$ depending only on the path $p$ and such that 
 $c(p)=1$ if $|p|=0$. 

Furthermore, for every $a_r>0$, $m_r\ge 2$, 
and $p^r\in \Pi^{(n)}_{a_r,m_r}$, $r=1,2$, one has
\begin{equation}
\label{e:nullo}  
\lim_{\eps\to0} \langle \overline {\mathcal T^\eps_{p^1} e_\bj},(-\gensy)^{-1}\mathcal T^\eps_{p^2} e_\bj\rangle= 0\,.
\end{equation}
\end{proposition}

\begin{proof}[of Proposition \ref{p:epsLim} given Proposition \ref{p:limite}] 
Let us first introduce a notation thanks to which we can express both the norm in~\eqref{e:nH1limit} 
and the $\fock$-norm of $\vN$. For this, let $\cS$ and $c$ be, respectively, a linear operator and 
a constant given by either $\cS=P_{i-1}\genam$ and $c=1$ or $\cS=(-\gensy)^{1/2}$ and $c=0$, 
with $P_{i-1}$ the orthogonal projector on the $(i-1)$-th chaos, $\Gamma L^2_{i-1}$.
Then, the quantity of interest in \eqref{e:nH1limit}  is 
\begin{equ}
\|(-\gensy)^{-\frac12}[\cS\vN-c\cD^{n-i+2}f_{i-1}]\|\,.
\end{equ}
The rest of the proof will be divided in several steps, the first of which reduces our 
study to the case where the test functions $\phi$ and $\psi$ in the definition of $f_{i-1}$ are 
two given elements of the Fourier basis. 
\medskip

\noindent\textit{Step 1: Reduction to Fourier basis.} 
For this step, we focus on the case $i=3$ since $i=2$ is easier and can be argued similarly. 
For $j_1,j_2\in\Z^d_0$, let $\vN[-j_{1:2}]$ be the solution to~\eqref{e:nonlinGenEq} in which we take $f_2$ to be 
$e_{j_{1:2}}=[
e_{j_1}\otimes e_{j_2}]_{sym}$. By linearity, we clearly have
$\vN=\sum_{j_1,j_2}\hat\phi(j_1)\hat\psi(j_2)\vN[-j_{1:2}]$. 
The definition of the norm in $\fock$ immediately gives 
\begin{equ}
\|(-\gensy)^{-\frac12}[\cS\vN-c\cD^{n-1}f_2]\|^2=\sum_{k_1,k_2}\frac2{ |k_{1:2}|^2}\Big|\langle \cS\vN-c\cD^{n-1}f_2, \frac{e_{-k_{1:2}}}{\|e_{k_{1:2}}\|}\rangle\Big|^2\,.
\end{equ}
Now, for $\alpha\geq 0$, we can apply Cauchy-Schwarz to get
\begin{equs}
&\Big|\langle \cS\vN-c\cD^{n-1}f_2, e_{-k_{1:2}}\rangle\Big|^2\\
&=\Big|\sum_{j_1,j_2}\hat\phi(j_1)\hat\psi(j_2)\langle \cS\vN[-j_{1:2}]-c\cD^{n-1}e_{j_{1:2}}, e_{-k_{1:2}}\rangle\Big|^2\\
&\leq \sum_{j_1,j_2}|j_{1:2}|^{2\alpha}|\hat\phi(j_1)|^2|\hat\psi(j_2)|^2\sum_{j_1,j_2}\frac{1}{|j_{1:2}|^{2\alpha}}\Big|\langle \cS\vN[-j_{1:2}]-c\cD^{n-1}e_{j_{1:2}}, e_{-k_{1:2}}\rangle\Big|^2\,,
\end{equs}
so that we conclude
\begin{equs}
&\|(-\gensy)^{-\frac12}[\cS\vN-c\cD^{n-1}f_2]\|^2\\
&\lesssim \|(-\gensy)^{\alpha}f_2\|^2\sum_{j_{1:2}}\frac{1}{|j_{1:2}|^{2\alpha}}\sum_{k_1,k_2}\frac{2}{ |k_{1:2}|^2}\Big|\langle \cS\vN[-j_{1:2}]-c\cD^{n-1}e_{j_{1:2}}, e_{-k_{1:2}}\rangle\Big|^2\\
&= \|(-\gensy)^{\alpha}f_2\|^2\sum_{j_{1:2}}\frac{1}{|j_{1:2}|^{2\alpha}}\|(-\gensy)^{-\frac12}[\cS\vN[-j_{1:2}]-c\cD^{n-1}e_{j_{1:2}}]\|^2\,.
\end{equs}
Now, the norm inside the sum is bounded by (a constant times) $|j_{1:2}|^2$. 
Indeed, if $\cS=P_2\genam$ and $c=1$, thanks to~\eqref{e:Abound}, 
the apriori estimates in Proposition~\ref{P:AprioriGeneratorEq}
and the definition of $\cD^{n-1}$ in Proposition~\ref{p:epsLim}, we get 
\begin{equs}
\|(-\gensy)^{-\frac12}&[\genam\vN_3[-j_{1:2}]-\cD^{n-1}e_{j_{1:2}}]\|^2\\
&\lesssim \|(-\gensy)^{1/2}\vN_3[-j_{1:2}]\|^2 + \|(-\gensy)^{-\frac12}\cD^{n-1}e_{j_{1:2}}\|^2\\
&\lesssim \|(-\gensy)^{-1/2}\genap e_{j_{1:2}}\|^2+\|(-\gensy)^{1/2}e_{j_{1:2}}\|^2\\
&\lesssim \|(-\gensy)^{1/2}e_{j_{1:2}}\|^2\lesssim |j_{1:2}|^2\,.
\end{equs}
In instead, $\cS=(-\gensy)^{1/2}$ and $c=0$, then 
\begin{equ}
\|\vN[-j_{1:2}]\|^2\leq \|(-\gensy)^{1/2}\vN[-j_{1:2}]\|^2\leq \|(-\gensy)^{1/2}e_{j_{1:2}}\|^2\lesssim |j_{1:2}|^2\,.
\end{equ}
As a consequence, upon choosing $\alpha$ big enough, and we are allowed to since $f_2$ is smooth, 
by dominated convergence the statement follows if we prove that for all given $j_1, j_2\in\Z^d_0$, 
the norm in the sum vanishes. 
Now, by opening up the squared norm and using the definition of $\cD^{n-1}$ we see that this  
is equivalent to showing that 
\begin{equ}[e:Norm]
\lim_{\eps\to 0} \|(-\gensy)^{-\frac12}\cS\vN[-j_{1:2}]\|=c D^{n-1} \|\gensy^{-\frac12}\gensy^\fw e_{j_{1:2}}\|
\end{equ}
and, if $c\neq 0$, that
\begin{equ}[e:ScalarProd]
\lim_{\eps\to 0} \langle e_{-j_{1:2}}, (-\gensy)^{-\frac12}\cS\vN[-j_{1:2}]\rangle=D^{n-1} \|\gensy^{-\frac12}\gensy^\fw e_{j_{1:2}}\|\,,
\end{equ}
and similarly for $i=2$, in which case instead of $\vN[-j_{1:2}]$, $\cD^{n-1}$ and $e_{j_{1:2}}$, we have 
$\vN[-j_1]$, $\cD^n$ and $e_{j_1}$ respectively. 

From now on, we adopt the notations of Proposition~\ref{p:limite} and let $\bj$ be either $j_1$ if $i=2$ or $j_{1:2}$ if $i=3$. 
\medskip

\noindent\textit{Step 2: Analysis of $\cS\vN[-\bj]$.} Recall the definition of the operators 
$T^\eps_{i,n}$ and $T^{\eps,\sigma}$, $\sigma\in\{+,-\}$ in~\eqref{eq:papiro}. 
Thanks to \eqref{e:nonlinGenEq},
\begin{equs}[e:ven-k]
\vN[-\bj]&=(-\genn)^{-1}\genap e_{\bj}=(-\gensy)^{-1/2}(I-T^\eps_{i,n})^{-1}T^{\eps,+}(-\gensy)^{1/2}e_{\bj}\\
&=\frac{|\bj|}{\sqrt{2}}(-\gensy)^{-1/2}(I-T^\eps_{i,n})^{-1}T^{\eps,+}e_{\bj}\\
&=\frac{|\bj|}{\sqrt{2}}(-\gensy)^{-1/2}\Big(\int_0^\infty e^{-s} e^{s T^\eps_{i,n}}\dd s\Big)T^{\eps,+}e_{\bj} \,. 
\end{equs}
Hence, the scalar product in~\eqref{e:ScalarProd} satisfies
\begin{equs}
\frac{\sqrt{2}}{|\bj|}\langle e_{-\bj}, &(-\gensy)^{-\frac12}\cS\vN[-\bj]\rangle\\
&=\langle e_{-\bj}, (-\gensy)^{-\frac12}\cS(-\gensy)^{-1/2}\Big(\int_0^\infty e^{-s} e^{s T^\eps_{i,n}}\dd s\Big)T^{\eps,+}e_{\bj}\rangle\\
&=\int_0^\infty e^{-s}\langle e_{-\bj}, (-\gensy)^{-\frac12}\cS(-\gensy)^{-1/2} e^{s T^\eps_{i,n}}T^{\eps,+}e_{\bj}\rangle\dd s
\end{equs}
where in the last step we applied Fubini, which is allowed as the operator 
$e^{s T^\eps_{2,n}}$ has norm at most $1$, since the spectrum of $T^\eps_{i,n}$ is purely imaginary. 
For the same reason, once we take the $\eps$-limit at both sides we can apply dominated convergence 
and pass the limit inside the integral in $s$, i.e. 
\begin{equs}\label{e:LimScalarProd}
&\lim_{\eps\to 0} \frac{\sqrt{2}}{|\bj|}\langle e_{-\bj}, (-\gensy)^{-\frac12}\cS\vN[-\bj]\rangle\\
&=\int_0^\infty e^{-s}\lim_{\eps\to 0}\langle e_{-\bj}, (-\gensy)^{-\frac12}\cS(-\gensy)^{-1/2} e^{s T^\eps_{i,n}}T^{\eps,+}e_{\bj}\rangle\dd s\\
&=\int_0^\infty e^{-s}\lim_{\eps\to 0}\langle e_{-\bj}, (-\gensy)^{-\frac12}\cS(-\gensy)^{-1/2} \sum_{b\ge0}\frac{s^b}{b!} (T^\eps_{i,n})^bT^{\eps,+}e_{\bj}\rangle\dd s\\
&=\int_0^\infty e^{-s} \sum_{b\ge0}\frac{s^b}{b!}\lim_{\eps\to 0} \langle e_{-\bj}, (-\gensy)^{-\frac12}\cS(-\gensy)^{-\frac12} (T^\eps_{i,n})^{b}T^{\eps,+}e_{\bj}\rangle\dd s 
\end{equs}
where in the last step we could exchange the limit with both the summation over $b$ and the scalar product 
because, for fixed $n,s$, the norm of $s T^\eps_n$ is uniformly bounded (recall that after the first, in the 
subsequent steps $s$ can be treated as fixed). 

For the norm in~\eqref{e:Norm}, we can follow the exact same arguments so that we are led to 
\begin{equs}[e:LimNorm]
&\lim_{\eps\to 0}\frac{2}{|\bj|^2}\|(-\gensy)^{-\frac12}\cS\vN[-\bj]\|^2\\
&=\lim_{\eps\to 0}\frac{2}{|\bj|^2}\langle \overline{(-\gensy)^{-\frac12}\cS\vN[-\bj]}, (-\gensy)^{-\frac12}\cS\vN[-\bj]\rangle\\
&=\int_0^\infty\int_0^\infty \dd s_1\dd s_2 \,e^{-s_1-s_2} \sum_{b_1,b_2\ge0}\frac{s_1^{b_1}s_2^{b_2}}{b_1! b_2!}\times\\
&\times \lim_{\eps\to 0}\langle \overline{(-\gensy)^{-\frac12}\cS(-\gensy)^{-\frac12} (T^\eps_{i,n})^{b_1}T^{\eps,+}e_{\bj}}, (-\gensy)^{-\frac12}\cS(-\gensy)^{-\frac12} (T^\eps_{i,n})^{b_2}T^{\eps,+}e_{\bj}\rangle\,.
\end{equs}

The previous arguments guarantee that 
we are left to determine, for every $b_1,b_2\ge0$ and $\theta\in\{0,1\}$, the limit 
\begin{equ}[e:AtLast]
\lim_{\eps\to0}\langle \overline{[(-\gensy)^{-\frac12}\cS(-\gensy)^{-\frac12} (T^\eps_{i,n})^{b_1}T^{\eps,+}]^\theta e_{\bj}}, (-\gensy)^{-\frac12}\cS(-\gensy)^{-\frac12} (T^\eps_{i,n})^{b_2}T^{\eps,+}e_\bj\rangle
\end{equ}
and to do so we will separately consider the case in which $\cS=P_{i-1}\genam$ and $c=1$, or
$\cS=(-\gensy)^{\frac12}$ and $c=0$, starting with the former. 
\medskip

\noindent\textit{Step 3: The case $\cS=P_{i-1}\genam$ and $c=1$.} With this choice of $\cS$ and $c$ we need to 
prove both~\eqref{e:Norm} and~\eqref{e:ScalarProd}, which respectively correspond to taking $\theta=1$ and 
$\theta=0$ in~\eqref{e:AtLast}. Furthermore, by~\eqref{eq:papiro}, we have 
\begin{equ}
(-\gensy)^{-\frac12}\cS(-\gensy)^{-\frac12}=(-\gensy)^{-\frac12}P_{i-1}\genam(-\gensy)^{-\frac12}=P_{i-1}  T^{\eps,-}
\end{equ}
so that the scalar product in~\eqref{e:AtLast} becomes
\begin{equ}
\langle \overline{[P_{i-1} T^{\eps,-}(T^\eps_{i,n})^{b_1}T^{\eps,+}]^\theta e_\bj}, P_{i-1} T^{\eps,-} (T^\eps_{i,n})^{b_2}T^{\eps,+}e_\bj\rangle\,.
\end{equ}
Let us begin with a few remarks. First of all, because of the projection $P_{i-1}$ and the fact that $e_\bj\in\fock_{i-1}$, 
$b_2$ must necessarily be even, hence we set $a_2\in\N\setminus\{0\}$ 
to be such that $b_2=2a_2-2$. The same holds for $b_1$ if $\theta=1$ (and we write $b_1=2a_1-2, a_1\in\N\setminus\{0\}$) while if $\theta=0$, $b_1$ plays no role and for convenience we will set $a_1\equiv 0$.
Next, we expand 
\begin{equs}\label{e:expansion}
\langle &\overline{[P_{i-1}  T^{\eps,-}(T^\eps_{i,n})^{2a_1-2}T^{\eps,+}]^\theta e_\bj}, P_{i-1} T^{\eps,-} (T^\eps_{i,n})^{2a_2-2}T^{\eps,+}e_\bj\rangle\\
&=\langle \overline{[P_{i-1}  T^{\eps,-}(T^{\eps,+}_{i,n}+T^{\eps,-}_{i,n})^{2a_1-2}T^{\eps,+}]^\theta e_\bj}, \\
&\qquad\qquad\qquad\qquad\qquad\qquad\qquad\qquad P_{i-1} T^{\eps,-} (T^{\eps,+}_{i,n}+T^{\eps,-}_{i,n})^{2a_2-2}T^{\eps,+}e_\bj\rangle\\
&=
\sum_{\substack{p_r\in\widetilde\Pi_{2a_r,1}\\r=1,2}}\langle \overline{[T^{\eps,\sigma_{2a_1}}T^{\eps,\sigma_{2a_1-1}}_{i,n}\dots T^{\eps,\sigma_2}_{i,n} T^{\eps,\sigma_1}]^\theta e_\bj}, \\
&\qquad\qquad\qquad\qquad\qquad\qquad\qquad\qquad T^{\eps,\sigma_{2a_2}}T^{\eps,\sigma_{2a_2-1}}_{i,n}\dots T^{\eps,\sigma_2}_{i,n} T^{\eps,\sigma_1} e_\bj\rangle
\end{equs}
where we define $\widetilde\Pi_{2a,1}$ to be the set of all walks $p$ which start at time $1$ from $1$, i.e. $p_0=1$, 
end at time $2a$ at $1$, i.e. $p_{2a}=1$, have increments of size $1$ and are such that 
$p_1-p_0=1=p_{2a-1}-p_{2a}$. If either of the previous 
conditions were violated the projection onto $P_{i-1}$ would set the scalar product to $0$.
Moreover, the last condition on 
the $p$'s ensures that $\sigma_1=+$ and $\sigma_{2a}=-$ (so that the $T^{\eps,-}$ and $T^{\eps,+}$ in the line 
above are present). 

Because of the
projector $P^n_i$ entering the definition of $T^{\eps,\sigma}_{i,n}$,
if a path $p_r$ reaches height $n+3-i$ or height $i-1$ (except at the
endpoints), then the corresponding scalar product is annihilated. 
Hence, we can replace the index set of the sum at the right hand side of~\eqref{e:expansion} to 
$\Pi^{(n+2-i)}_{2a_r,1}$, and get
\begin{equs}
\,&\langle \overline{[P_{i-1}  T^{\eps,-}(T^\eps_{i,n})^{2a_1-2}T^{\eps,+}]^\theta e_\bj}, P_{i-1} T^{\eps,-} (T^\eps_{i,n})^{2a_2-2}T^{\eps,+}e_\bj\rangle\\
& =
\sum_{\substack{p_r\in\Pi^{(n+2-i)}_{2a_r,1}\\ r=1,2}}\langle \overline{[T^{\eps,\sigma_{2a_1}}T^{\eps,\sigma_{2a_1-1}}\dots T^{\eps,\sigma_2} T^{\eps,\sigma_1}]^\theta e_\bj}, 
T^{\eps,\sigma_{2a_2}}\dots T^{\eps,\sigma_2} T^{\eps,\sigma_1} e_\bj\rangle\\
&=
\sum_{\substack{p_r\in\Pi^{(n+2-i)}_{2a_r,1}\\ r=1,2}}\langle \overline{\ST^\eps_{p_1} e_\bj}, \ST^\eps_{p_2} e_\bj\rangle\\
\end{equs}
where in the first step we replaced all instances of $T^{\eps,\sigma}_{i,n}$ with $T^{\eps,\sigma}$, since 
for $p_r\in\Pi^{(n+2-i)}_{2a_r,1}$ the projection $P^n_i$ acts as the identity operator,
and in the second 
we used the convention that  $a_1=0$ if $\theta=0$, in which case $p_1$ has length zero and $\ST^\eps_{p_1} =1$.

At last, getting back to~\eqref{e:LimScalarProd}, and applying~\eqref{e:cp} in Proposition~\ref{p:limite}, 
we obtain
\begin{equs}
\lim_{\eps\to 0} &\langle e_{-\bj}, (-\gensy)^{-\frac12}\cS\vN[-\bj]\rangle \\
&= \frac{|\bj|}{\sqrt{2}}\int_0^\infty e^{-s} \sum_{a\geq 1}\frac{s^{2a}}{(2a)!}\lim_{\eps\to 0} \sum_{p\in\Pi^{(n+2-i)}_{2a,1}}\langle  e_{-\bj}, \ST^\eps_{p} e_\bj\rangle\dd s \\
&= \frac{|\bj|}{\sqrt{2}}\int_0^\infty e^{-s} \sum_{a\geq 1}\frac{s^{2a}}{(2a)!} \sum_{p\in\Pi^{(n+2-i)}_{2a,1}} \lim_{\eps\to 0}\langle  e_{-\bj}, \ST^\eps_{p} e_\bj\rangle\dd s \\
&=\|(-\gensy)^{-\frac12}(-\gensy^\fw) e_{\bj}\|\int_0^\infty e^{-s} \sum_{a\geq 1}\frac{s^{2a}}{(2a)!}\sum_{p\in\Pi^{(n+2-i)}_{2a,1}} c(p)\dd s\\
&=:\|(-\gensy)^{-\frac12}(-\gensy^\fw) e_{\bj}\|D^{n+2-i}\label{e:defDn}
\end{equs}
where the last constant $D^{n+2-i}$ is uniquely defined by the expression before, so that~\eqref{e:ScalarProd} 
follows at once. 
Similarly,~\eqref{e:Norm} holds since, from~\eqref{e:LimNorm} and by~\eqref{e:cp}, we get
\begin{equs}
&\lim_{\eps\to 0}\|(-\gensy)^{-\frac12}\cS\vN[-\bj]\|^2\\
&=\frac{|\bj|^2}{2}\int_0^\infty\int_0^\infty \dd s_1\dd s_2 \,e^{-s_1-s_2} \sum_{\substack{a_1\geq 1\\a_2\geq 1}}\frac{s_1^{2a_1}s_2^{2a_2}}{(2a_1)! (2a_2)!} \sum_{\substack{p_r\in\Pi^{(n+2-i)}_{2a_r,1}\\ r=1,2}}\lim_{\eps\to 0}\langle \overline{\ST^\eps_{p_1} e_\bj}, \ST^\eps_{p_2} e_\bj\rangle\\
&=\|(-\gensy)^{-\frac12}(-\gensy^\fw) e_{\bj}\|^2\Big(\int_0^\infty e^{-s} \sum_{a\geq 1}\frac{s^{2a}}{(2a)!}\sum_{p\in\Pi^{(n+2-i)}_{2a,1}} c(p)\dd s\Big)^2\\
&=\|(-\gensy)^{-\frac12}(-\gensy^\fw) e_{\bj}\|^2(D^{n+2-i})^2
\end{equs}
for $D^{n+2-i}$ be given as above. Therefore, the proof of~\eqref{e:Norm} is concluded and so is that of~\eqref{e:nH1limit}. 
\medskip

\noindent\textit{Step 4: The case $\cS=(-\gensy)^{\frac12}$ and $c=0$.} This choice of $\cS$ turns 
the scalar product in~\eqref{e:AtLast} into 
\begin{equ}[e:AtLastNew]
\langle \overline{(T^\eps_{i,n})^{b_1}T^{\eps,+} e_{\bj}}, (-\gensy)^{-1} (T^\eps_{i,n})^{b_2}T^{\eps,+}e_\bj\rangle\,.
\end{equ}
By expanding and arguing as in the previous step, we deduce that 
\begin{equs}
&\langle \overline{(T^\eps_{i,n})^{b_1}T^{\eps,+} e_{\bj}}, (-\gensy)^{-1} (T^\eps_{i,n})^{b_2}T^{\eps,+}e_\bj\rangle\\
&=\langle \overline{(T^{\eps,+}_{i,n}+T^{\eps,-}_{i,n})^{b_1}T^{\eps,+} e_{\bj}}, (-\gensy)^{-1} (T^{\eps,+}_{i,n}+T^{\eps,-}_{i,n})^{b_2}T^{\eps,+}e_\bj\rangle\\
&=\sum_{r=1}^2\sum_{m_r=i}^{b_r+i}\sum_{p_r\in\widetilde\Pi_{b_r,m_r}}\langle \overline{T^{\eps,\sigma_{b_1}}_{i,n}\dots T^{\eps,\sigma_2}_{i,n} T^{\eps,\sigma_1} e_\bj}, (-\gensy)^{-1}T^{\eps,\sigma_{b_2}}_{i,n}... T^{\eps,\sigma_2}_{i,n} T^{\eps,\sigma_1} e_\bj\rangle\\
&=\sum_{r=1}^2\sum_{m_r=i}^{b_r+i}\sum_{p_r\in\Pi^{(n+2-i)}_{b_r,m_r}}\langle \overline{T^{\eps,\sigma_{b_1}}_{i,n}\dots T^{\eps,\sigma_2}_{i,n} T^{\eps,\sigma_1} e_\bj}, (-\gensy)^{-1}T^{\eps,\sigma_{b_2}}_{i,n}... T^{\eps,\sigma_2}_{i,n} T^{\eps,\sigma_1} e_\bj\rangle\\
&=\sum_{r=1}^2\sum_{m_r=i}^{b_r+i}\sum_{p_r\in\Pi^{(n+2-i)}_{b_r,m_r}}\langle \overline{T^{\eps,\sigma_{b_1}}\dots T^{\eps,\sigma_2} T^{\eps,\sigma_1} e_\bj}, (-\gensy)^{-1}T^{\eps,\sigma_{b_2}}... T^{\eps,\sigma_2} T^{\eps,\sigma_1} e_\bj\rangle\\
&=\sum_{r=1}^2\sum_{m_r=i}^{b_r+i}\sum_{p_r\in\Pi^{(n+2-i)}_{b_r,m_r}}\langle \overline{\ST^\eps_{p^1} e_\bj}, (-\gensy)^{-1}\ST^\eps_{p^2} e_\bj\rangle
\end{equs}
where $\widetilde\Pi_{b_r,m_r}$, $r=1,2$, 
is the set of all walks $p_r$ which start at time $1$ from $1$, i.e. $p_0=1$, 
end at time $b_r$ at $m_r$, i.e. $p_{b_r}=b_r$, have increments of size $1$ and are such that 
$p_1-p_0=1$. In the third step, because of the projector $P^n_i$ in the definition of $T^{\eps,\sigma}_{i,n}$, 
we replaced the sum over $\widetilde\Pi_{b_r,m_r}$ with that over $\Pi^{(n+2-i)}_{b_r,m_r}$ and 
consequently removed the projection. 

Now, since the scalar product at the right hand side converges to $0$ by~\eqref{e:nullo}, from~\eqref{e:LimNorm} 
we deduce that $\|\vN\|^2$ converges to $0$ 
so that the proof of Proposition~\ref{p:epsLim} is concluded. 
\end{proof}

We now turn to the proof of Proposition~\ref{p:limite}, for which we will need some notation and 
a few preliminary results. 
First of all, the action of the operators $\genap,\genam$ of \eqref{e:gen} on
$\Gamma L^2_n$ can be rewritten as follows:
\begin{equs}[e:GenaOpenUp]
  \cF(\genap f) (k_{1:n+1})&=\frac{1}{n+1}\sum_{1\le i<i'\le n+1} \genap[(i,i')]f(k_{1:n+1})\\
  \cF(\genam  f) (k_{1:n-1})&=n\sum_{q=1}^{n-1} \genam[q]f(k_{1:n-1})\,,
\end{equs}
where, for $1\le i<i'\le n+1$ and $1\le q\le n-1$, 
\begin{equs}[e:GenaNonSym]
\genap[(i,i')]f(k_{1:n+1})&=-\frac{2\iota \eps^{\frac{d}{2}-1}}{(2\pi)^{d/2}} 
  [\fw\cdot (k_i+k_{i'})]\indN{k_i,k_{i'}}\hat f(k_i+k_{i'},k_{\{1:n+1\}\setminus\{i,i'\}})\\
\genam[q]f(k_{1:n-1})&=\frac{2\iota}{(2\pi)^{d/2}} \eps^{\frac{d}{2}-1}(\fw\cdot k_q) \sum_{\ell+m=k_q}\indN{\ell,m}\hat f(\ell,m,k_{\{1:n-1\}\setminus\{q\}})\,.
\end{equs}
Let $T^{\eps,+}[(i,i')]$ and $T^{\eps,-}[q]$ be defined as
in~\eqref{eq:papiro}, but with $\cA^{\eps,+}$, $\cA^{\eps,-}$ 
replaced by $\genap[ (i,i')]$, $\genam[ q]$.

The advantage of this decomposition is that  the operators in~\eqref{e:GenaNonSym} 
 can be extended to functions which are not necessarily symmetric, i.e. they 
can be viewed as operators on $\bigoplus_n L^2_n\eqdef \bigoplus_n L^2((\T^d)^n)$. 

Before detailing their properties, we also introduce operators $\cA^{\eps,\delta}_{\pm}$, $\delta\geq 0$, 
which will allow us to derive a {\it uniform integrability}-type condition. This will be the essential 
tool to prove the limit in~\eqref{e:cp}. 
For $f\in L^2_n$,  $1\le i<i'\le n+1$, $1\le q\le n-1$  and $\delta\geq 0$, set 
\begin{equs}[e:GenaNonSymDelta]
\cA^{\eps,\delta}_{+}[(i,i')]f(k_{1:n+1})&=\eps^{\frac{d}{2}} \Big|\frac{2\eps^{-1}[\fw\cdot (k_i+k_{i'})]}{(2\pi)^{d/2}}
  \Big|^{1+\delta}\indN{k_i,k_{i'}}\hat f(k_i+k_{i'},k_{\{1:n+1\}\setminus\{i,i'\}})\\
\cA^{\eps,\delta}_{-}[q]f(k_{1:n-1})&=\eps^{\frac{d}{2}}\Big|\frac{2\eps^{-1} \fw\cdot k_q}{(2\pi)^{d/2}}\Big|^{1+\delta}\sum_{\ell+m=k_q}\indN{\ell,m}\hat f(\ell,m,k_{\{1:n-1\}\setminus\{q\}})\,,
\end{equs}
and 
\begin{equs}[e:TdeltaNonSym]
T^{\eps,\delta,+}[ (i,i')]\eqdef(-\gensy)^{-\frac{1+\delta}{2}}\cA^{\eps,\delta}_{+}[(i,i')](-\gensy)^{-\frac{1+\delta}{2}}
\end{equs}
and similarly for $T^{\eps,\delta,-}[ q]$. 
In the following lemma, we collect estimates on the operators introduced above. 

\begin{lemma}\label{l:AnonsymProp}
  Let $\eps>0$ and $n\in\N$. For every $1\le i<i'\le n+1$ and $1\le q\le n-1$, the linear 
  operators $\genap[(i,i')]$ and $\genam[q]$
  in~\eqref{e:GenaNonSym} map $L^2_n$ to $L^2_{n+1}$ and $L^2_{n-1}$
  respectively, and commute with the momentum operator defined in
  Lemma~\ref{lem:essid}. 
 Moreover, there exists a constant $C>0$ such that for every $f\in L^2_n$, 
\begin{equ}[e:TboundNew]
\|T^{\eps,+}[(i,i')] f\|_{L^2_{n+1}}\vee \|T^{\eps,-}[q] f\|_{L^2_{n-1}}\leq C \|f\|_{L^2_n}
\end{equ}
and for every $\delta\in[0,(d-2)/2)$ we also have 
\begin{equ}[e:TboundDelta]
\|T^{\eps,\delta,+}[(i,i')] f\|_{L^2_{n+1}}\vee \|T^{\eps,\delta,-}[q] f\|_{L^2_{n-1}}\leq C(\delta) \|f\|_{L^2_n}\,.
\end{equ}
\end{lemma}
\begin{proof}
Commutativity with the momentum operator is immediately checked from the definition of the operators.
Concerning~\eqref{e:TboundNew}, it follows from~\eqref{e:TboundDelta} upon noting 
that for every $f\in L^2_n$
\begin{equ}
\|T^{\eps,+}[(i,i')] f\|_{L^2_{n+1}}\leq \|T^{\eps,\delta,+}[(i,i')] f\|_{L^2_{n+1}}
\end{equ}
with $\delta=0$, and similarly for $T^{\eps,-}[q]$. 
Let $\delta\in[0,(d-2)/2)$ and $f\in L^2_n$. Then, with a simple change of variables, we have
\begin{equs}
\|T^{\eps,\delta,+}[(i,i')] &f\|_{L^2_{n+1}}^2=\|(-\gensy)^{-\frac{1+\delta}{2}}\cA^{\eps,\delta}_+[(i,i')] (-\gensy)^{-\frac{1+\delta}{2}}f\|^2_{L^2_{n+1}}\\
&\lesssim\sum_{k_{1:n}} \Big(\frac{|\fw\cdot k_1|}{|k_{1:n}|}\Big)^{2+2\delta}|\hat f(k_{1:n})|^2\eps^{d-2(1+\delta)} \sum_{\ell+m=k_1}\frac{\indN{\ell,m}}{(|\ell|^2+|m|^2)^{1+\delta}}\\
&\lesssim \|f\|_{L^2_n}^2 \eps^{d-2(1+\delta)} \sum_{|\ell|\leq \eps^{-1}}\frac{1}{|\ell|^{2+2\delta}}\,.
\end{equs}
Since the sum above equals
 \begin{equ}[eq:thesum]
    \eps^{d-2(1+\delta)} \sum_{|\ell|\leq \eps^{-1}}\frac{1}{|\ell|^{2+2\delta}}=\eps^{d}\sum_{|\eps\ell|\leq 1} \frac{1}{|\eps\ell|^{2+2\delta}}\lesssim_d  \int_{|x|\leq 1} \frac{\dd x}{|x|^{2+2\delta}}
  \end{equ}
and the right hand side is bounded as long as $\delta<(d-2)/2$,~\eqref{e:TboundDelta} follows. 
Concerning $T^{\eps,\delta,-}[q]$, we apply Cauchy-Schwarz so to obtain
\begin{equs}
\|&T^{\eps,\delta,-}[q] f\|_{L^2_{n-1}}=\|(-\gensy)^{-\frac{1+\delta}{2}}A^{\eps,\delta}_-[q] f(-\gensy)^{-\frac{1+\delta}{2}}\|_{L^2_{n-1}}\\
&\lesssim \sum_{k_{1:n-1}}\Big(\frac{|\fw\cdot k_q|}{|k_{1:n-1}|}\Big)^{2+2\delta}\times\\
&\qquad\times\Big|\eps^{\frac{d}{2}-(1+\delta)}\sum_{\ell+m=k_q}\frac{\indN{\ell,m}}{(|\ell|^2+|m|^2+|k_{1:n-1\setminus\{q\}}|^2)^{\frac{1+\delta}{2}}}\hat f(\ell,m, k_{1:n-1\setminus\{q\}})\Big|^2\\
&\leq \sum_{k_{1:n}} |\hat f(k_{1:n})|^2 \eps^{d-(2+2\delta)}\sum_{|\ell|\leq \eps^{-1}}\frac{1}{|\ell|^{2+2\delta}}\label{e:CSchepalle}
\end{equs}
and once again the inner sum is bounded uniformly over $k_q$ thanks to~\eqref{eq:thesum}, so that 
the proof of~\eqref{e:TboundDelta} is complete. 
\end{proof}

In the following lemma, we show that the main contribution to the norm
of $T^{\eps,+}[(i,i')] f$ and $T^{\eps,-}[q] f$ comes from Fourier
modes whose size diverges with $\eps$.
 
\begin{lemma}\label{l:HighModes}
In the setting of Lemma~\ref{l:AnonsymProp}, let $\gamma\in(0,1)$ and, for $f\in L^2_n$, define the operators 
\begin{equs}[e:Ttilde]
\widetilde T^{\eps,+}[(i,i')] f(k_{1:n+1})&\eqdef \1_{|k_i|\wedge|k_{i'}|>\eps^{-\gamma}}T^{\eps,+}[(i,i')] f(k_{1:n+1})\\
\widetilde T^{\eps,-}[q] f(k_{1:n-1})&\eqdef T^{\eps,-}[q] f^{>}(k_{1:n-1})
\end{equs}
where $\hat f^{>}(k_{1:n})\eqdef \1_{|k_1|\wedge|k_2|>\eps^{-\gamma}}\hat f(k_{1:n})$. 
Then, there exists a constant $C>0$ such that 
\begin{equ}[e:ApproxLarge]
\|(T^{\eps,+}[(i,i')]-\widetilde T^{\eps,+}[(i,i')])f\|_{L^2_{n+1}}\leq C\eps^{(1-\gamma)(d-2)}\|f\|_{L^2_n}
\end{equ}
 and the same bound holds for $\|(T^{\eps,-}[q]-\widetilde T^{\eps,-}[q])f\|_{L^2_{n-1}}$. 
\end{lemma}
\begin{proof}
By the definition of the operator $\cA^\eps_+[(i,i')]$ in~\eqref{e:GenaNonSym} 
and that of $\widetilde T^{\eps,+}[(i,i')]$ in~\eqref{e:Ttilde}, we have 
\begin{equs}
(T^{\eps,+}[(i,i')]-&\widetilde T^{\eps,+}[(i,i')])f(k_{1:n+1})=\1_{|k_i|\wedge|k_{i'}|\leq\eps^{-\gamma}}T^{\eps,+}[(i,i')] f(k_{1:n+1})\,.
\end{equs}
Hence, 
\begin{equs}
\|&(T^{\eps,+}[(i,i')]-\widetilde T^{\eps,+}[(i,i')])f\|_{L^2_{n+1}}^2\\
&\lesssim \sum_{k_{1:n+1}}\frac{\eps^{d-2}}{|k_{1:n+1}|^2}\frac{[\fw\cdot (k_i+k_{i'})]^2 \1_{|k_i|\wedge|k_{i'}|\leq\eps^{-\gamma}}}{|k_i+k_{i'}|^2+|k_{1:n+1\setminus\{i,i'\}}|^2}|\hat f(k_i+k_{i'}, k_{1:n+1\setminus\{i,i'\}})|^2\\
&\leq \sum_{k_{1:n+1}}\frac{\eps^{d-2}}{|k_{1:n+1}|^2}\1_{|k_i|\wedge|k_{i'}|\leq\eps^{-\gamma}}|\hat f(k_i+k_{i'}, k_{1:n+1\setminus\{i,i'\}})|^2\\
&=\sum_{k_{1:n}}|\hat f(k_{1:n})|^2\sum_{\ell+m=k_1}\eps^{d-2}\frac{\1_{|\ell|\wedge|m|\leq\eps^{-\gamma}}}{|\ell|^2+|m|^2+|k_{1:n+1\setminus\{i,i'\}}|^2}
\end{equs}
where in the last step we simply relabelled all the variables in the sum. 
Now, the inner sum can be bounded as
\begin{equs}[e:TruncSum]
\sum_{\ell+m=k_1}&\eps^{d-2}\frac{\1_{|\ell|\wedge|m|\leq\eps^{-\gamma}}}{|\ell|^2+|m|^2+|k_{1:n+1\setminus\{i,i'\}}|^2}\leq \sum_{\ell+m=k_1}\eps^{d-2}\frac{\1_{|\ell|\wedge|m|\leq\eps^{-\gamma}}}{(|\ell|\wedge|m|)^2}\\
&\lesssim \sum_{|\ell|\leq \eps^{-\gamma}}\eps^d\frac{1}{|\eps\ell|^2}\lesssim \int_{|x|\leq \eps^{1-\gamma}}\frac{\dd x}{|x|^2}
\lesssim\int_0^{\eps^{1-\gamma}}r^{d-3}\dd r \leq \eps^{(1-\gamma)(d-2)}
\end{equs}
from which~\eqref{e:ApproxLarge} follows. 
The analogous statement for $T^{\eps,-}[q]-\widetilde T^{\eps,-}[q]$ can be obtained 
as in the previous lemma, i.e. applying Cauchy-Schwarz as in~\eqref{e:CSchepalle} and then 
controlling the resulting sum with~\eqref{e:TruncSum}. 
\end{proof}

Before proceeding, we want to rewrite~\eqref{e:TTT} in terms of the operators 
we just introduced. Let $a\geq 0$, $m\geq 1$, $p\in \Pi^{(n)}_{a,m}$, $\bj,e_\bj$ be as above. 
For $f\in L^2_i$, let $\kappa(f)=i$ denote the eigenvalue of the number operator, 
so that $\kappa\eqdef\kappa(e_\bj)=1$ if $\bj=j_1$ and $\kappa=2$ if $\bj=j_{1:2}$. 
Replacing each of the $T^{\eps,+}$, $T^{\eps,-}$ with the sum of the $T^{\eps,+}[(i,i')]$, $T^{\eps,+}[q]$ 
(as in~\eqref{e:GenaOpenUp}), we obtain 
\begin{equ}[e:NewRep]
\ST^\eps_p e_\bj=\frac{1}{C_{m,\kappa}}\sum_{g\in \cG^\kappa[p]}\ST^\eps_p[g] e_\bj\eqdef \frac{1}{C_{m,\kappa}}\sum_{g\in \cG^\kappa[p]} T^{\eps,\sigma_{a}}[g_a]\dots T^{\eps,\sigma_1}[g_1] e_\bj
\end{equ}
where $C_{m,\kappa}=\prod_{j=1}^{m-1}(\kappa+j)$ (in particular, $C_{1,\kappa}=1$), 
$\cG^\kappa[p]$ is a set whose elements $g=(g_s)_{s=1,\dots,a}$ are of the form
\begin{equ}
g_s\eqdef
\begin{cases}
(i_s,i'_s)\,,&\text{if $\sigma_s=+$}\\
q_s\,,&\text{if $\sigma_s=-$,}
\end{cases}
\end{equ}
for $1\le i_s<i'_s,q_s\le p_{s}+\kappa-1$. 
In particular, the cardinality of $\cG^\kappa[p]$ is finite and independent of $\eps$. 
\medskip

Thanks to the representation in~\eqref{e:NewRep}, we are ready to state the 
next lemma which will give~\eqref{e:nullo} as a corollary. 

\begin{lemma}\label{l:finernullo}
In the same setting as Proposition~\ref{p:limite}, let $a> 0$, $m\geq 2$ and $p\in \Pi^{(n)}_{a,m}$. 
Then, for every $g\in \cG^\kappa[p]$, we have 
\begin{equ}[e:finernullo]
\lim_{\eps\to 0} \|(-\gensy)^{-\frac12}\ST^\eps_{p}[g] e_\bj\|_{L^2_{m_\kappa}}=0\,, 
\end{equ}
where $\kappa=\kappa(e_{\bj})$ is defined as above, and $m_\kappa\eqdef m+\kappa-1$. 
\end{lemma}
\begin{proof}
By the definition of $\ST^\eps_{p}[g]$ in~\eqref{e:NewRep}, we have 
\begin{equs}
\|(-\gensy)^{-\frac12}\ST^\eps_{p}[g] e_\bj\|_{L^2_{m_\kappa}}=\|(-\gensy)^{-\frac12}T^{\eps,\sigma_{a}}[g_a]\dots T^{\eps,\sigma_1}[g_1] e_\bj\|_{L^2_{m_\kappa}}\,.
\end{equs}
Let $\gamma\in(0,1)$ and recall the definition of $\widetilde T^{\eps,\sigma}[g]$ from Lemma~\ref{l:HighModes}. 
Thanks to~\eqref{e:ApproxLarge}, we can replace every instance of $T^{\eps,\sigma}[g]$ with 
$\widetilde T^{\eps,\sigma}[g]$ since, by~\eqref{e:TboundNew}, 
the $T^{\eps,\sigma_{a}}[g_a]$'s are bounded operators and $e_\bj$ has clearly finite $L^2$ norm.  
Therefore we are left to consider 
\begin{equs}[e:Finalnullo]
\|(-\gensy)^{-\frac12}&\widetilde T^{\eps,\sigma_{a}}[g_a]\dots \widetilde T^{\eps,\sigma_1}[g_1] e_\bj\|_{L^2_{m_\kappa}}^2 \\
&=\sum_{k_{1:m_{\bj}}}\frac{2}{|k_{1:m_{\bj}}|^2}\Big|\widetilde T^{\eps,\sigma_{a}}[g_a]\dots \widetilde T^{\eps,\sigma_1}[g_1] e_\bj(k_{1:m_\kappa})\Big|^2\\
&\leq \eps^{2\gamma}\sum_{k_{1:m_{\bj}}}\Big|\widetilde T^{\eps,\sigma_{a}}[g_a]\dots \widetilde T^{\eps,\sigma_1}[g_1] e_\bj(k_{1:m_\kappa})\Big|^2\\
&=\eps^{2\gamma}\|\widetilde T^{\eps,\sigma_{a}}[g_a]\dots \widetilde T^{\eps,\sigma_1}[g_1] e_\bj\|_{L^2_{m_\kappa}}^2\,
\end{equs}
where, in the second step, we used the fact that $|k_{1:m_\kappa}|>\eps^{-\gamma}$. 
To see this, notice that, by definition, the new Fourier modes produced by $\widetilde T^{\eps,+}[g]$ 
have modulus bigger than $\eps^{-\gamma}$. Now, the assumption $m\geq 2$ ensures 
that the number of variables on which 
$\widetilde T^{\eps,\sigma_{a}}[g_a]\dots \widetilde T^{\eps,\sigma_1}[g_1] e_\bj$ depends (i.e. $m_\kappa$) is 
strictly bigger than that of $e_\bj$. Therefore, at least one among $k_1,\dots, k_{m_\kappa}$ 
must be generated by a $\widetilde T^{\eps,+}[g]$, which then implies $|k_{1:m_\kappa}|>\eps^{-\gamma}$. 
To conclude, it suffices to observe that by~\eqref{e:ApproxLarge} and~\eqref{e:TboundNew}, 
the operators $\widetilde T^{\eps,\sigma}[g]$ are bounded and that $e_\bj$ has finite $L^2$-norm, 
so that the right hand side of~\eqref{e:Finalnullo} 
converges to $0$, thus proving~\eqref{e:finernullo}. 
\end{proof}

To prove~\eqref{e:cp}, let $a\geq 0$, $m\geq 1$, $p\in \Pi^{(n)}_{a,m}$ and $g\in\cG^\kappa[p]$, 
for $\kappa=1,2$. 
We introduce a graphical representation for $g\in\cG^\kappa[p]$, 
which consists of associating to it a graph as that depicted in Figure~\ref{fig:1}.

\begin{figure}
\begin{center}
\includegraphics[width=5cm]{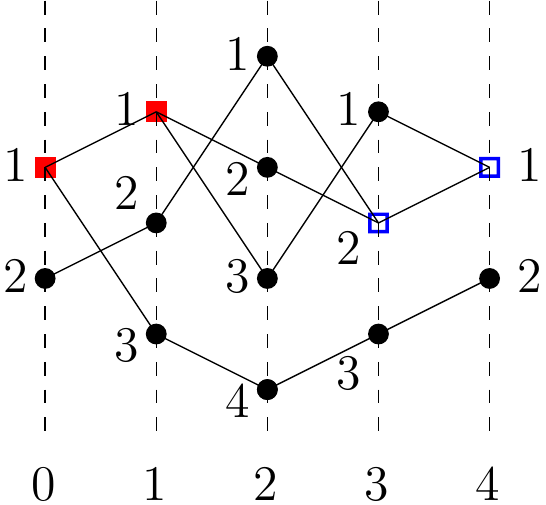}
\caption{A graphical representation for a graph
  $g\in \mathcal G^\bj[p]$. In this
  example, $p\in \Pi^{(n)}_{a,m}$ with $a=|p|=4$, $n\ge 3$ and $m=1$, $\kappa=2$ and 
  $g=\{(1,3), (2,3), 2, 1\}$. 
  Columns are labelled from $0$ to $a$ starting from the left, vertices are
  labelled by increasing integers starting from top. Vertices marked as empty squares are
  merging points, full squares are branching points. The number of
  vertices in column $j$ is $p_j+1$.}
\label{fig:1}
\end{center}
\end{figure}
The vertices of the graph are divided in $a+1$ columns, labeled from $0$ to $a$. The column with label $s$ contains
 $p_{s}+\kappa-1$ vertices; 
 the vertices inside each column are labeled with  increasing positive integers, starting from the uppermost vertex.
The edges only connect vertices in consecutive columns, with each vertex in one column 
connected to at least one and at most two vertices in the next. The edges between 
column $s-1$ and $s$, $s\in \{1,\dots,a\}$, are drawn starting from the vertices involved in the definition of $g_s$:   
if $g_s=(i_s,i'_s)$ then the first vertex from the top in column $s-1$ 
is called a {\it branching point} and 
we draw two edges connecting it 
to the vertices with label $i_s$ and $i'_s$ in column $s$; if instead $g_s=q_s$, 
then the vertex with label $q_s$ in column $s$ is called a {\it merging point}  and 
we draw an edge from the vertices with labels $1$ and $2$ in column $s-1$ to  
$q_s$; for the other vertices in column $s-1$ (i.e. those whose label is not $1$ if $g_s=(i_s,i'_s)$, 
and neither $1$ nor $2$ if $g_s=q_s$), 
we inductively draw an edge from the vertex with the lowest label to that of lowest label in column $s$ 
which is not already connected to another vertex in column $s-1$. 
\medskip

One should think of the vertices as carrying a momentum, and of the edges as the relation the 
momenta on the vertices connected by them satisfy. In particular, when $\ST^\varepsilon_p[g]$ is applied to  $e_{j_1}$ or\footnote{here we atually mean $e_{j_1}\otimes e_{j_2}$, not its symmetrized version. Recall that the action of the  operators $\mathcal T^\varepsilon_p[g]$ is well-defined on not necessarily symmetric functions} $e_{j_1}\otimes e_{j_2}$,
the vertex in position $1$ of the first column 
carries momentum $j_1$, that in position $2$ (if there is one) carries momentum $j_2$, a branching point determines 
the {\it creation} of a new momentum, while a merging point, the {\it annihilation} of one. Moreover, 
if $x,y,z$ are vertices then
\begin{itemize}[noitemsep]
\item if $x$ in column $s$ is either a branching or a merging point, and is connected to $y,z$ in column $s+1$ or $s-1$ respectively, then the sum of the momenta at $y,z$ must coincide with the momentum at $x$;  
\item if there is an edge between $x$ in column $s$ and $y$ in column $s+1$, $x$ is not a branching point and 
$y$ is not a merging point, then the momenta at $x$ and $y$ coincide. 
\end{itemize}
In the next lemma, we identify a large class of elements $g\in\cG^{2}[p]$ for which 
$\ST^\eps_p[g] (e_{j_1}\otimes e_{j_2})$ 
does not contribute to the limit  in \eqref{e:cp}. To state it, let us define a {\it path} $\pi$ in the graph associated to $g$ 
as a directed sequence of connected vertices in which $\pi(j)$ 
is the label of the vertex $\pi$ encounters in column $j$. 
%
%
%

\begin{lemma}\label{l:fastidi}
Let $a\geq 0$, $p\in \Pi^{(n)}_{a,1}$ and $g\in\cG^2[p]$. 
If in the graph associated to $g$ there exists a path $\pi$ starting from the second vertex  
of the $0$-th column that contains either a branching or a merging point, then 
\begin{equ}[e:fastidi]
\lim_{\eps\to 0} \|\ST^\eps_p[g] (e_{j_1}\otimes e_{j_2})\|_{L^2_2}=0\,.
\end{equ}
\end{lemma}
\begin{proof} 

Let $\pi$ be a path starting from vertex with label $2$ in the column of label $0$ 
and containing either a branching or a merging point. See Fig. \ref{fig:2}.
\begin{figure}
\begin{center}
\includegraphics[width=4cm]{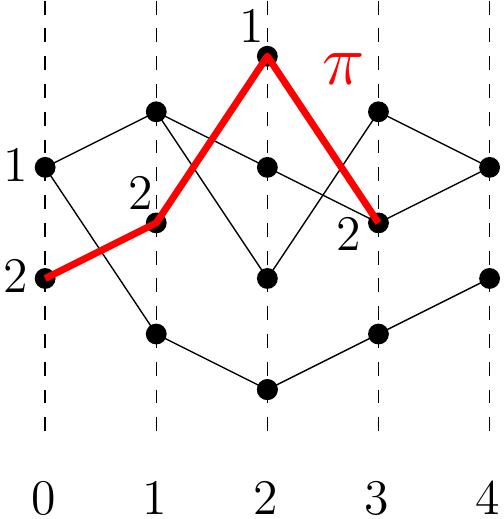}\qquad \qquad\qquad\includegraphics[width=4cm]{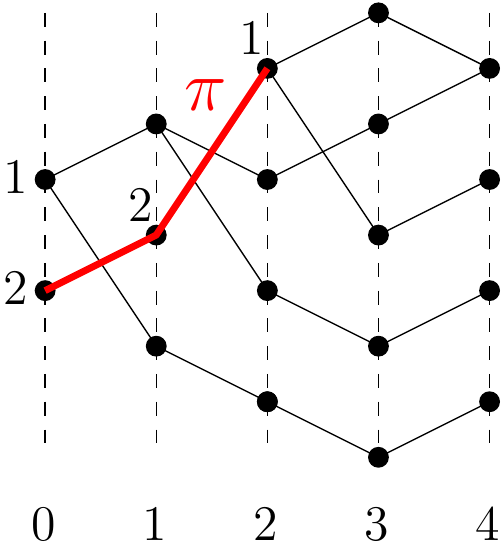}
\caption{Left drawing: the path $\pi$ from the vertex $2$ of the $0$-th column, until the first merging point (in column labeled $\bar a=3$). Right: the path $\pi$ until the first branching point (in column labeled $\bar a-1=2$).}
\label{fig:2}
\end{center}
\end{figure}
Let $\bar a\in\{2,\dots, a\}$ be such that, in the former case, the first branching point $\pi$ encounters 
lies in the $(\bar a-1)$-th column, so that in particular $\pi(\bar a-1)=1$. In the latter case, 
$\bar a$ is such that $\pi(\bar a)$ is a merging point, and $\pi(\bar a-1)\in\{1,2\}$.  
Then, invoking~\eqref{e:TboundNew}, we upper bound the norm
  in~\eqref{e:fastidi} as
\begin{equs}
\|\ST^\eps_p[g] (e_{j_1}\otimes e_{j_2})\|_{L^2_2}&=\|T^{\eps,\sigma_{a}}[g_a]\dots T^{\eps,\sigma_1}[g_1] (e_{j_1}\otimes e_{j_2})\|_{L^2_2}\\
&\lesssim \|T^{\eps,\sigma_{\bar a}}[g_{\bar a}]\dots T^{\eps,\sigma_1}[g_1] (e_{j_1}\otimes e_{j_2})\|_{L^2_{M+p_{\bar a}-p_{\bar a-1}}}\,.
\end{equs}
for $M\eqdef 1+p_{\bar a-1}$. 
To simplify the notation, set 
\begin{equ}[e:Psiuffa]
\Psi\eqdef T^{\eps,\sigma_{\bar a-1}}[g_{\bar a-1}]\dots T^{\eps,\sigma_1}[g_1] (e_{j_1}\otimes e_{j_2})\,,
\end{equ} 
and $\Psi\in L^2_M$. 

Let us begin by considering the case in which the path reaches a
branching point.  Then, by definition, $\sigma_{\bar a}=p_{\bar a}-p_{\bar a-1}=+$ and there
exist $i_{\bar a}< i'_{\bar a}$ such that $g_{\bar a}= (i_{\bar a},
i'_{\bar a})$. Let us immediately point out that, by the description
of the graph we have given above, the momentum carried by the vertex
at the branching point is necessarily $j_2$ since this is the first
branching point on $\pi$ by construction. Also, the sum of the momenta
carried by the vertices at $i_{\bar a}, i'_{\bar a}$ must be
$j_2$. Hence, the quantity we need to control equals
\begin{equs}
\|&T^{\eps,+}[(i_{\bar a}, i'_{\bar a})]\Psi\|_{L^2_{M+1}}^2\\
&=\frac{16}{(2\pi)^d}\eps^{d-2}\sum_{k_{1:M+1}}\frac{1}{|k_{1:M+1}|^2} \frac{[\fw\cdot (k_{i_{\bar a}}+k_{i'_{\bar a}})]^2}{|k_{i_{\bar a}}+k_{i'_{\bar a}}|^2+|k_{1:M+1\setminus\{i_{\bar a}, i'_{\bar a}\}}|^2}\times\\
&\qquad\qquad\qquad\times\indN{k_{i_{\bar a}},k_{i'_{\bar a}}}|\hat \Psi(k_{i_{\bar a}}+k_{i'_{\bar a}}, k_{1:M+1\setminus\{i_{\bar a}, i'_{\bar a}\}})|^2\\
&=\frac{16}{(2\pi)^d}\sum_{k_{1:M-1}}\frac{[\fw\cdot j_2]^2}{|j_2|^2+ |k_{1:M-1}|^2}|\hat \Psi(j_2, k_{1:M-1})|^2\times\\
&\qquad\qquad\qquad\times \eps^{d-2}\sum_{\ell+m=j_2}\frac{\indN{\ell,m}}{|k_{1:M-1}|^2+|\ell|^2+|m|^2}
\end{equs}
where in the last step we renamed the variables $(k_{i_{\bar a}},k_{i'_{\bar a}})=(\ell,m)$ and  
removed the variable corresponding to $j_2$ from the outer sum since $j_2$ is fixed. 
Now, the inner sum is bounded uniformly in $k_{1:M-1}$ (as can be seen by~\eqref{eq:thesum} with $\delta=0$) 
so that we deduce 
\begin{equs}
\|&T^{\eps,+}[ (i_{\bar a}, i'_{\bar a})]\Psi\|_{L^2_{M+1}}^2
\lesssim \sum_{k_{1:M-1}}\frac{[\fw\cdot j_2]^2}{|k_{1:M-1}|^2+|j_2|^2}|\hat \Psi( j_2,k_{1:M-1})|^2\\
&\leq  [\fw\cdot j_2]^2\sum_{k_{1:M}}\frac{2}{|k_{1:M}|^2}|\hat \Psi(k_{1:M})|^2=[\fw\cdot j_2]^2\|(-\gensy)^{-\frac12}\Psi\|_{L^2_M}^2\,.
\end{equs}
Since by definition~\eqref{e:Psiuffa}, $\Psi=\ST^\eps_{p'}[g']$ for a suitable choice of $p'$ and $g'$, 
the last norm at the right hand side converges to $0$ by~\eqref{e:finernullo}. 
\medskip

We now turn to the case in which $\pi$ reaches a merging point.  
This time, $\sigma_{\bar a}=-$ and there exists $q_{\bar a}$ such that 
$g_{\bar a}=q_{\bar a}$, and $\pi(\bar a-1)\in\{1,2\}$. Without loss of generality, assume $\pi(\bar a-1)=1$.  
The norm under study is
\begin{equs}
\|&T^{\eps,-}[q_{\bar a}]\Psi\|_{L^2_{M-1}}^2\\
&=\frac{16}{(2\pi)^{d}} \sum_{k_{1:M-1}}
\frac{(\fw\cdot k_{q_{\bar a}})^2}{|k_{1:M-1}|^2} \Big|\eps^{\frac{d}{2}-1}\sum_{\ell+m=k_{q_{\bar a}}}\indN{\ell,m}\frac{\hat \Psi(\ell,m, k_{1:M-1\setminus\{q_{\bar a}\}})}{(|\ell|^2+|m|^2+|k_{1:M-1\setminus\{q_{\bar a\}}}|^2)^{1/2}}\Big|^2\\
&\lesssim \sum_{k_{1:M-1}} \Big|\eps^{\frac{d}{2}-1}\sum_{\ell+m=k_{q_{\bar a}}}\indN{\ell,m}\frac{\hat \Psi(\ell,m, k_{1:M-1\setminus\{q_{\bar a}\}})}{(|\ell|^2+|m|^2+|k_{1:M-1\setminus\{q_{\bar a\}}}|^2)^{1/2}}\Big|^2\label{e:merge}
\end{equs}
Now, in the sum above, if $\ell$ is the momentum carried by the vertex at $\pi(\bar a-1)$, 
then $\ell$ must be $j_2$. In particular, $\ell=j_2$ is fixed and, therefore, so is $m$ since it must coincide with 
$k_{q_{\bar a}}-j_2$ ($k_{q_{\bar a}}$ is already summed in the external sum). 
This means that the sum inside the square at the right hand side of~\eqref{e:merge} has a unique summand. 
Hence, by applying the change of variables $k_{q_{\bar a}}\mapsto k_{q_{\bar a}}-j_2$ and lower bounding the denominator with $1$, we deduce that the right hand side of~\eqref{e:merge} 
is given by
\begin{equs}
\eps^{d-2}\sum_{k_{1:M-1}} \Big|\hat \Psi(j_2,k_{q_{\bar a}},k_{1:M-1\setminus\{q_{\bar a}\}})\Big|^2\leq \eps^{d-2}\|\Psi\|^2_{L^2_M}\,.
\end{equs}
Since the norm of $\Psi$ is bounded uniformly in $\eps$,~\eqref{e:fastidi} follows at once. 
\end{proof}

Recall that our final goal is to prove Proposition \ref{p:limite} and that each 
$\mathcal T_p^\varepsilon e_{\bf j}$ has been decomposed as a sum over $g\in \cG^\kappa[p]$.
Thanks to Lemma \ref{l:fastidi}, we are left to consider only $g\in \cG^1[p]$
and those $g\in \cG^2[p]$  whose associated graph has two connected 
components, one of which is a path with no branching points. We call the collection of such graphs $\widetilde \cG^2[p]$; see Fig. \ref{fig:3} for an example. In the next lemma we determine a correspondence between 
$\cG^1[p]$ and $\widetilde \cG^2[p]$ valid in case $m=1$. 

\begin{lemma}\label{l:Card}
  Let $a\geq 0$ be even and $p\in \Pi^{(n)}_{a,1}$. Then, for every
  $g\in\cG^1[p]$ there exist exactly two elements of
  $\widetilde \cG^2[p]$, called $\hat g$ and $\tilde g$, such
  that the graph associated to $g$ can be obtained from that of
  $\hat g$ and $\tilde g$ by removing their connected component
  corresponding to the single path $\pi$.  
\end{lemma}
\begin{proof}
  The construction of the graphs $\hat g,\tilde g\in  \cG^2[p]$ given $g\in \cG^1[p]$ is explained in Figure \ref{fig:biiezione}. The fact that there is no other $g'\in \cG^2[p]$ that reduces to
  $g$ once the path $\pi$ is removed, follows from the fact that the path $\pi$ can cross the connected component starting at the vertex $1$ of colum labelled $0$ only in the step from the second to last to the last column. This can be seen for instance by induction, following the graph from right to left. 
\end{proof}

\begin{figure}\begin{center}\includegraphics[width=3.7cm]{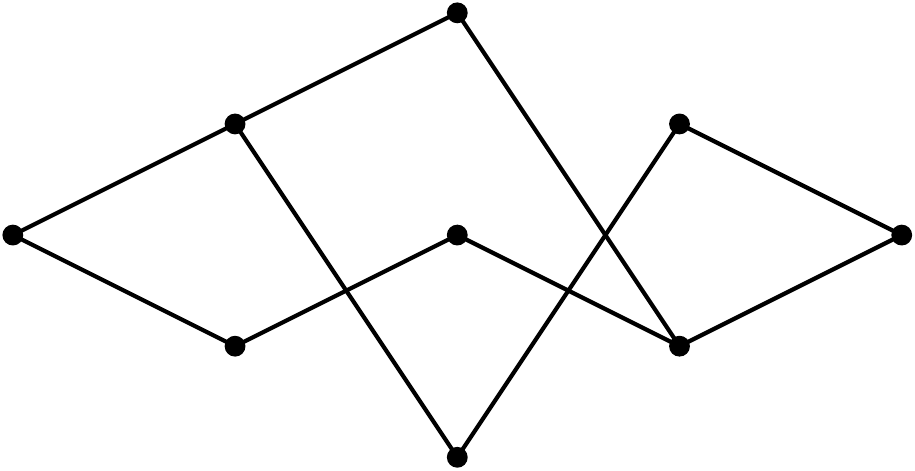}\qquad
    \includegraphics[width=3.7cm]{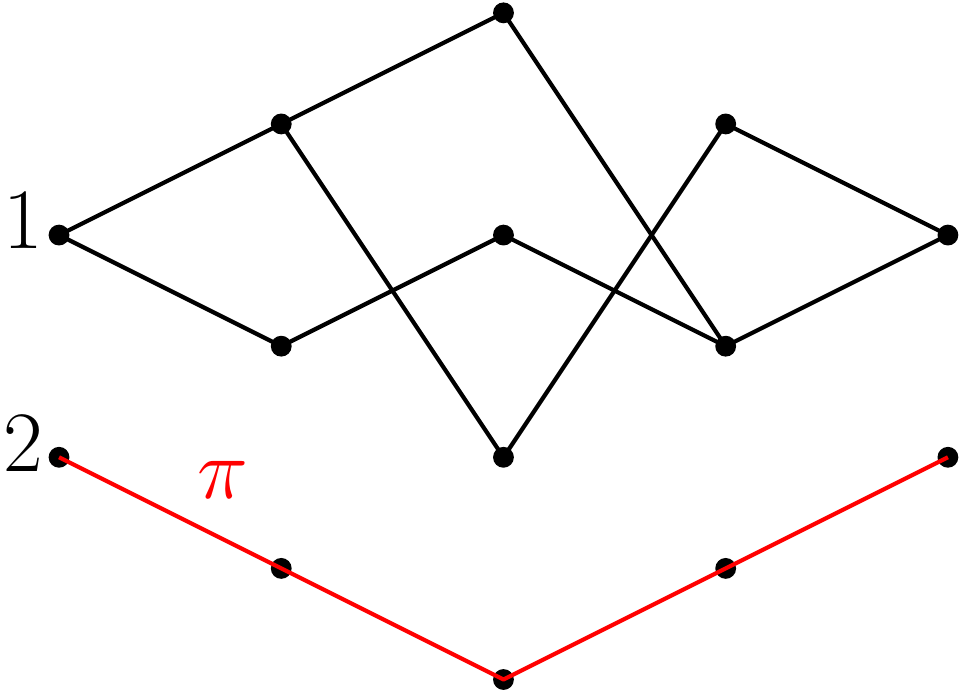}\qquad\includegraphics[width=3.7cm]{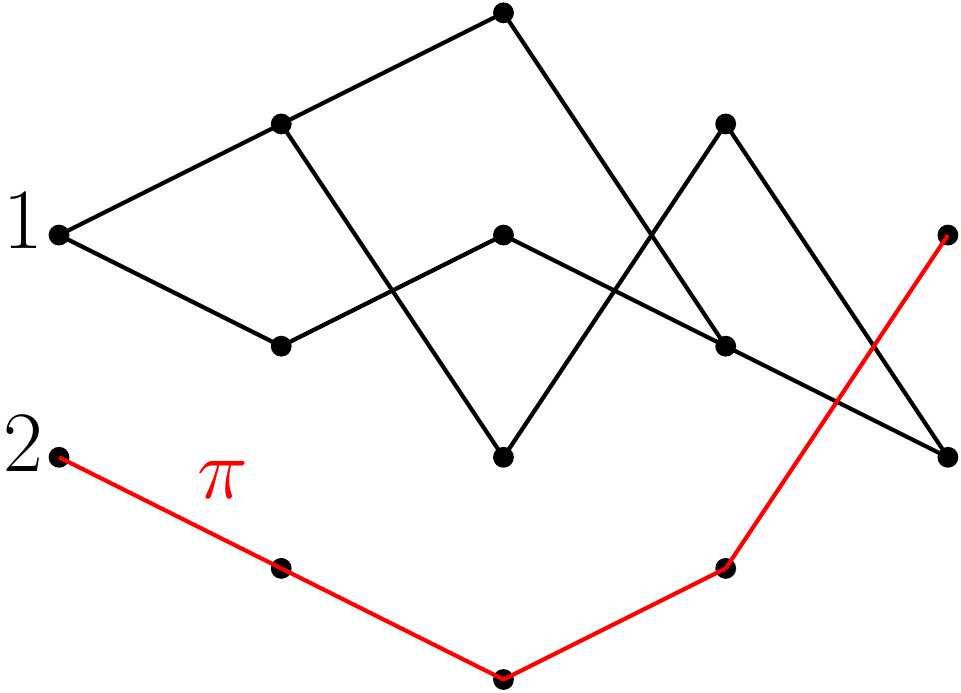}
\caption{Left: a graph $g\in  \cG^1[p]$. Center and right: the graphs $\hat g,\tilde g\in  \cG^2[p]$ corresponding to $g$. In particular, $\hat g$ is obtained by simply attaching the path $\pi$ below $g$, and $\tilde g$ is obtained from $\hat g$ by changing $\pi$ at the very last step. Note that, when $\pi$ is removed from $\tilde g$, the last column contains a single vertex, so the resulting graph is again $g$. The path $\pi$ cannot cross the other connected component anywhere else.} 
\label{fig:biiezione}
\end{center}
\end{figure}

\begin{figure}\begin{center}\includegraphics[width=5cm]{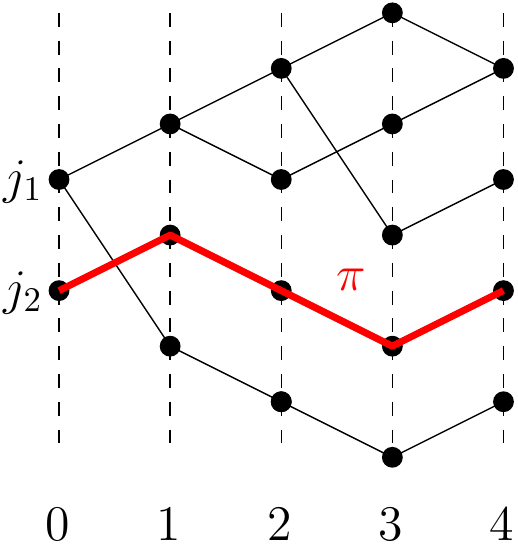}
\caption{A graph corresponding to $\bj=j_{1:2}$ and to $p\in \Pi^{(n)}_{a,m}$ with $n\ge 4$, $a=4$, $m=3$. In this graph, the path $\pi$ has neither branching nor merging points: the graph is disconnected. The label of the unique vertex of $\pi$ in column $i$ defines $\pi(i)$.} 
\label{fig:3}
\end{center}
\end{figure}

Let $p\in \Pi^{(n)}_{a,m}$ and $g\in\cG^1[p]$ or in $\widetilde\cG^2[p]$. 
Let, as above, $m_\kappa=m+\kappa-1$, where $\kappa=1$ if $g\in\cG^1[p]$ and $2$ in the other case. 
Moreover, let
\begin{equ}[eq:mfe]
{\mathfrak e}_\bj\eqdef 
\begin{cases}
e_{j_1}\,, & \text{ if } \bj=j_1 \\
e_{j_1}\otimes e_{j_2}\,&  \text{ if } \bj=j_{1:2}
\end{cases}
\end{equ}
(note that in the latter case, ${\mathfrak e_\bj}\ne e_\bj$ if $j_1\ne j_2$).
Then, $\ST^\eps_p[g] {\mathfrak e}_{\bj}\in L^2_{m_\kappa}$, 
and we write the Fourier transform of its kernel in the form 
\begin{equ}\label{eq:formkernel}
f_g^\eps(k_{1:m_\kappa},\bj)= \eps^{\frac d2(m_\kappa-1)}F^\eps_g(\eps k_{1:m_\kappa};\eps \bj)\,,
\end{equ}
in which we singled out the scaling factor $\eps^{\frac d2(m_\kappa-1)}$. 
We derive an expression for $F^\eps_g$ in terms of ratios of polynomials, whose form and main properties are summarised 
in the following lemma.

\begin{lemma}\label{lemma:kernels}
  Given $p\in \Pi^{(n)}_{a,m}$ and $a\ge 1$, let
  $M=(1+a-m)/2\,(\ge 0)$ be the number of $T^{\eps,-}$ in the product
  $\ST^\eps_p$.
  
  Let $j_1,j_2\in\Z_0^d$, $\bj=j_{1:2}$ and $\bz=(z_1,z_2)$. For 
  $g\in\widetilde\cG^2[p]$, the kernel
  $F_g^\eps(x_{1:m+1};\bz)$ in \eqref{eq:formkernel}, corresponding
  to $\ST^\eps_p[g]{\mathfrak e_\bj}$, can be written as
\begin{equs}\label{eq:formkernelG}
F_g^\eps&(x_{1:m+1};\bz)=\1_{\sum_{i\neq \pi(a)} x_i= z_1}\1_{x_{\pi(a)}=z_2}\frac{(\fw\cdot z_1)}{|\bz|} 
\frac{\iota^a }{|x_{1:m+1}|}\times\\
&\times \eps^{d M}\sum_{\substack{y_{1:M}\in (\eps \Z^d_0)^M:\\|y_i|_
\infty\le 1, i\le M}} \frac{P_g(y_{1:M};x_{1:m+1\setminus\{\pi(a)\}})}{Q_g(y_{1:M};x_{1:m+1})}I_g(y_{1:M};x_{1:m+1\setminus\{\pi(a)\}})\,,
\end{equs}
where $\pi(a)$ is the label of the vertex of the path $\pi$ belonging to the $a$-th column and
\begin{enumerate}[noitemsep]
\item  [(i)] $I_g$ is a product of  indicator functions, each  imposing that certain linear combinations of its arguments have $|\cdot|_\infty$ norm in $(0,1]$,
    
  \item [(ii)] $P_g$ (resp. $Q_g$) is a homogenous  real-valued polynomial of degree $a-1$ (resp. $2(a-1)$) 
  in $y_{1:M};x_{1:m+1\setminus\{\pi(a)\}}$ (resp. $y_{1:M};x_{1:m+1}$). Further, $Q_g\ge0$ and it does not vanish on the support of $I_g$,
    
  \item [(iii)] if $m=1$, then $Q_g(y_{1:M};0)>0$ for almost every $y_{1:M}$\,.
    
  \end{enumerate}
Moreover, if $m=1$ (so that, in particular, $a$ is even and $a\ge2$), 
then there exists a polynomial $\tilde P_g$ of degree $a-2$ such that for every $j_1', j_2'\in\Z^d_0$ 
and $\bj'=j'_{1:2}$, we have 
  \begin{equs}\label{eq:casem1}
    \langle {\mathfrak e}_{-\bj'}\,,&\ST^\eps_p[g] {\mathfrak e}_\bj\rangle_{L^2_{m_\kappa}}\\
    &=\1_{A_{\bj; \bj'}(g)}\iota^a \frac{(\fw\cdot j_1)^2}{|\bj|^2}
     \eps^{d M}\sum_{\substack{y_{1:M}\in (\eps \Z^d_0)^M:\\|y_i|_\infty\le 1,\, i\le M}} \frac{\tilde P_g(y_{1:M};\eps j_1)}{Q_g(y_{1:M};\eps \bj)}I_g(y_{1:M};\eps j_1)
  \end{equs}
where $\mathfrak{e}_\bj$ is defined according to~\eqref{eq:mfe} and 
 $A_{\bj; \bj'}(g)=\{j_{\pi(a)^c}'=j_1\;;\;j_{\pi(a)}'=j_2\}$, 
for $\pi(a)^c\eqdef\{1,2\}\setminus\{\pi(a)\}$. 

Similarly, if $j_1,j_1'\in \Z^d_0$ and  $g\in \cG^1[p]$, 
then the analog of \eqref{eq:formkernelG} and \ref{eq:casem1} holds upon
setting $\bz=z_1$, removing the dependence on $z_2$ (and $\pi(a)$) 
from the right hand side of~\eqref{eq:formkernelG}, replacing both $x_{1:m+1}$ and $x_{1:m+1\setminus\{\pi(a)\}}$ with $x_{1:m}$ in~\eqref{eq:formkernelG} and setting $A_{\bj; \bj'}(g)=\{j_1'=j_1\}$ in \eqref{eq:casem1}.
\end{lemma}

We emphasise that the functions $P_g, Q_g, I_g$ and $\tilde P_g$ are {\it $\eps$-independent}. 
The only $\eps$-dependence in $F^\eps_p$ comes from the fact that the Riemann sums runs over $\eps \Z^d_0$. 

\begin{proof}
The argument for $\bj=j_1$ is identical (and simpler) than that for $\bj=j_{1:2}$, so we will 
only focus on the latter. 

The statement is proven via induction on $|p|=a\ge1$. In order to verify condition {\it (iii)}, we will 
simultaneously show that  $Q_g$ further satisfies
\begin{enumerate}
\item[$(iii')$] $Q_g(y_{1:M};x_{1:m+1})$ is the product of homogeneous quadratic polynomials $Q^r$, $1\leq r\leq a-1$, 
each being the sum of $|x_{\pi(a)}|^2$ and 
the squared norms of linear combinations of the other arguments (which are $(y_{1:M};x_{1:m+1\setminus\{\pi(a)\}})$).
For every $1\leq r\leq a-1$, 
$Q^r\neq |\sum_{i\neq \pi(a)}x_i|^2\,,|x_{\pi(a)}|^2$ or the sum of the two.
\end{enumerate}
We begin our induction with $a=1$. In this case, if $p\in  \Pi^{(n)}_{a,m}$, 
then necessarily $m=2,M=0$, $\sigma_1=+$ and $g_1=(i,i')$ for some $1\le i<i'\le 3$. 
From~\eqref{e:GenaNonSym} and the definition of $T^{\eps,+}[(i,i')]$, 
one sees that 
\begin{equ}\label{eq:casoa=1}
    F^\eps_g(x_{1:3};\bz)=\1_{\sum_{i\neq \pi(a)} x_i= z_1}\1_{ x_{\pi(a)}=z_2}\frac{(\fw\cdot z_1)}{|\bz|}\frac{\iota}{|x_{1:3}|}\frac{-4}{(2\pi)^{d/2}}\mathbb J^1_{x_{i},x_{i'}}
\end{equ}
where the first indicator function comes from the fact that $T^{\eps,+}[(i,i')]$ commutes 
with the momentum operator (see Lemma~\ref{l:AnonsymProp}) and that the kernel of 
$e_{j_1}\otimes e_{j_2}$ at $(k_1, k_2)$ is $\1_{k_1=j_1}\1_{k_2=j_2}$, 
This expression is of the correct form \eqref{eq:formkernelG}, 
with $P_g\equiv-4/(2\pi)^{d/2}$, $Q_g\equiv 1$, $I_g\equiv \mathbb J^1_{x_{i},x_{i'}}$, 
so that $(i),(ii)$ and $(iii')$ are clearly satisfied. 
As for the induction step, we assume the validity of the statement for a given $a\ge1$ and 
take $p\in\Pi^{(n)}_{a+1,m}$ and $g\in\widetilde\cG^2[p]$. 
Write $\ST^\eps_p[g]=T^{\eps,\sigma}[g_{a+1}]\mathcal T_{p'}^\eps[g']$ for some 
$p'\in\Pi^{(n)}_{a,m-\sigma_{a+1}}$ and $g'\in \widetilde\cG^2[p']$. 
Consider first the case $\sigma_{a+1}=+$ and let $i, i'$ be such that 
$g_{a+1}=(i, i')$. 
Recalling the definition of $T^{\eps,+}[(i, i')]$ and~\eqref{eq:formkernel}, we have 
 \begin{equs}
 (T^{\eps,+}[(i,i')]&f^\eps_{g'})(k_{1:m+1};\bj)=-\frac{4\iota}{(2\pi)^{\frac{d}{2}}}\eps^{\frac{d}{2}-1+\frac{d}{2}(m_\kappa-2)}\frac1{|k_{1:m+1}|}\times\\
&\times
\frac{\fw\cdot(k_i+k_{i'}) \indN{k_i,k_{i'}}}{\sqrt{|k_i+k_{i'}|^2+|k_{1:m+1\setminus\{i,i'\}}|^2}} F^\eps_{g'}(\eps k_i+\eps k_{i'},\eps k_{1:m+1\setminus\{i,i'\}};\eps \bj)\\
=&\eps^{\frac{d}{2}(m_\kappa-1)}F^\eps_{g}(\eps k_{1:m+1};\eps \bj), \label{eq:T+}
\end{equs}
with
\begin{equs}\label{eq:F+}
F^\eps_{g}&(x_{1:m+1};\bz)=\frac\iota{|x_{1:m}|}
\frac{-4}{(2\pi)^{\frac{d}{2}}}\times\\
&\times \frac{\fw\cdot (x_i+x_{i'})}{\sqrt{|x_i+x_{i'}|^2+|x_{1:m+1\setminus\{i,i'\}}|^2}}\mathbb J^1_{x_i, x_{i'}}F^\eps_{g'}(x_i+x_{i'}, x_{1:m+1\setminus\{i,i'\}};\bz)\,.
\end{equs}
By induction, $F_{g'}$ has the form in~\eqref{eq:formkernelG}, 
and therefore so does $F_g$ once we set 
\begin{equs}[e:casosale]
P_g(y_{1:M};x_{1:m+1\setminus\{\pi(a+1)\}})&=-\frac{4}{(2\pi)^{\frac{d}{2}}} (\fw\cdot (x_i+x_{i'}))P_{g'}\\
Q_g(y_{1:M};x_{1:m+1})&=(| x_i+x_{i'}|^2+|x_{1:m+1\setminus\{i,i'\}}|^2)Q_{g'}\\
    I_g(y_{1:M};x_{1:m+1\setminus\{\pi(a+1)\}})&= \mathbb J^1_{x_i,x_{i'}}I_{g'}\\
\end{equs}
where now the arguments of $P_{g'}$ and $I_{g'}$ are $(y_{1:M};x_i+x_{i'},x_{1:m+1\setminus\{i,i', \pi(a+1)\}})$ 
(and the lack of dependence on $x_{\pi(a+1)}$ follows from the fact that, by construction, $\pi(a+1)\neq i,i'$) 
while the arguments of  $Q_{g'}$ are $(y_{1:M};x_i+x_{i'},x_{1:m+1\setminus\{i,i'\}})$. 

If $\sigma_{a+1}=-$, then there exists $q$ such that $g_{a+1}=q$. By~\eqref{e:GenaNonSym}, 
\begin{equs}
  (T^{\eps,-}[q]f^\eps_{g'})(k_{1:m-1};\bj) =&\frac{4\iota }{(2\pi)^{\frac{d}{2}}}\eps^{\frac{d}{2}-1}\frac{(\fw\cdot k_q)}{|k_{1:m-1}|}\times\\
  &\times\sum_{y,z\in \Z^d\colon y+z=k_q}\indN{y,z}\frac{\eps^{\frac{d}{2}m_\kappa}F^\eps_{g'}(\eps y,\eps z,
    \eps k_{1:m-1\setminus\{q\}};\eps \bj)}{\sqrt{|y|^2+|z|^2+|k_{1:m-1\setminus\{q\}}|^2}}
  \\
  =&\eps^{\frac{d}{2}(m_\kappa-1)}F^\eps_{p}(\eps k_{1:m};\eps k),  
\end{equs}
where in this case
\begin{equs}  \label{eq:F-}
F^\eps_{g}(x_{1:m};\bz)&= \frac{4\iota}{(2\pi)^{\frac{d}{2}}}\frac1{|x_{1:m-1}|}(\fw\cdot x_q)\1_{0<|x_q|_\infty\le 1}\\
&\times\sum_{\substack{y\in \eps \mathbb Z^d_0:\\ |y|_\infty
    \le 1}}\eps^d\1_{0<|x_q-y|_\infty\le 1}\frac{F^\eps_{g'}(y,x_q-y, x_{1:m-1\setminus\{q\}};\bz)}{\sqrt{|y|^2+|x_q-y|^2+|x_{1:m-1\setminus\{q\}}|^2}}.
\end{equs}
Again, as $F^\eps_{g'}$ is of the form \eqref{eq:formkernelG} by the induction hypothesis, 
the same holds for $F^\eps_{g}$, where this time we need to set 
\begin{equs}
     P_g(y_{1:M};x_{1:m-1\setminus\{\pi(a+1)\}})&=\frac{4}{(2\pi)^{\frac{d}{2}}} (\fw\cdot x_q)P_{g'},\\
    Q_g(y_{1:M};x_{1:m-1})&=(|y_M|^2+|x_q-y_M|^2+|x_{1:m-1\setminus\{q\}}|^2)Q_{g'}\\
    I_g(y_{1:M};x_{1:m-1\setminus\{\pi(a+1)\}})&=\1_{0<|x_q|_\infty\le 1,0<|y_M-x_q|\le 1}I_{g'}\,,\label{e:casoscende}
\end{equs}
where the omitted arguments of the functions $P_{g'}$ and $I_{g'}$ are 
$(y_{1:M-1};y_M,x_q-y_M,x_{1:m-1\setminus \{q, \pi(a+1)\}})$, while those for $Q_{g'}$ 
are $(y_{1:M-1};y_M,x_q-y_M,x_{1:m-1\setminus \{q\}})$. Once again, the lack of dependence of 
$P_g$ and $I_g$ on $x_{\pi(a+1)}$ is due to the fact that, by construction, $q\neq \pi(a)$.

Now, for both $\sigma_{a+1}=\pm$,~\eqref{e:casosale} and~\eqref{e:casoscende} clearly ensure that 
if $P_{g'}$, $Q_{g'}$ and $I_{g'}$ satisfy conditions $(i)-(ii)$ 
so do $P_{g}$, $Q_{g}$ and $I_{g}$. The first part of $(iii')$ is obvious since the variable 
corresponding $x_{\pi(a)}$ is simply turned in $x_{\pi(a+1)}$ and is not affected by the operators $\cA^{\eps,\pm}[g_{a+1}]$ 
in the definition of $T^{\eps,\pm}[g_{a+1}]$. 
{For the second part of $(iii')$, note that the first factor in the expression for $Q_g$ in~\eqref{e:casosale} and~\eqref{e:casoscende}
is different from $|\sum_{i\neq \pi(a+1)}x_i|^2,\,|x_{\pi(a+1)}|^2$ or their sum. 
While this is clear for~\eqref{e:casoscende} because of the non-trivial $y_M$ dependence, 
for~\eqref{e:casosale} this follows by the fact that $m+1\geq 4$. Indeed, 
in the base induction step (see~\eqref{eq:casoa=1}) corresponding to $a=1$, $m+1=p_a+1=3$. 
Now, for $a>1$ and $\sigma_a=+$, by definition of $p\in \Pi^{(n)}_{a,m}$, $p_a=m>2$, which implies that the first factor 
in~\eqref{e:casosale} has at least three summands. It remains to show that 
 $Q_{g'}$, evaluated at its new arguments, does not contain a term of the aforementioned form. 
 But this can only happen if a similar term was there in the old variables and this is ruled out 
 by the induction hypothesis.  }
As for property $(iii)$, note that by $(iii')$, $Q_g$ can be written as the product of quadratic polynomials 
$Q^r$, $1\leq r\leq a-1$, and if $Q^r$ depends on at least one of the $y_i,\,i\le M$, then 
necessarily $Q^r(y_{1:M};0)\ne0$ for almost all $y_{1:M}$. 
The only potentially problematic case 
arises when one of the $Q^r$ is such that $Q^r(y_{1:M};x_{1:2})=|x_{1}|^2,\,|x_2|^2$ or their sum, but 
this is ruled out by $(iii')$ for $m=1$. 

At last, to verify~\eqref{eq:casem1}, we note that 
 \eqref{eq:formkernel} and \eqref{eq:formkernelG} give 
\begin{equs}\label{eq:plugged}
    \langle {\mathfrak e}_{-\bj'}\,,&\ST^\eps_p[g] {\mathfrak e}_\bj\rangle_{L^2_{m_\kappa}}\\
    &=\1_{A_{\bj;\bj'}(g)} \iota^a\frac{(\fw\cdot j_1)}{|\bj|} \frac{1}{\eps|\bj|}\eps^{d M}\sum_{\substack{y_{1:M}\in (\eps \Z^d_0)^M:\\|y_i|\le 1, i\le M}} \frac{P_g(y_{1:M};\eps j_1)}{Q_g(y_{1:M};\eps \bj)}I_g(y_{1:M};\eps j_1)\,.
\end{equs}
But, since $m=1$, the left-most operator in the product $\ST^\eps_p[g]$ is $T^{\eps,-}[q]$, $q\in\{1,2\}$, 
so that $\ST^\eps_p[q] {\mathfrak e}_\bj=T^{\eps,-}[q]\ST^\eps_{p'}[g']{\mathfrak e}_\bj$ 
and $\ST^\eps_{p'}[g']{\mathfrak e}_\bj\in L^2_{3}$.
Then, from \eqref{e:casoscende} we see that
\begin{equ}
    P_g(y_{1:M};\eps j_1)= \frac{4}{(2\pi)^{\frac{d}{2}}} (\fw\cdot (\eps j_1))P_{g'}(y_{1:M-1};y_M,\eps j_1-y_M)
\end{equ}
which,  upon defining
\begin{equ}
\tilde P_g(y_{1:M};\eps j_1)\eqdef \frac{8}{(2\pi)^{\frac{d}{2}}} P_{g'}(y_{1:M-1};y_M,\eps j_1-y_M)\,,
\end{equ}
reduces to \eqref{eq:casem1} once plugged into \eqref{eq:plugged}. 
\end{proof}

As a last result before turning to the proof of Proposition~\ref{p:limite}, 
in the next lemma we show that~\eqref{eq:casem1} converges in the limit for $\eps\to 0$. 

\begin{lemma}\label{l:RiemanConv}
In the setting and notation of the previous lemma, let $a\geq 1$ be even and $p\in \Pi^{(n)}_{a,1}$. 
Let $g$ be an element of either $\cG^1[p]$ or $\widetilde{\cG}^{2}[p]$.  
Then, there exists a constant $c_g(p)\in\R$ independent of $\bj$ such that for every $\bj'$
\begin{equ}[e:Const1]
\lim_{\eps\to 0} \langle {\mathfrak e}_{-\bj'}\,,\ST^\eps_p[g] {\mathfrak e}_\bj\rangle= \1_{A_{\bj; \bj'}(g)} \frac{(\fw\cdot j_1)^2}{|\bj|^2}c_g(p) \,.
\end{equ}
Furthermore, if $g\in \cG^1[p]$ and $\hat g$, $\tilde g\in\widetilde{\cG}^{2}[p]$ are those in Lemma~\ref{l:Card}, 
then 
\begin{equ}[e:ConstEqu]
c_g(p)=c_{\hat g}(p)=c_{\tilde g}(p)\,.
\end{equ}
\end{lemma}
\begin{proof}
Since $a$ is even, set $a=2b$ so that $M=b$. 
Our goal is to show that the Riemann-sum in~\eqref{eq:casem1} converges 
and that the limit is independent of $\bj$. Formally, such limit should be 
\begin{equ}[e:cg]
2^{bd}\int \mu(\dd y_{1:b}) \frac{\tilde P_g(y_{1:b};0)}{Q_g(y_{1:b};0)}I_g(y_{1:b};0)=: c[g]\in\R
\end{equ}
where\footnote{note that we added the volume factor $2^{ad}$ 
to pass from the Lebesgue measure to the uniform one. The scope of this change of measure will soon become apparent.} 
$\mu$ is the uniform probability measure on $[-1,1]^{bd}$ 
and  the integrand is well-defined and finite (except possibly on a zero-measure set) thanks to the  properties of $Q_g$ stated in Lemma \ref{lemma:kernels}. 
Note that if $c[g]$ is given by the expression above, then not only it is independent of $\bj$ but also~\eqref{e:ConstEqu} 
clearly holds. 
Hence, it remains to prove that the sum indeed converges to the integral, and that the integral is finite. 

At first, we want to write the Riemann-sum in~\eqref{eq:casem1} as an integral with respect to the 
uniform probability measure on $[-1,1]^{bd}$. 
This is quite standard as it suffices to take a piece-wise constant extension of the polynomials 
$\tilde P_g$ and $Q_g$, but since the quotient displays singularities, we provide the details below. 

Recall that we assume $1/\eps\in\mathbb N+1/2$. Then, the box
$[-1,1]^d$ contains exactly $(2/\eps)^d$ points of $\eps\Z^d$ and it can be written
as the union of cubes $C(y^*)$ of side $\eps$, centred at each of the
points $y^*\in \eps\Z^d\cap[-1,1]^d$. Given $y\in [-1,1]^d$, we let
$y^*(y)$ denote the (unique, up to Lebesgue-measure zero sets) $y^*$ such that $y\in
C(y^*)$.  The sum in \eqref{eq:casem1} is then
\begin{equ}
  \label{eq:sumint}
2^{bd} \int \mu(\dd y_{1:b})
  \frac{\tilde P^{(\eps)}_g(y_{1:b};\eps j_1)}{Q^{(\eps)}_g(y_{1:b};\eps \bj)}I^{(\eps)}_g(y_{1:b};\eps j_1),
\end{equ}
where $\mu$ is the uniform probability measure on 
$[-1,1]^{ad}$, $P^{(\eps)}_g$ is given by
\begin{equ}
\tilde P^{(\eps)}_g(y_{1:b};\eps j_1)\eqdef \tilde P_g(y^*(y_1),\dots,y^*(y_b))\1_{y_{i}\not\in[-\eps/2,\eps/2]^d, i\le b}\,,
\end{equ} 
and similarly $Q^{(\eps)}_g,I^{(\eps)}_g$, and the indicator function has been added 
because in \eqref{eq:casem1} the sums are restricted to $y_i\ne0$.
By construction, the integrand is piecewise constant. 
\medskip

At this point, existence, finiteness of the $\eps\to0$ limit of \eqref{eq:sumint} and its independence
w.r.t. $\bj$ follow if we can prove uniform integrability of the
integrand.
For this, we will exploit the operators $T^{\eps,\delta, \sigma}[\cdot]$ introduced in~\eqref{e:TdeltaNonSym}, 
and their composition $\ST^{\eps,\delta}_p[g]$.  
The same argument as in Lemma~\ref{lemma:kernels} together with the derivation of~\eqref{eq:sumint}, 
shows that 
\begin{equs}\label{e:stanco}
  \langle {\mathfrak e}_{-\bj'},& \ST^{\eps,\delta}_p[g] {\mathfrak e}_{\bj}\rangle\\
  &=\1_{A_{\bj; \bj'}(g)} 2^{d b}\frac{|\fw\cdot j_1|^{2(1+\delta)}}{|\bj|^{2(1+\delta)}} \int\mu(\dd y_{1:b}) \left|\frac{\tilde P^{(\eps)}_g(y_{1:b},\eps j_1)}{Q^{(\eps)}_g(y_{1:b},\eps \bj)}\right|^{1+\delta}I^{(\eps)}_g(y_{1:b},\eps j_1)\,.
\end{equs}
Note that, the integral at the right hand side of~\eqref{e:stanco} is the same 
as that in~\eqref{eq:sumint}, but in which the modulus of the integrand is raised to the power $1+\delta$. 
Now, for any $\delta\in[0,(d-2)/2)$, the quantity at the left hand side is bounded uniformly in $\eps$ 
(but the bound might depend on $m$ and $b$, which is though irrelevant as they are fixed)
since $\ST^{\eps,\delta}_p[g]$ is just a composition of 
$T^{\eps,\delta, \sigma}[\cdot]$ and, by~\eqref{e:TboundDelta}, 
these operators are uniformly bounded. Hence, the integral at the right hand side of~\eqref{e:stanco} 
is also uniformly bounded which then ensures uniform integrability of the integrand in~\eqref{eq:sumint}. 
As a consequence,~\eqref{e:cg} is well-defined and~\eqref{e:Const1} holds. 
\end{proof}

Let us now collect all the results obtained so far and complete the proof of Proposition~\ref{p:limite}. 

\begin{proof}[of Proposition \ref{p:limite}] Let us begin with~\eqref{e:nullo}. 
Let $a_r>0$, $m_r\geq 2$ and $p^r\in\Pi^{(n)}_{a_r,m_r}$, $r=1,2$. 
By Cauchy-Schwarz, we have 
\begin{equ}
\langle \overline {\mathcal T^\eps_{p^1} e_\bj},(-\gensy)^{-1}\mathcal T^\eps_{p^2} e_\bj\rangle\leq \|(-\gensy)^{-\frac12}\mathcal T^\eps_{p^1} e_\bj\|\|(-\gensy)^{-\frac12}\mathcal T^\eps_{p^2} e_\bj\|\,,
\end{equ}
so that we can reduce to study the norm at the right hand side for a single $p\in \Pi^{(n)}_{a,m}$. 
We expand $\ST^\eps_p e_\bj$ as in~\eqref{e:NewRep}, thus obtaining 
\begin{equs}
\|(-\gensy)^{-\frac12}\ST^\eps[p]e_\bj\|\lesssim \sum_{g\in\cG^\kappa[p]}\|(-\gensy)^{-\frac12} T^{\eps,\sigma_{a}}[g_a]\dots T^{\eps,\sigma_1}[g_1] e_\bj\|\,.
\end{equs}
Since the sum above has finitely many terms whose number does not depend on $\eps$, 
it suffices to show that each summand 
converges to $0$, which in turn is a direct consequence of Lemma~\ref{e:finernullo}. 
\medskip

We now turn to~\eqref{e:cp}, for which we consider $p^r\in\Pi^{(n)}_{a_r,1}$, $r=1,2$. 
As above, we expand the scalar product via~\eqref{e:NewRep} 
\begin{equ}\label{e:NewLim}
\lim_{\eps\to 0}\langle \overline {\mathcal T^\eps_{p^1} e_\bj},\mathcal T^\eps_{p^2} e_\bj\rangle=\sum_{\substack{g^r\in\cG^\kappa[p^r]\\r=1,2}}\lim_{\eps\to 0}\langle \overline{\ST^\eps_{p^1}[g^1]e_\bj}, \ST^\eps_{p^2}[g^2]e_\bj\rangle_{L^2_{m_\kappa}}\,,
\end{equ}
where we used that $C_{1,\kappa}=1$ and, to pass the limit inside the sum, that the sums above have finitely many terms. 

First, we focus on the case $\bj=j_1$, so that $\kappa=1=m$. The limit in~\eqref{e:NewLim} is given by 
\begin{equs}
\lim_{\eps\to 0}&\langle \overline{\ST^\eps_{p^1}[g^1]e_{j_1}}, \ST^\eps_{p^2}[g^2]e_{j_1}\rangle
=\lim_{\eps\to 0}\sum_{j_1'}\langle e_{-j_1'}, \ST^\eps_{p^r}[g^r]e_{j_1}\rangle\langle e_{-j_1'}, \ST^\eps_{p^2}[g^2]e_{j_1}\rangle\\
&=\lim_{\eps\to 0}\prod_{r=1,2}\langle e_{-j_1}, \ST^\eps_{p^r}[g^r]e_{j_1}\rangle=\prod_{r=1,2}\Big(\frac{(\fw\cdot j_1)^2}{|j_1|^2}\Big)^{\1_{a_r\ne0}}c_{g^r}(p^r)\\
&=\Big(\frac{\|(-\gensy)^{-\frac12}(-\gensy^\fw) e_{j_1}\|}{|j_1|/\sqrt{2}}\Big)^{\sum_{r}\1_{a_r\ne0}}c_{g^1}(p^1)c_{g^2}(p^2)\label{e:nome}
\end{equs}
where the first equality follows by the fact that, by Lemma \ref{l:RiemanConv} the scalar products are real, the second 
by the fact that, since for every $\eps>0$ fixed, 
at the right hand side of~\eqref{eq:casem1} there is the indicator function $\1_{A_{j_1, j_1'}}(g)=\1_{j_1'=j_1}$, 
the sum over $j_1'$ only consists of $1$ term, the third comes from~\eqref{e:Const1} and the last 
from the definition of $\gensy$ and $\gensy^\fw$. 
Therefore, upon setting
\begin{equ}[e:cpFinal]
c(p)\eqdef \sum_{g\in\cG^1[p]}c_g(p)
\end{equ}
~\eqref{e:cp} follows at once. 

For $\bj=j_{1:2}$ (so that $\kappa=2$), assume first that the graph associated to 
$g^1$ (or, equivalently, $g^2$) is in $\cG^2[p^1]\setminus\widetilde\cG^2[p^1]$, 
Then, we apply Cauchy-Schwarz and obtain 
\begin{equ}
\langle \overline{\ST^\eps_{p^1}[g^1]e_\bj}, \ST^\eps_{p^2}[g^2]e_\bj\rangle\leq \|\ST^\eps_{p^1}[g^1]e_\bj\|\|\ST^\eps_{p^2}[g^2]e_\bj\|\lesssim \|\ST^\eps_{p^1}[g^1]e_\bj\|
\end{equ}
which holds since, by Lemma~\ref{l:AnonsymProp}, $\ST^\eps_{p^2}[g^2]$ is a bounded operator 
and $e_{\bj}$ has bounded $L^2$ norm. Now, 
\begin{equ}
\|\ST^\eps_{p^1}[g^1]e_\bj\|\leq \frac12\|\ST^\eps_{p^1}[g^1](e_{j_1}\otimes e_{j_2})\|+ \frac12\|\ST^\eps_{p^1}[g^1](e_{j_2}\otimes e_{j_1})\|
\end{equ}
and Lemma~\ref{l:fastidi} implies that both summands at the right hand side converge to $0$. 

As a consequence, for $\bj=j_{1:2}$, 
we can replace $\cG^2[p^r]$ with $\widetilde\cG^2[p^r]$ in~\eqref{e:NewLim}, 
the latter being defined as in Lemma~\ref{l:Card}. 
Arguing as in~\eqref{e:nome}, we have
\begin{equs}
\langle \overline {\mathcal T^\eps_{p^1} e_\bj},\mathcal T^\eps_{p^2} e_\bj\rangle&=\sum_{\substack{g^r\in\widetilde\cG^2[p^r]\\r=1,2}}\langle \overline{\ST^\eps_{p^1}[g^1]e_\bj}, \ST^\eps_{p^2}[g^2]e_\bj\rangle_{L^2_{2}}\\
&=\sum_{\substack{g^r\in\widetilde\cG^{2}[p^r]\\r=1,2}}\sum_{\bj'}\langle \mathfrak{e}_{-\bj'}, \ST^\eps_{p^1}[g^1]e_{\bj}\rangle_{L^2_{2}}\langle \mathfrak{e}_{-\bj'}, \ST^\eps_{p^2}[g^2]e_{\bj}\rangle_{L^2_{2}}\\
&=\sum_{\bj'}\sum_{\substack{g^r\in\widetilde\cG^{2}[p^r]\\r=1,2}}\prod_{r=1,2}\langle \mathfrak{e}_{-\bj'}, \ST^\eps_{p^r}[g^r]e_{\bj}\rangle_{L^2_{2}}
\end{equs}
where $\mathfrak{e}_{\bj'}$ is defined according to~\eqref{eq:mfe} and 
we used that, thanks to the indicator function at the right hand side of~\eqref{eq:casem1}, 
the sum over $\bj'$ has only finitely many terms corresponding to $\bj'=j_{1:2}$ and $\bj'=j_{2:1}$. 
Since by Lemma~\ref{l:Card}, to each $g\in\cG^1[p]$ there exist exactly two elements 
of $\widetilde\cG^2[p]$, denoted by $\hat g$ and $\tilde g$, whose associated graph 
turns into that of $g$ by removing the unique connected component consisting of a path, 
the sum over the $g^r$'s satisfies
\begin{equs}
\sum_{\substack{g^r\in\widetilde\cG^{2}[p^r]\\r=1,2}}&\prod_{r=1,2}\langle \mathfrak{e}_{-\bj'}, \ST^\eps_{p^r}[g^r]e_{\bj}\rangle\\
&=\sum_{\substack{g^r\in\cG^1[p^r]\\r=1,2}}\prod_{r=1,2}\Big(\langle \mathfrak{e}_{-\bj'}, \ST^\eps_{p^r}[\hat g^r]e_{\bj}\rangle+\langle \mathfrak{e}_{-\bj'}, \ST^\eps_{p^r}[\tilde g^r]e_{\bj}\rangle\Big)
\end{equs}
so that 
\begin{equs}\label{uffa3}
\lim_{\eps\to0}&\langle \overline {\mathcal T^\eps_{p^1} e_\bj},\mathcal T^\eps_{p^2} e_\bj\rangle\\
&=\sum_{\substack{g^r\in\cG^1[p^r]\\r=1,2}}\sum_{\bj'}\prod_{r=1,2}\lim_{\eps\to 0}\Big(\langle \mathfrak{e}_{-\bj'}, \ST^\eps_{p^r}[\hat g^r]e_{\bj}\rangle+\langle \mathfrak{e}_{-\bj'}, \ST^\eps_{p^r}[\tilde g^r]e_{\bj}\rangle\Big)\,.
\end{equs}
To compute the inner limit, we expand $e_\bj$ and use Lemma~\ref{l:RiemanConv}, from which we deduce 
\begin{equs}
\,&\lim_{\eps\to 0}\langle \mathfrak{e}_{-\bj'}, \ST^\eps_{p^r}[\hat g^r]e_{\bj}\rangle+\lim_{\eps\to 0}\langle \mathfrak{e}_{-\bj'}, \ST^\eps_{p^r}[\tilde g^r]e_{\bj}\rangle\\
&=\tfrac12\lim_{\eps\to 0}\langle \mathfrak{e}_{-\bj'}, \ST^\eps_{p^r}[\hat g^r](e_{j_1}\otimes e_{j_2})\rangle + \tfrac12\lim_{\eps\to 0}\langle \mathfrak{e}_{-\bj'}, \ST^\eps_{p^r}[\hat g^r](e_{j_2}\otimes e_{j_1})\rangle\\
&\quad +\tfrac12\lim_{\eps\to 0}\langle \mathfrak{e}_{-\bj'}, \ST^\eps_{p^r}[\tilde g^r](e_{j_1}\otimes e_{j_2})\rangle + \tfrac12\lim_{\eps\to 0}\langle \mathfrak{e}_{-\bj'}, \ST^\eps_{p^r}[\tilde g^r](e_{j_2}\otimes e_{j_1})\rangle\\
&=\tfrac12\Big[ \1_{A_{(j_1,j_2);\bj'}(\hat g^r)}\Big(\frac{(\fw\cdot j_1)^2}{|\bj|^2}\Big)^{\1_{a_r\neq 0} }+ \1_{A_{(j_2,j_1);\bj'}(\hat g^r)}\Big(\frac{(\fw\cdot j_2)^2}{|\bj|^2}\Big)^{\1_{a_r\neq 0} } \Big]c_{\hat g^r}(p^r)\\
&\quad +\tfrac12\Big[ \1_{A_{(j_1,j_2);\bj'}(\tilde g^r)}\Big(\frac{(\fw\cdot j_1)^2}{|\bj|^2}\Big)^{\1_{a_r\neq 0} }+ \1_{A_{(j_2,j_1);\bj'}(\tilde g^r)}\Big(\frac{(\fw\cdot j_2)^2}{|\bj|^2}\Big)^{\1_{a_r\neq 0} } \Big]c_{\tilde g^r}(p^r)\\
&=\tfrac12\Big(\frac{(\fw\cdot j_1)^2}{|\bj|^2}+\frac{(\fw\cdot j_2)^2}{|\bj|^2}\Big)^{\1_{a_r\neq 0} }c_{g^r}(p^r)\Big(\1_{\bj'=j_{1:2}}+\1_{\bj'=j_{2:1}}\Big)\\
&=\frac{1}{2\|e_\bj\|}\Big(\frac{\|(-\gensy)^{-\frac12}(-\gensy^\fw) e_{\bj}\|}{|\bj|/\sqrt{2}}\Big)^{\1_{a_r\ne0}}c_{g^r}(p^r)\Big(\1_{\bj'=j_{1:2}}+\1_{\bj'=j_{2:1}}\Big)
\end{equs}
where, in the third step, we used~\eqref{e:ConstEqu} 
and that the only $\bj'$ which contribute to the sum in~\eqref{uffa3} 
are $\bj'=j_{1:2}$ and $j_{2:1}$, so that exactly one of the indicator functions multiplying 
$(\fw\cdot j_i)^2/|\bj|^2$ is $1$. 
Hence, getting back to~\eqref{uffa3}, we obtain 
\begin{equs}
\,&\lim_{\eps\to0}\langle \overline {\mathcal T^\eps_{p^1} e_\bj},\mathcal T^\eps_{p^2} e_\bj\rangle\\
&=\Big(\frac{\|(-\gensy)^{-\frac12}(-\gensy^\fw) e_{\bj}\|}{|\bj|/\sqrt{2}}\Big)^{\sum_r\1_{a_r\ne0}}\sum_{\substack{g^r\in\cG^1[p^r]\\r=1,2}}c_{g^r}(p^r)\sum_{\bj'}\Big(\frac{\1_{\bj'=j_{1:2}}+\1_{\bj'=j_{2:1}}}{2\|e_\bj\|}\Big)^2\\
&=\Big(\frac{\|(-\gensy)^{-\frac12}(-\gensy^\fw) e_{\bj}\|}{|\bj|/\sqrt{2}}\Big)^{\sum_r\1_{a_r\ne0}}c(p^1)c(p^2)\,,
\end{equs}
using that $\|e_\bj\|=1$ or $=1/\sqrt 2$ according to whether $\bj =j_1$ or $\bj=j_{1:2}$,
so that~\eqref{e:cpFinal} is proven and the proof of Proposition~\ref{p:limite} is complete. 
\end{proof}

\section{Convergence of Burgers to renormalised SHE}\label{sec:ProofofMain}

Thanks to the results in the previous section, we have all the ingredients we need in order to 
prove Theorem~\ref{thm:main}. 
To do so, we will first show that the sequence $\eta^\eps$ is tight and then verify that every limit 
point solves the stationary stochastic heat equation (SHE) 
\begin{equ}[e:SHE2]
\partial_t \eta =\tfrac12(\Delta + D_\SHE(\fw\cdot\nabla)^2)\eta + (-\Delta - D_\SHE(\fw\cdot\nabla)^2)^{1/2}\xi\,,\qquad \eta(0,\cdot)=\mu
\end{equ}
where $\mu$ is a  zero-average space white noise and 
$D_\SHE>0$ is the constant identified in Theorem~\ref{thm:MainConv2}.

\subsection{Tightness}

The proof of tightness follows a by now standard route (see
e.g.~\cite{GubinelliJara2012, GT, CK, CES}) which exploits Mitoma's
criterion~\cite{Mitoma}, Kolmogorov's continuity theorem and, more
importantly, the so-called It\^o trick, introduced
in~\cite{GubinelliJara2012} and since then exploited in a variety of
contexts, see~\cite{Gubinelli2016} for a pedagogical introduction.
For the reader's convenience, we recall below the statement of the
latter:
\begin{lemma}[It\^o trick]\label{l:Ito}
Let $d\geq 2$, $\eta^\eps$ be the stationary solution to \eqref{e:BurgersScaled} with $\lambda_\eps$ 
given as in~\eqref{eq:lambdaeps}. For any
$p \ge 2$, $T > 0$ and $F\in L^2(\Omega)$ with finite chaos expansion, i.e. $F\in\oplus_{j=1}^n\cH_j$ for some $n\in\N$, 
there exists a constant $C=C(p,n)>0$ such that
  \begin{equ}[e:Itotrick]
    \mathbf E\left[\sup_{t\in[0,T]}\left|\int_0^t F(\eta^\eps_s)
      \dd s
      \right|^p
      \right]^{1/p}\leq C T^{1/2}\|(-\gensy)^{-1/2}F\|\,.
    \end{equ}
    For $p=2$, $C$ can be taken independent of $n$.
\end{lemma}
We refer the reader to~\cite[Lemma 4.1]{CES}  for a proof. The $n$-dependence of $C$ for $p>2$ arises when estimating the $L^p$ norm in the right-hand side of \cite[Eq. (4.7)]{CES}
with the $L^2$ norm, using Gaussian hypercontractivity.

We now move to the proof of tightness. 

\begin{proposition}\label{p:Tight}
The sequence $\{\eta^\eps\}_\eps$ is tight in $C([0,T],\cS'(\T^d))$. 
\end{proposition}
\begin{proof}
By Mitoma's criterion~\cite{Mitoma}, the sequence $\{\eta^\eps\}_\eps$ is tight in $C([0,T],\cS'(\T^d))$ if and only if 
$\{\eta^\eps(\phi)\}_\eps$ is tight in $C([0,T],\R)$ for all $\phi\in\cS(\T^d)$. Therefore, we fix $\phi\in\cS(\T^d)$ and 
consider the process $\{\eta^\eps_t(\phi)\}_{t\in[0,T]}$. In view of~\eqref{e:SPDEweak}, it suffices to show that 
each of the terms at the right hand side is tight. This is clear for the last, as it is independent of $\eps$, 
while for the first and second, Lemma~\ref{l:Ito} implies that, for every $0\leq t\leq T$ and $p>2$, we have 
\begin{equ}[e:Lapla]
\Exp\left[\left(\int_0^t\eta^\eps_s(\Delta\phi)\dd s\right)^p\right]^{\frac1{p}}\lesssim t^{1/2}\|(-\gensy)^{-1/2}(\Delta\phi)\|=t^{1/2}\|\phi\|_{H^1(\T^d)}
\end{equ} 
and 
\begin{equs}[e:NonlinT]
\Exp\left[\left(\int_0^t\cN^\eps_\phi(\eta^\eps)\dd s\right)^p\right]^{\frac1{p}}\lesssim t^{1/2}\|(-\gensy)^{-1/2}\genap\phi\|\lesssim t^{1/2}\|\phi\|_{H^1(\T^d)}
\end{equs}
where in the last step we used that the kernel of $\cN^\eps_\phi(\eta^\eps)$ is $\genap\phi$. 
Since the process $\{\eta^\eps_t(\phi)\}_{t\in[0,T]}$ is Markov,~\eqref{e:Lapla} and~\eqref{e:NonlinT} together with 
Kolmogorov's criterion ensure that the sequence $\{\eta^\eps(\phi)\}_\eps$ is tight in $C([0,T],\R)$, 
so that the proof of the statement is concluded. 
\end{proof}

\subsection{Convergence}\label{sec:Conv}

In order to identify the limit equation, we will prove that any limit point satisfies 
the martingale problem associated to~\eqref{e:SHE2} which we now recall. 

\begin{definition}\label{def:MPSHE}
Let $T>0$, $\Omega=C([0,T], \cS'(\T^d))$ and $\CB$ the canonical Borel $\sigma$-algebra on $C([0,T], \cS'(\T^d))$. 
Let $\mu$ be a zero-average space white noise on $\T^d$.
We say that a probability measure $\BP$ on $(\Omega,\CG)$ {\it solves the martingale problem for 
$\geneff\eqdef  \gensy+\cD= \frac12(\Delta+D_\SHE (\fw\cdot \nabla)^2)$
with initial distribution $\mu$}, {if for all $\phi\in\cS(\T^d)$, 
the canonical process $\eta$ under $\BP$ is such that 
\begin{equs}[e:Mart]
\BM_t(F_j)&\eqdef F_{j}(\eta_t) - F_{j}(\mu) - \int_0^t \geneff F_j(\eta_s)\dd s\,,\qquad j=1,2
\end{equs}
is a local martingale, where $F_1(\eta)\eqdef \eta(\phi)$ and $F_2(\eta)\eqdef \eta(\phi)^2-\|\phi\|^2_{L^2(\mathbb T^d)}$. }
\end{definition}

In the next theorem, we state well-posedness for the previous martingale problem. 

\begin{theorem}\label{thm:WellPosedMP}
The martingale problem for $\geneff$ with initial distribution $\mu$ in Definition~\ref{def:MPSHE} has a unique solution 
and uniquely characterises the law of the solution to~\eqref{e:SHE2} on $C([0,T], \cS'(\T^d))$.
\end{theorem}
{\begin{proof}
Translating the results of~\cite[Appendix D]{MW} to our setting, we have that Theorem D.1 therein 
ensures that the martingale problem is well-posed provided that condition~\eqref{e:Mart} holds for $j=1$ and, in addition,
$(\BM_t(F_1))^2-2t\|(-\geneff)^{\frac12}\phi\|^2$ is a local martingale.
On the other hand, for $j=2$,   \eqref{e:Mart} gives
\begin{equ}
  \BM_t(F_2)=\eta_t(\phi)^2-\mu(\phi)^2-2 \int_0^t\eta_s(\phi)\eta_s(\geneff \phi)\dd s-2t \|(-\geneff)^{\frac12}\phi\|^2
\end{equ}
because, for $F_2(\eta)=\eta(\phi)^2-\|\phi\|^2_{L^2(\mathbb T^d)}=:\eta(\phi)^2:$, one has
\begin{equ}
\geneff F_2(\eta)=2:\eta(\phi) \eta(\geneff\phi):=2\eta(\phi)\eta(\geneff \phi)+2\|(-\geneff)^{\frac12}\phi\|^2\,.
\end{equ}
As a consequence,
\begin{equs}
  (\BM_t(F_1))^2&-2t\|(-\geneff)^{\frac12}\phi\|^2\\
  &=\BM_t(F_2)-2\mu(\phi)\BM_t(F_1)+\int_0^t \dd s\int_0^t \dd \bar s\;\eta_s(\geneff \phi) \eta_{\bar s}(\geneff \phi)\\
  &\qquad\qquad\qquad\qquad\qquad\qquad-2\int_0^t(\eta_t(\phi)-\eta_s(\phi))\eta_s(\geneff\phi)\dd s\\
  &=\BM_t(F_2)-2\mu(\phi)\BM_t(F_1)-2\int_0^t\eta_s(\geneff\phi)\left(\int_s^t\dd \BM_{\bar s}(F_1)\right)\dd s\\
  &=
  \BM_t(F_2)-2\mu(\phi)\BM_t(F_1)-2\int_0^t\left(\int_0^{\bar s}\eta_s(\geneff\phi)\dd s\right)\dd \BM_{\bar s}(F_1).
\end{equs}
Now, all terms at the right hand side are local martingales, which implies that so is the left hand side. 
Hence, the proof of the statement is complete. 
\end{proof} }

We are now ready to prove Theorem~\ref{thm:main}. 

\begin{proof}[of Theorem~\ref{thm:main}]
By Proposition~\ref{p:Tight}, we know that the sequence $\{\eta^\eps\}_\eps$ is tight in $C([0,T],\cS'(\T^d))$, 
hence it converges along subsequences. If we prove that any limit point is a solution of the martingale 
problem in Definition~\ref{def:MPSHE} then the statement follows by Theorem~\ref{thm:WellPosedMP}. 

Let $\eta\in C([0,T],\cS'(\T^d))$ be a limit point. To verify that $\eta$ satisfies the martingale problem, 
{ we need to show that for every given $\phi\in\cS(\T^d)$ the processes
$\BM(F_{i-1})$, $i=2,3$, defined according to~\eqref{e:Mart} with $F_1=\eta(\phi)$ and 
$F_2=\eta(\phi)^2-\|\phi\|^2_{L^2(\mathbb T^d)}$, are local martingales, 
for which it suffices to prove that for every $s\in[0,T]$ and 
$G\colon C([0,T],\cS'(\T^d))\to\R$ bounded continuous 
we have 
\begin{equ}[e:FINITA]
\Exp[\delta_{s,t}\BM_\cdot(F_{i-1}) G(\eta\restr_{[0,s]})]=0\,,
\end{equ} 
where we introduced a convenient notation for the time increment, i.e. $\delta_{s,t}f\eqdef f(t)-f(s)$. 
Now, by the definition 
of $\BM(F_{i-1})$, we deduce that 
\begin{equs}[e:Firstlim]
\Exp[(\BM_t&(F_{i-1})-\BM_s(F_{i-1})) G(\eta\restr_{[0,s]})]\\
&=\Exp\Big[\Big( F_{i-1}(\eta_t)- F_{i-1}(\eta_s)-\int_s^t \geneff  F_{i-1}(\eta_r)\dd r\Big) G(\eta\restr_{[0,s]})\Big]\\
&=\lim_{\eps\to 0}\Exp\Big[\Big( F_{i-1}(\eta^\eps_t)- F_{i-1}(\eta^\eps_s)-\int_s^t \geneff  F_{i-1}(\eta^\eps_r)\dd r\Big) G(\eta^\eps\restr_{[0,s]})\Big]
\end{equs}
where 
we used that $\eta^\eps$ converges in law to $\eta$ in
$C([0,T],\cS'(\T^d))$ together with~\cite[Eq. (5.11)]{CETWeak} to
approximate the first factor in the expectation with bounded
continuous functionals.  To analyse this latter term,} let us write
the equation for $\eta^\eps$ in such a way that we can identify the
elements which are relevant to the limit.

{
By Dynkin's formula and the weak formulation of~\eqref{e:BurgersScaled} in~\eqref{e:SPDEweak}, we have
\begin{equ}[e:Weak]
F_{i-1}(\eta^\eps_t)- F_{i-1}(\eta^\eps_s)-\int_s^t \gensy  F_{i-1}(\eta^\eps_r)\dd r-\int_s^t\gena F_{i-1}(\eta^\eps_r)\dd r=\delta_{s,t}M^\eps_\cdot( F_{i-1})
\end{equ}
where $M^\eps( F_{i-1})$ is the martingale whose quadratic variation is
\begin{equ}[e:M1]
\langle M^\eps_\cdot( F_{i-1})\rangle_t=\int_0^t\sum_k|k|^2|[D_k F_{i-1}](\eta^\eps_s)|^2 \dd s
\end{equ} 
and $D_k$ is the Malliavin derivative in~\eqref{e:Malliavin}. 

The term which is responsible for creating the new noise and the new Laplacian, is that containing $\genap F_{i-1}$. 
In order to describe it, notice that 
the kernel of $ F_1$ in $\fock$ is $ f_1=\phi$, while that of $F_2$ is  
$\phi\otimes\phi$. Then, we consider the random variable $\WN\in L^2(\P)$, 
for $n\in\N$, whose kernel is $\wN$, which is the solution of the truncated generator equation $\vN$ in~\eqref{e:nonlinGenEq} if $d\geq 3$ or $\tvN$
given by~\eqref{e:nonlinAnsatz} if $d=2$. 
By Dynkin's formula applied to $\WN$, we have 
\begin{equ}[e:DynkinCiao]
\WN(\eta^\eps_t)-\WN(\eta^\eps_s)-\int_s^t \gen\WN(\eta^\eps_r)\dd r=\delta_{s,t}M_\cdot^\eps(\WN)
\end{equ}
where $M^\eps(\WN)$ is the martingale whose quadratic variation is the same as~\eqref{e:M1} with 
$ F_{i-1}$ replaced by $\WN$. 
We now go back to~\eqref{e:Weak} which we rewrite as 
\begin{equ}[e:etaWeak]
 F_{i-1}(\eta^\eps_t)- F_{i-1}(\eta^\eps_s)-\int_s^t \geneff  F_{i-1}(\eta^\eps_r)\dd r=\delta_{s,t}(M^\eps_\cdot(F_{i-1}) +M_\cdot^\eps(\WN))+\delta_{s,t} R^{\eps,n}
\end{equ}
where $R^{\eps,n}\eqdef \sum_{j=1}^4 R_j^{\eps,n}$ and the $R_j^{\eps,n}$'s are defined as  
\begin{equs}
R_1^{\eps,n}(t)&\eqdef \WN(\mu)-\WN(\eta^\eps_t)\,,\\
R_2^{\eps,n}(t)&\eqdef \int_0^t \Big(\gen \WN+\genap F_{i-1}-\genam \WN_{i}\Big)(\eta^\eps_s)\dd s\,,\\
R_3^{\eps,n}(t)&\eqdef \int_0^t \Big(\genam \WN_{i}(\eta^\eps_s) -\cD  F_{i-1}(\eta^\eps_s)\Big)\dd s\,,\\
R_4^{\eps,n}(t)&\eqdef \int_0^t\genam F_{i-1}(\eta^\eps_s)\dd s\,.
\end{equs}
We now get back to~\eqref{e:Firstlim}, which, in view of~\eqref{e:DynkinCiao} equals
\begin{equ}
\lim_{n\to\infty}\lim_{\eps\to0}\Exp\Big[\Big( \delta_{s,t}(M^\eps_\cdot(F_{i-1}) +M_\cdot^\eps(\WN))+\delta_{s,t} R^{\eps,n}\Big) G(\eta^\eps\restr_{[0,s]})\Big]\,.
\end{equ}
Now, since $M^\eps(F_{i-1})$ and $M^\eps(\WN)$ are martingales, for every $n$ and every $\eps$ 
we have 
\begin{equ}[e:Secondlim]
\Exp\Big[\delta_{s,t}(M^\eps_\cdot(F_{i-1}) +M_\cdot^\eps(\WN)) G(\eta^\eps\restr_{[0,s]})\Big]=0\,.
\end{equ}
For the other term instead, we apply Cauchy-Schwarz and exploit the boundedness of $G$, 
from which we obtain
\begin{equ}
  \Exp\Big[\delta_{s,t} R^{\eps,n} G(\eta^\eps\restr_{[0,s]})\Big]\leq \|G\|_\infty\Exp\Big[| \delta_{s,t} R^{\eps,n}|^2\Big]^{\frac12}\leq  \|G\|_\infty\sum_{j=1}^4\Exp\Big[| \delta_{s,t} R_j^{\eps,n}|^2\Big]^{\frac12}\,,
\end{equ}
so that we are left to show that each of the summands at the right hand side converges to $0$. 
Let us begin with the first, which can be controlled as 
\begin{equs}
\Exp\Big[| \delta_{s,t} R_1^{\eps,n}|^2\Big]^{\frac12}&=\Exp\Big[| \WN(\eta^\eps_t)-\WN(\eta^\eps_s)|^2\Big]^{\frac12}\\
&\leq \Exp\Big[| \WN(\eta^\eps_t)|^2\Big]^{\frac12}+\Exp\Big[| \WN(\eta^\eps_s)|^2\Big]^{\frac12}=2\|\wN\|
\end{equs}
and the right hand side converges to $0$ as $\eps$ goes to $0$ by~\eqref{e:L2norm2}. 
For the others, we apply the It\^o trick, Lemma~\ref{l:Ito}, which gives
\begin{equs}
\Exp\Big[(\delta_{s,t}R_2^{\eps,n})^2\Big]^{\frac12}&\leq Ct^{\frac12}  \|(-\gensy)^{-\frac12}(-\gen \wN-\genap f_{i-1}+\genam \wN_{i})\|\,,\label{e:rem2}\\
\Exp\Big[(\delta_{s,t}R_3^{\eps,n})^2\Big]^{\frac12}&\leq C{t}^{\frac12}  \|(-\gensy)^{-1/2}[\genam\wN_{i}-\cD f_{i-1}\|\,,\label{e:rem3}\\
\Exp\Big[(\delta_{s,t}R_4^{\eps,n})^2\Big]^{\frac12}&\leq C{t}^{\frac12} \|(-\gensy)^{-1/2}\genam f_{i-1}\|\,,\label{e:rem4}
\end{equs}
with $C$ independent of $n$. Now,~\eqref{e:rem2} and~\eqref{e:rem3} converge to $0$ in the double 
limit as $\eps\to 0$ first and $n\to\infty$ then, by~\eqref{e:ApproxH1} and~\eqref{e:Diffusivity2}. 
Concerning~\eqref{e:rem4}, for $i=1$, $\genam f_1=0$, while for 
$i=2$ we leverage the smoothness of $ f_2$. Indeed,  
\begin{equs}
\|(-\gensy)^{-1/2}\genam f_2\|^2&\leq\sum_{k} \frac{(\fw\cdot k)^2}{|k|^2}\lambda_\eps^2\Big(\sum_{\ell+m=k}\indN{\ell,m}\hat f_2(\ell,m)\Big)^2\\
&\lesssim  \Big(\lambda_\eps^2\sum_{|\ell|<\eps^{-1}}\frac{1}{|\ell|^{2\alpha}}\Big) \sum_{\ell,m}(|\ell|^2+|m|^2)^{\alpha}|\hat f_2(\ell,m)|^2\\
&=\Big(\lambda_\eps^2\sum_{\ell}\frac{1}{|\ell|^{2\alpha}}\Big)\|(-\gensy)^{\alpha/2} f_2\|^2
\end{equs}
for $\alpha>1$. The last norm is finite and it is not hard to see that the quantity in parenthesis is converging to $0$, 
thus implying that the $\eps\to0$ limit of~\eqref{e:rem4} is $0$. 

Collecting the results above together with~\eqref{e:Firstlim} and~\eqref{e:Secondlim},~\eqref{e:FINITA} 
follows at once so that the proof of the theorem is complete. }
\end{proof}

\begin{appendix}


\section{Some technical results}\label{a:Approx}

In this appendix, we show the technical part of the Replacement Lemma~\ref{l:Replacement}. 

\begin{proposition}\label{p:Approx}
 For $\eps>0$, $n\in\N$ and $k_{1:n}\in\Z^{2n}_0$, define  $P^\eps$ as in~\eqref{e:P1}, i.e. 
\begin{equ}[e:P]
P^\eps(k_{1:n})\eqdef\frac{\lambda_\eps^2}{\pi^2} \sum_{\ell + m=k_1} \frac{\indN{\ell,m}}{\tilde\Gamma + \tilde\Gamma^wG(\Ll(\tilde\Gamma))}
\end{equ}
for
\begin{equ}
\tilde\Gamma\eqdef\tfrac12(|\ell|^2+|m|^2+|k_{2:n}|^2)\,,\qquad\tilde\Gamma^w\eqdef \tfrac12((\fw\cdot\ell)^2+(\fw\cdot m)^2+(\fw\cdot k)^2_{2:n})
\end{equ}
and $G$ is as in \eqref{eq:defG}.
Then, there exists a constant $C>0$ independent of $\eps,n$ such that 
\begin{equ}[e:Approx]
\sup_{n\in \N,k_{1:n}\in\Z^{2n}_0}\Big|P^\eps(k_{1:n})-G\big(\Ll(\tfrac12|k_{1:n}|^2)\big)   \Big| \leq C \lambda_\eps^2\,.
\end{equ}
\end{proposition}

\begin{proof}
  Note first of all that we can assume that $|k_1|\le 1/(2\eps)$. Indeed,
  in the complementary case both $P^\eps(k_{1:n})$ and 
  $G(\Ll(\frac12|k_{1:n}|^2))$ would be $o(\lambda_\eps^2)$ (uniformly in
 $n$ and  $k_{1:n}$). For $G(\Ll(\frac12|k_{1:n}|^2))$, this easily follows from the definition of $G$ and $\Ll$. 
 As for $P^\eps(k_{1:n})$, if  $|k_1|\ge 1/(2\eps)$, then either $|\ell|$ or $|m|$ are larger than $1/(4\eps)$. 
 Assume, for instance, that the former is the case. Then,
  \begin{equ}
    \label{eq:isthecase}
P^\eps(k_{1:n})\lesssim \lambda_\eps^2\sum_{\frac1{4\eps}\le |\ell|\le\frac1\eps}\frac1{|\ell|^2}\lesssim \lambda_\eps^2.
  \end{equ}
  For  $|k_1|\le 1/(2\eps)$,  we  rewrite \eqref{e:P} as
  \begin{equ}
    \label{eq:PP}
    P^\eps(k_{1:n})\eqdef\frac{\lambda_\eps^2}{\pi^2} \sum_{\substack{\ell + m=k_1\\ |\ell|,|m|\le \eps^{-1}}} \frac{1}{\tilde\Gamma + \tilde\Gamma^wG(\Ll(\tilde\Gamma))}
  \end{equ}
Note  that, when $\ell+m=k_1$,
\begin{equs}
  \tilde\Gamma&=\Gamma-\ell\cdot k_1,
  \qquad\qquad\qquad \Gamma\eqdef
  |\ell|^2+\frac12|k_{1:n}|^2,\\ 
  \tilde \Gamma^w&=\Gamma^w-(\fw\cdot k_1)(\fw\cdot \ell), \qquad\Gamma^w\eqdef
  (\fw\cdot \ell)^2+\frac12(\fw\cdot k)^2_{1:n}.
\end{equs}
For future reference, we point out that $\Gamma$ and $\tilde\Gamma$ are comparable in that, 
by the triangular inequality, we have
\begin{equ}[e:lm]
\Gamma\lesssim \tilde\Gamma\lesssim \Gamma\,
\end{equ}
uniformly in $\ell,m=k_1-\ell$ and  $k_{1:n}\in\Z^{2n}_0$.
\medskip

The proof of Proposition~\ref{p:Approx} follows from a sequence of technical steps in which the summand
in the definition of $P^\eps$ is progressively simplified and 
the Riemann sum is replaced by an integral. More precisely, we will first replace $P_0^\eps\eqdef P^\eps$ 
with $P^\eps_1$ and then $P^\eps_i$ with $P^\eps_{i+1}$, $i=1,\dots, 4$, in such a way that 
for all $i=0,\dots,5$, 
\begin{equs}
  \label{eq:pallosissimi}
  \sup_{n\in \N,k_{1:n}\in\Z^{2n}_0}\Big|P^\eps_i(k_{1:n})-P_{i+1}^\eps(k_{1:n})  \Big| \lesssim \lambda_\eps^2\,. 
\end{equs}
At last, we will compute $P^\eps_5$ and show that it equals $G(\Ll)$. 
\medskip

\noindent {\bf Step 1} First of all, we define $P_1^\eps(k_{1:n})$ as $P^{\eps}(k_{1:n})$ in \eqref{eq:PP}, 
except that $\tilde\Gamma$ and $\tilde\Gamma^w$ are replaced by $\Gamma$ and $\Gamma^w$, respectively. 
The absolute value of the difference is upper bounded
by a constant times
\begin{equ}
   \lambda_\eps^2\sum_{\substack{\ell+m=k_1\\|\ell|,|m|\le 1/\eps}}\frac{A+B}
{(|\ell|^2+|k_1|^2)^2}
\end{equ}
where in the denominator we used $\Gamma\wedge \tilde \Gamma\gtrsim |\ell|^2+|k_1|^2$ and 
$\Gamma^w\wedge\tilde \Gamma^w\geq 0$, while in the numerator $A$ and $B$ are defined as
\begin{equs}
A&\eqdef  |\ell\cdot k_1+(\fw\cdot \ell)(\fw\cdot k_1)G(\Ll(\tilde\Gamma))|,\\
B&\eqdef |G(\Ll(\tilde\Gamma))-G(\Ll(\Gamma)) |\Big(|\ell|^2+\frac12|k_{1:n}|^2\Big)\,.
\end{equs}
Now, $A\lesssim |\ell||k_1|$ so that
\begin{equs}
\lambda_\eps^2\sum_{\substack{\ell+m=k_1\\|\ell|,|m|\le 1/\eps}}\frac A{(|\ell|^2+|k_1|^2)^2} \lesssim \lambda_\eps^2,
\end{equs} 
as can be immediately seen by splitting the sum according to  whether $|\ell|\le |k_1|$ or the converse.
The term containing $B$ is dealt with similarly. Indeed, since $|\Gamma-\tilde\Gamma|\le |\ell|| k_1|$,  one has 
\begin{equs}
  |G(\Ll(\tilde\Gamma))-G(\Ll(\Gamma)) |\lesssim |\ell||k_1|\sup_{y\in [\Gamma,\tilde \Gamma]}\left|\frac1{|\log\eps|}\frac{G'(\Ll(y))
    }{y(\eps y+1)}
  \right|\lesssim \frac{ |\ell||k_1|}{|\ell|^2+\frac12|k_{1:n}|^2}
\end{equs}
where we further used that the derivative of $G$ is bounded and \eqref{e:lm} holds. 
Hence, $B\lesssim |\ell||k_1|$ and one concludes as above. Inequality \eqref{eq:pallosissimi} for $i=0$ is then proven.
\medskip

\noindent {\bf Step 2}
Secondly, we define $P_2^\eps(k_{1:n})$ as $P_1^{\eps}(k_{1:n})$, except that $\Gamma^w$ is replaced by 
$(\fw\cdot \ell)^2$, that is,
\begin{equ}
  P_2^\eps(k_{1:n})\eqdef \frac{\lambda_\eps^2}{\pi^2}\sum_{\substack{\ell+m=k_1\\|\ell|,|m|\le \eps^{-1}}}\frac{1}{\Gamma+(\fw\cdot \ell)^2G(\Ll(\Gamma))}\,.
\end{equ}
The difference between $P^\eps_1$ and $P^\eps_2$ is bounded by a constant times
\begin{equs}
 \lambda_\eps^2\sum_{\substack{\ell+m=k_1\\|\ell|,|m|\le \eps^{-1}}}\frac{(\fw\cdot k)_{1:n}^2}{(|\ell|^2+|k_{1:n}|^2)^2}
\lesssim \lambda_\eps^2\sum_{\substack{\ell+m=k_1\\|\ell|,|m|\le \eps^{-1}}}\frac{|k_{1:n}|^2}{(|\ell|^2+|k_{1:n}|^2)^2}\lesssim \lambda_\eps^2\,,
\end{equs}
where we used that $G(\Ll(\Gamma))\lesssim 1$. 
\medskip

\noindent {\bf Step 3}
As a third step, we  replace $P_2^\eps(k_{1:n})$ with
\begin{equ}
  P_3^\eps(k_{1:n})\eqdef \frac{\lambda_\eps^2}{\pi^2}\sum_{|\ell|\le \eps^{-1}}\frac{1}{\Gamma+(\fw\cdot \ell)^2G(\Ll(\Gamma))}.
\end{equ}
With respect to $P^\eps_2$, we have only removed the indicator function of $|m|=|k_1-\ell|\leq \eps^{-1}$. 
Hence, the difference between $P_2^\eps$ and $P_3^\eps$ only involves the sum over those values of 
$\ell$ such that $|k_1-\ell|\ge 1/\eps$ while $|\ell|\le 1/\eps$. 
But, since $|k_1|\le 1/(2\eps)$, this corresponds to $1/(2\eps)\le |\ell|\le 1/\eps$ and
the resulting sum can be bounded as in \eqref{eq:isthecase}.
\medskip

\noindent {\bf Step 4}  Next, we turn the sum into an integral. Setting
$\alpha_\eps\eqdef \frac12|\eps k_{1:n}|^2$, we define $P_4^\eps$ as
\begin{equ}[e:P4]
  P_4^\eps(k_{1:n})\eqdef \frac{\lambda_\eps^2}{\pi^2}\int_{|x|\le 1}\frac{\dd x}{|x|^2+\alpha_\eps+(x\cdot \fw)^2G(\Ll(\eps^{-2}(|x|^2+\alpha_\eps)))}\,.
\end{equ}
To obtain~\eqref{eq:pallosissimi} for $i=3$, we argue as in~\cite[Lemma C.7]{CET}. 
Denoting by $I$ the integrand in~\eqref{e:P4} and by $Q_\ell^\eps$ the square of side-length $\eps$ centred at 
$\ell\in\Z^2_0$,
the difference between $P_3^\eps$ and $P_4^\eps$ is bounded above by
\begin{equs}
\lambda_\eps^2\sum_{1\leq |\ell|\leq \eps^{-1}}\int_{Q_\ell^\eps} |I(\eps\ell)-I(x)|\dd x&\leq \lambda_\eps^2\eps\sum_{1\leq |\ell|\leq \eps^{-1}}\int_{Q_\ell^\eps} \sup_{y\in Q^\eps_\ell} |\nabla I(y)|\dd x\\
&\leq\lambda_\eps^2\eps^3\sum_{1\leq |\ell|\leq \eps^{-1}}\sup_{y\in Q^\eps_\ell} |\nabla I(y)|\,.
\end{equs} 
Now, a simple computation yields  
\begin{equ}
\sup_{y\in Q^\eps_\ell} |\nabla I(y)|\lesssim \lambda_\eps^2\frac{1}{(|\eps\ell|^2 +\alpha_\eps)^{3/2}}
\end{equ}
so that~\eqref{eq:pallosissimi} follows at once. 
\medskip 

Before moving to the next step, note that, by passing to polar coordinates, 
i.e. setting $x=r v_\theta$ for $v_\theta=(\cos\theta,\sin\theta)$, we have 
\begin{equs}
P_4^\eps(k_{1:n})&= \frac{\lambda_\eps^2}{\pi^2}\int_0^{2\pi}\dd \theta\int_{0}^1\frac{r\dd r}{r^2+\alpha_\eps+r^2(\fw\cdot v_\theta)^2G(\Ll(\eps^{-2}(r^2+\alpha_\eps)))}\\
&=\frac{\lambda_\eps^2}{2\pi^2}\int_0^{2\pi}\dd \theta\int_{0}^1\frac{\dd r}{r+\alpha_\eps+r(\fw\cdot v_\theta)^2G(\Ll(\eps^{-2}(r+\alpha_\eps)))}\\
&=\frac{\lambda_\eps^2}{2\pi^2}\int_0^{2\pi}\dd \theta\int_{0}^1\frac{\dd r}{r+\alpha_\eps+r|\fw|^2 \cos^2(\theta-\theta_w)G(\Ll(\eps^{-2}(r+\alpha_\eps)))}\\
&=\frac{\lambda_\eps^2}{\pi^2}\int_{0}^1\dd r\int_0^{\pi}\frac{\dd \theta}{r+\alpha_\eps+r|\fw|^2 \cos^2(\theta)G(\Ll(\eps^{-2}(r+\alpha_\eps)))}
\end{equs}
where we set $w=|w|v_{\theta_w}$ and used periodicity of the integrand in $\theta$. 
\medskip

\noindent{\bf Step 5} In this step, we insert two elements at the denominator of the integrand of $P_4^\eps$, 
that will allow us to make a simple change of variables. 
More precisely, we replace $P_4^\eps$ with $P_5^\eps$, the latter being defined as  
\begin{equ}
P_5^\eps(k_{1:n})\eqdef \frac{\lambda_\eps^2}{\pi^2}\int_0^{\pi}\int_{0}^1\frac{\dd \theta\dd r}{(r+\alpha_\eps)(r+\alpha_\eps+1)[1+|\fw|^2 \cos^2(\theta)G(\Ll(\eps^{-2}(r+\alpha_\eps)))]}\,.
\end{equ}
Now,~\eqref{eq:pallosissimi} for $i=4$ 
can be easily seen to hold by arguing as in the proof of~\cite[Eq. (A.9)]{CETWeak}.
\medskip

At this point, it remains to analyse the last integral. With the change of variables
\begin{equs}
y= \Ll(\eps^{-2}(r+\alpha_\eps))&=\lambda_\eps^2\log\Big(1+\frac1{r+\alpha_\eps}\Big), \\
\frac{\dd r}{(r+\alpha_\eps)(1+r+\alpha_\eps)}&=- \lambda_\eps^{-2} \dd y,
\end{equs}
one finds
\begin{equs}
P_5^\eps(k_{1:n})=\frac{1}{\pi^2}\int_0^{\pi}\dd \theta \int_{\Ll(\eps^{-2}+\frac12|k_{1:n}|^2)}^{\Ll(\frac12|k_{1:n}|^2)}\frac{\dd y}{1+|\fw|^2 \cos^2(\theta)G(y)}
\end{equs}
It is then immediate to verify that, again at the price of an error of order $\lambda_\eps^2$, we can replace the lower 
integration index by $0$, and obtain
\begin{equs}[e:G?]
\frac{1}{\pi^2}\int_{0}^{\Ll(\frac12|k_{1:n}|^2)}&\dd y\int_0^{\pi} \frac{\dd \theta}{1+|\fw|^2 \cos^2(\theta)G(y)}\\
&=\frac{1}{\pi}\int_{0}^{\Ll(\frac12|k_{1:n}|^2)}\frac{\dd y}{\sqrt{1+|\fw|^2G(y)}}=G(\Ll(\frac12|k_{1:n}|^2))
\end{equs}
where the last step is a consequence of our choice of $G$ (actually, {\it we have chosen $G$ in such a way that 
it holds}). Hence,~\eqref{e:Approx} follows at once and 
the proof of the proposition is concluded. 
\end{proof}


\end{appendix}

%
%
%
%
%

\section*{Acknowledgments} 
The authors would like to thank L. Haunschmid-Sibitz for many
discussions in an early stage of this project, C.~Landim for
explaining some of the ideas in~\cite{Komorowski2012}, and R.~Bauerschmidt for useful discussions about RG.  A
special thank goes to L.~Gr\"afner, who presented us the strategy
in~\cite{GrPer} which allowed us to perform the steps
in~\eqref{e:LimScalarProd} and lead to the proof of the main theorem in the
super-critical case, and to H. Giles, who pointed out a mistake in an earlier version of the manuscript. 
G.~C. gratefully acknowledges financial support
via the UKRI FL fellowship ``Large-scale universal behaviour of Random
Interfaces and Stochastic Operators'' MR/W008246/1.  F.~T.  was
supported by the Austrian Science Fund (FWF): P35428-N.

\bibliography{bibtex}
\bibliographystyle{Martin}

\end{document}